\documentclass[12pt]{amsart}
\usepackage{TDefn}
\title{Categorifications of rational Hilbert series and characters of $\FSop$ modules}

\author{Philip Tosteson}
\address{Department of Mathematics, University of Chicago, Chicago, IL}
\email{\href{mailto:ptoste@umich.edu}{ptoste@math.uchicago.edu}}
\urladdr{\url{https://math.uchicago.edu/~ptoste/}}

\thanks{This work was supported by NSF grant  DMS-1903040.}

\begin{document}
\maketitle
\begin{abstract}We introduce a method for associating a chain complex to a module over a combinatorial category, such that if the complex is exact then the module has a rational Hilbert series.  We prove homology--vanishing theorems for these complexes for several combinatorial categories including: the category of finite sets and injections,  the opposite of the category of finite sets and surjections, and the category of finite dimensional vector spaces over a finite field and injections.    

Our main applications are to modules over the opposite of the category of finite sets and surjections, known as $\FSop$ modules.  We obtain many constraints on the sequence of symmetric group representations underlying a finitely generated $\FSop$ module.  In particular, we describe its character in terms of functions that we call character exponentials.   Our results have new consequences for the character of the homology of the moduli space of  stable marked curves, and for the equivariant Kazhdan--Lusztig polynomial of the braid matroid.
\end{abstract}

\tableofcontents
\section{Introduction}

Prompted by the work of Snowden \cite{snowden2013syzygies} and Church--Ellenberg--Farb \cite{CEF},  there have been many recent applications of the representation theory of categories to sequences of group representations arising in topology, algebra, and combinatorics.  

Let $(V_n)_{n \in \bbN}$ be a sequence of representations of groups $G_n$ over a field $k$.  
In these applications,  $(V_n)_{n \in \bbN}$ can  be lifted to a representation of a category $\cC$ whose objects (or isomorphism classes) 
 are indexed by natural numbers, such that the automorphism group of  the $n^{\rm th}$ object of $\cC$ is $G_n$.   A   \emph{$\cC$--representation}, or  \emph{$\cC$--module}, is a functor $\cC \to \Mod k$.
One attempts to use the representation theory of $\cC$ to show that the whole sequence of representations is determined by a finite amount of data,  and to discover universal patterns which the sequence must satisfy.  Often, these patterns are expressed in terms of generating functions.  

One coarse invariant of a sequence of representations is its sequence of dimensions:  $(\dim V_n)_{n \in \bbN}.$  The data can be recorded as a generating function,   the \emph{Hilbert Series}  of $ (V_n)_{n \in \bbN}$  $$h_V(t):= \sum_{n}  \dim_k V_n ~ t^n \in \bbZ[[t]].$$
Often, theorems from the representation theory of $\cC$ imply that $h_V(t)$ is rational with denominator of a specific form.   We list three important instances.

\begin{itemize}
\item   Let $\FI$ be the category of finite sets and injections.  If $V$ is a finitely generated $\FI$ module  then  $h_V(t)$ is a rational function,  with denominator $(1-t)^d,$  for some $d \in \bbN$.  

Equivalently,  the sequence $(\dim V_n)_{n \in \bbN}$ eventually agrees with a polynomial in $n$, proved in characteristic $0$ in \cite{CEF,sam2016gl}  and for all fields by  Church--Ellenberg--Farb--Nagpal \cite{CEFN}.
\\

\item  Let $ \FSop$  be the opposite of the category of finite sets and surjections.    If $V$ is a finitely generated $\FSop$ module, then $h_V(t)$ is a rational function with denominator a power of $\prod_{j= 0}^{d-1} (1-jt),$ for some $d \in \bbN$.  This was proved by Sam--Snowden \cite{sam2017grobner}. \\


\item Let $\VI_q$ be the category of finite dimensional vector spaces over $\bbF_q$ and linear injections between them.  If  $V$ is a finitely generated $\VI_q$ module and  ${\rm char ~} k \neq {\rm char ~}  \bbF_q$,  then $h_V(t)$ is rational with denominator  $\prod_{j = 0}^{d-1} (1- q^jt)$  for some $d \in \bbN$. 

 Equivalently, $(\dim V_n)_{n \in \bbN}$ eventually agrees with a polynomial in $q^n$,  proved by Nagpal  \cite{nagpal2019vi}.      

\end{itemize}
In this paper  we \emph{categorify}  these three results, in a uniform way.  

  Let  $\cC = \FI,~ \FSop,$ or $\VI_q$,  and  let $w_d(t) = (1-t)^d$,  $\prod_{j = 0}^{d-1} (1-jt)$, or $\prod_{j  = 0}^{d-1} (1 - q^j t)$,  respectively.  Let $V$ be a $\cC$ module.  Using methods from poset topology, we construct  a chain complex of $\cC$ modules $\rK_d(V)$  for each $d \in \bbN$, such that exactness of $\rK_d(V)$ categorifies the equation $w_d(t) h_V(t) = p(t)$  for $p$ a polynomial.   In each case,  we prove a theorem  showing if $V$ is finitely generated,  either $\rK_d$  or a power of $\rK_d$  applied to $V$ is exact (modulo torsion).

 In particular, our categorification explains why the denominators $(1-t)^d,~\prod_{j = 0}^{d-1} (1-jt),$ and $\prod_{j  = 0}^{d-1} (1 - q^j t)$ appear.  They are the \emph{Whitney polynomials}  of the poset of subsets of a $d$-element set, of set partitions of a $d$-element set,  and of subspaces of  a $d$-dimensional vector space, respectively  (see Definition \ref{def:whitney}).

\subsection{ $\FSop$ modules}  Our main applications are in the case  where $\cC$ is the opposite of the category of finite sets and surjections,  $\FSop$.   

 In practice,  $\FSop$ modules have been used to study sequences of symmetric group representations $(V_n)_{ n \in\bbN}$ such that $\dim( V_n)$  grows at an exponential rate---too quickly for representation stability in the sense of Church--Farb \cite{church2013representation} to hold.    Proudfoot--Young used $\FSop$ to study the Kazhdan--Lusztig polynomial of the braid matroid, by constructing an action of $\FSop$ on the  intersection homology of the reciprocal plane of the braid arrangement \cite{proudfoot2017configuration}.  In \cite{tosteson2018stability},  we showed that the homology of the moduli space of stable curves is a finitely generated $\FSop$ module. In forthcoming work, we extend this result to the Kontsevich space of stable maps.  Proudfoot--Ramos constructed  finitely generated $\FSop$ modules in the cohomology of the resonance arrangement \cite{ramos2020stability}.  

In other directions, $\FSop$ was used by Sam--Snowden \cite{sam2017grobner} to prove the Lannes--Schwartz Artinian conjecture:  for a finite ring $R$ any finitely generated  $\VI_R$ module (Definition \ref{def:VIR}) restricts to a finitely generated $\FSop$ module.  The category of $\FSop$ modules is the category of right modules over the commutative operad,  and thus is closely tied to homology theories for commutative algebras, and $\FSop$  may be used to compute the homology groups of pointed mapping spaces \cite{pirashvili2000hodge}.  In forthcoming work with Sam and Snowden, we relate $\FSop$ modules to representations of the Witt Lie algebra and the monoid of endomorphisms of $\bbA^n$.

Despite these applications, the representation theory of $\FSop$ is not well understood.  For finitely generated $\FI$ modules and $\VI_q$ modules in non-describing characteristic,  Nagpal has proved structure theorems \cite{nagpal2015fi, nagpal2019vi}, known as shift theorems, which we use to establish our results. But $\FSop$ modules do not satisfy a shift theorem, and are qualitatively different:  the category of finitely generated $\FSop$ modules has infinite Krull--Gabriel dimension, as opposed to dimension $1$.   
  Little has been known about the sequence of symmetric group representations underlying a finitely generated $\FSop$ module.  

\subsection{Categorification for $\FSop$ modules}
For an $\FSop$ module $M$, the complex $\rK_d(M)$ we construct takes the form
$$\Sigma^{d} M \leftarrow \Sigma^{d-1} M^{\oplus {d \choose 2}} \leftarrow \dots \leftarrow \Sigma^1 M^{\oplus (d-1)!}.$$  Here  $\Sigma^n M$ denotes the $\FSop$ module  $x \mapsto M_{x \sqcup [n]}$,  where $[n] := \{1, \dots, n\}$ is the distinguished finite set with $n$ elements.  In general, the degree $i$ term of $\rK_d(M)$ is $\Sigma^{d-i} M^{\oplus s(d, d-i)}$,  where $s(d,d-i)$ is the  unsigned \emph{Stirling number of the first kind}:  by definition $(-1)^i s(d, d- i)$ is the coefficient of $t^i$ in the product $\prod_{j = 0}^{d-1} (1-jt)$.     We may extend the definition of $\rK_d$ to chain complexes of $\FSop$ modules, by  applying $\rK_d$ degreewise and forming the total complex.

To state our categorification for $\FSop$ modules, we recall the following definition.    

\begin{defn}  For $d \in \bbN$, we say that $M$  is \emph{finitely generated in degree $\leq d$}  if $M_x$ is a finitely generated $k$ module for all $x \in \FSop$ with $|x| \leq d$, and for any $y \in \FSop$,   $M_y$ is spanned by elements of the form $M_{f}(v)$, where $ v \in M_x, f: y \onto x$ and $|x| \leq d$.
\end{defn}

\begin{thm}\label{FSop}
Let $M$ be an $\FSop$ module over a field, which is a subquotient of an $\FSop$ module that is finitely generated in degree $\leq d$.   Then there exists $s \in \bbN,$ such that for all $r \geq s$  and all $(\ell_t) \in \bbN^r$  satisfying $\ell_t \geq d$  for all $t$ and $\ell_r \geq d+1$,  the complex  $$\rK_{\ell_1} \circ \rK_{\ell_2} \circ \dots \circ \rK_{\ell_r}(M)$$ is exact.
\end{thm}
In particular there exists an $s \geq 1$ such that $ \rK_{d}^{\circ s-1}(\rK_{d+1}(M))$  is exact.  This categorifies the theorem of Sam--Snowden that $h_M(t)$ is rational with denominator a power of   $\rW_{d+1}(t):=\prod_{j = 0}^{d} (1 - jt)$, by categorifying the equation  $ \rW_{d}(t)^{s-1} \rW_{d+1}(t) h_M(t) = p(t)$  for $p$ a polynomial.  We prove Theorem \ref{FSop} by using Sam--Snowden Gr\"obner theory to reduce to a combinatorial statement about regular languages associated to $M$.  The value of $s$ is related to the minimal number of states in  the DFAs accepting these languages.   

We list two cases in which Theorem \ref{FSop} applies. 
\begin{ex}
		For $i,g \in \bbN$,  the sequence of $\rS_n$ representation  $n \mapsto H_i(\bMgn, \bbQ)$,  where $\bMgn$ is the moduli space of marked stable curves, can be extended to an $\FSop$ module. In \cite{tosteson2018stability} we showed that this $\FSop$ module is a subquotient of one that is finitely  generated in degree $ \leq 8 g^2 i^2 + 29 g^2 i + 16 g i^2 + 21 g^2 + 10 g i$.  
\end{ex}

\begin{ex}
		The $i^{\rm th}$ coefficient of the equivariant Kazhdan--Lusztig polynomial of the $n^{\rm th}$ braid arrangement is the character of an $\rS_n$ representation. For fixed $i$,  Proudfoot--Young extended this  sequence of $\rS_n$  representations to an $\FSop$ module and proved it is a subquotient of one that is finitely generated in degree $\leq 2i$ \cite{proudfoot2017configuration}.  
\end{ex}

\subsection{Characters of $\FSop$ modules}

For each $n, d \in \bbN$ the group $\rS_d \times \rS_n$ acts on $\rK_d(M)_n$.  The action of $\rS_d$ is closely related to its action on the Whitney homology of the partition lattice.  Via this action, Theorem \ref{FSop} decategorifies to a system of differential equations for the Frobenius character of $M$. By solving  these equations, we obtain new results about the character of a finitely generated $\FSop$ module.   
To state them,  we take $k = \bbQ$.  

 Associated to every $\FSop$ module $M$ is its \emph{character} $\sqcup_n \rS_n \to \bbQ$,  given by $\sigma \in \rS_n \mapsto  \Tr(\sigma, M_{[n]})$.    Our first result describes the character of an $\FSop$ module in terms of character polynomials and new functions that we call character exponentials.

\begin{defn}
  Let $d \in \bbN$, $d \geq 1$.  Then $\bbX_d: \sqcup_n \rS_n \to \bbQ$ is the class function $$\bbX_d(\sigma)  := \#\{ \text{$d$-cycles of }\sigma\}.$$
A \emph{character polynomial}  is any class function  that is a polynomial in the functions $\{\bbX_d\}_{d \in \bbN}.$
\end{defn}

\begin{defn}Let $A = 1^{a_1} 2^{a_2} \dots $ be an integer partition. We define the \emph{character exponential}  of $A$  to be class function $$A^\bbX := \prod_{n \geq 1}  ( \sum_{d | n}  d a_d )^{\bbX_n}, $$
where sum is over all natural numbers $d \geq 1$  which divide $n$.   The product is well defined because any $\sigma \in \rS_m$ has finitely many cycles,  so $A^\bbX(\sigma)$ is a natural number.  
\end{defn}

In this definition, it is important that $0^0 = 1$ and $0^n = 0$ for all $n 
> 0$. We give examples to clarify:

\begin{ex}
		If  $a_i = 0$ for $i > 1$,  then  $A^\bbX  = (a_1)^{\sum_{n \geq 1} \bbX_n}$.  
\end{ex}

\begin{ex}
 If $a_1 = 0$ and $a_i = 0$ for $i > 2$,  then  $$A^\bbX(\sigma)  = \begin{cases} (2a_2)^{\sum_{n \geq 1} \bbX_{2n}(\sigma)}  &  \text{ if all cycles of $\sigma$ have even length}  \\ 0 & \text{ otherwise } \end{cases} $$
\end{ex}

In fact, to state our results we need a more general function, which is similar to (but not equal to) a product of a character polynomial and a character exponential.


\begin{defn}
Let  $A =1^{a_1} 2^{a_2} \dots$ and $\nu = 1^{m_1}2^{m_2} \dots$ be integer partitions.  We define $${\bbX \choose \nu}  A^{\bbX - \nu}:= \prod_{n \geq 1}  {\bbX_n \choose m_n} (\sum_{d|n} d a_d)^{\bbX_n - m_n},$$ where ${\bbX_i \choose m_i}(\sigma) = {\bbX_i(\sigma) \choose m_i}.$  We define the \emph{rank of $\nu$} to be $\rank(\nu) := \sum_{i \geq 1} m_i$  so that ${\bbX \choose \nu}$ is a character polynomial of degree $\rank(\nu)$.  
\end{defn}

The function ${\bbX \choose \nu }A^{\bbX-\nu}$ specializes to the character polynomial ${\bbX \choose \nu}$ when $A = 1$  and  to the character exponential $A^{\bbX}$ when $\nu = 0$.  

\begin{ex}
	If $a_i = 0$ for all $i$  and $\nu = 1^{m_1} 2^{m_2} \dots  $  then  
$${\bbX \choose \nu } A^{\bbX-\nu}(\sigma)  = \begin{cases} 1  &  \text{ if $\sigma$ is in the conjugacy class corresponding to $\nu$ }  \\ 0 & \text{ otherwise } \end{cases} $$
\end{ex}

Let $d, s \in \bbN$.  We say that an $\FSop$ module has \emph{class $(d,s)$} if it satisfies the hypothesis and the conclusion of Theorem \ref{FSop}  for $d$ and $s$.  
Then our first result on characters is:

\begin{thm}\label{characterexps} Let $M$ be an $\FSop$ module of class $(d,s)$.    Then the character function of $M_n$ takes the form:
$$\sum_{\nu, A ~{ \rm partitions}} c(\nu, A)  { \bbX \choose \nu}  A^{\bbX - \nu},$$ where $c(\nu,A) \in \bbQ$  and $c(\nu,A) = 0$  if $|A| > d$; or $\rank(\nu) \geq s$; or  $|A| = d$ and $\rank(\nu) > 1$.     
\end{thm}

In the case $s = 1$ this sum is finite and Theorem \ref{characterexps} sharp, in the sense that the characters of degree $d$ projective $\FSop$ modules span the space of character functions.  We emphasize that the sum in Theorem \ref{characterexps}  is not finite when $s > 1$. So by itself Theorem \ref{characterexps} does not reduce the computation of the character of $M$ to a finite problem.  However if we restrict the character of $M$ to elements in $\rS_n$ whose cycles all have length $\leq C$ (which corresponds to setting $\bbX_i = 0$ for $i > C$),  then there are only finitely many terms in the sum. In this case Theorem \ref{characterexps} can be rephrased as a rationality theorem for a multivariate generating function (left to the reader).  When $C = 1$,  we precisely recover Sam--Snowden's rationality theorem.  

\begin{ex}\label{ex:charexp}
 Let $M$ be an $\FSop$ module of class $(2,2)$.   Theorem \ref{characterexps} states that the character of $M$ takes the form  $$ c_1 ~  2^{\sum_{i\geq 1 } \bbX_i} ~+ ~ c_2 ~   ~\delta_{\rm even}~ 2^{\sum_{i \geq 1} \bbX_{2i}}  ~  + c_3 + ~ \sum_{n \geq 1}  b_n ~\bbX_n ~ + ~ \sum_{n \geq 0} f_n ~  \delta_{\text{$n$-cycle}} ,$$ where $c_1,c_2, b_n, f_n \in \bbQ$.    Here $\delta_{\rm even}(\sigma)$  (resp. $\delta_{\text{$n$-cycle}}(\sigma)$)   is  $1$ if all of the cycles of $\sigma$  have even length (resp.  if $\sigma$ is an $n$-cycle)  and zero otherwise.   
\end{ex}

There are further constraints on $\FSop$ modules of class $(d,s)$:  if $M$ is a subquotient of an $\FSop$ module generated in degree $d$  we show the character of $M$ is \emph{uniquely determined} by its restriction to elements of $\rS_n$ whose cycles have length $\leq d$.  Using this, we can prove the following.  

\begin{thm}  \label{finitedim}Let $\hat \Lambda$ be the completed ring of symmetric functions.  
		    There is a finite dimensional subspace $\cU_{d,s} \subseteq \hat \Lambda,$ such that if $M$ is an $\FSop$ module of class $(d,s)$,  then its Frobenius character $\ch(M)$ is an element of $\cU_{d,s}$.     The dimension of $\cU_{d,s}$ is $p(d) + { d+ s - 1 \choose s-1}  \sum_{i = 0}^{d-1} p(i)$  where $p(i)$ is the number of integer partitions of $i$.  
\end{thm}

		This reduces the computation of the character of an $\FSop$ module of class $(d,s)$ to a finite problem.  We describe  (refined versions of)  the space $\cU_{d,s}$  in more detail in \S \ref{sec:characterspaces} and \S\ref{sec:genfunctiondualring}.  In particular, we construct bases for $\cU_{d,s}$.  We also show that if $F$ is a (restricted) class function,  defined on elements of $\rS_n$  whose cycles have length $\leq d$,  which takes the form of Theorem \ref{characterexps},  then there is a unique lift of $F$ to a (full) class function taking the same form.  This lift corresponds to an element of $\cU_{d,s}$, and we describe how to compute it.   

Using these methods, we can refine Example \ref{ex:charexp}.  

\begin{ex}\label{ex:refined}
 Let $M$ be an $\FSop$ module of class $(2,2)$.  Then its character takes the form  \begin{equation} \label{chareq} c_1 ~ 2^{\sum_{i\geq 1} \bbX_i}  ~+ ~ c_2 ~  2^{\sum_{i\geq 1} \bbX_{2i}}  ~\delta_{\rm even} ~  + ~c_3~ +  ~c_4~ \bbX_1 ~ + ~ c_5 \sum_{n \geq 2}  n ~\bbX_n +~ \sum_{n = 0}^2 f_n ~ \delta_{ n-\text{cycle}}\end{equation} where $c_n,f_n\in \bbQ$. This behavior is representative: in general the coefficients $c(\nu, A)$ of Theorem \ref{characterexps}  are governed by recurrence relations.  There are class $(2,2)$  $\FSop$ modules whose characters span the space of characters of the form \eqref{chareq} with $c_5 = 0$.  We do not know an example of a class $(2,2)$   $\FSop$ module with $c_5 \neq 0$,  but heuristic comparisons with $\FI_d$ modules  suggest that they may exist.
\end{ex}

Our final result on $\FSop$ module characters is an analog of multiplicity stability for $\FI$ modules.  

\begin{defn}
For $\lambda$ an integer partition $\lambda_1 \geq \lambda_2 \dots $ and $n \geq \lambda_1$, we let $(n,\lambda)$  be the integer partition $n \geq \lambda_1 \geq \lambda_2 \geq \dots $.  We write $|\lambda| = \sum_{i} \lambda_i$.  Given an $\rS_{|\lambda|}$ representation $V$ over a field of characteristic zero,  we write $\mult_{\lambda}  (V)$ to denote the multiplicity of the irreducible  representation corresponding to $\lambda$  in $V$.   
\end{defn}

\begin{thm}\label{GrowingRows}
	Let $\lambda$ be an integer partition.  Let $M$ be a finitely generated $\FSop$ module over a field of characteristic zero.  Then the generating function $$\sum_{n \geq |\lambda| } \mult_{(n,\lambda)}  (M_{|\lambda| + n})  ~ t^{n + |\lambda|}$$
is rational.  If $M$ is of class $(d,s)$, then the roots of the denominator are roots of unity of order $\leq d$ and $\mult_{(n,\lambda)}  (M_{|\lambda| + n})$ is eventually a quasi-polynomial of degree $\leq ds$.  
\end{thm}
Whereas for $\FI$ modules the multiplicities in Theorem \ref{GrowingRows} are eventually constant,  for  $\FSop$ modules they are eventually quasi-polynomials.





\subsection{Construction of chain complexes}
 Let $(\cC, \oplus)$  be a small monoidal category. Let $d \in \cC$ be an object, and $\cC/d$ be the over-category.  To construct functors $\Mod \cC \to \Ch(\Mod \cC)$  from $\cC$ modules to chain complexes of $\cC$ modules,   we proceed in the following simple way.  First we restrict along $\oplus: \cC \times \cC/d \to \cC$  to obtain a functor $\Mod \cC \to \Mod(\cC \times \cC/d)$.   Then we apply a functor  $F:\Mod (\cC/d) \to \Ch(\Mod k)$ in the second factor  to obtain  $\Mod(\cC \times \cC/d) \to  \Ch(\Mod \cC)$.  The functors we use are the composite, $(\id \times F) \circ \Res^{\oplus}:  \Mod \cC \to \Ch(\Mod \cC)$.

Of course, there are many possible choices for $F$: up to quasi-isomorphism they are parameterized by chain complexes of $(\cC/d)\op$ modules.  When  $\cC/d$  is equivalent to a poset $P$  
there is a particularly natural choice:  a bar construction $\rB_P: \Mod P \to \Mod k$, which we recall in \S \ref{sec:Poset}.  When  $P$ satisfies a Cohen--Macaulay property,  the functor $\rB_P$ can be replaced by a  smaller functor  $\rK_P$, which is quasi-isomorphic to $\rB_P$.  Corresponding to these two constructions, there are functors $\rK_d, \rB_d: \Mod \cC \to \Ch(\Mod \cC)$.  The conditions for $\rB_d$ and $\rK_d$ to exist are satsified for any object $d \in \cC$ when $\cC$ is one of the categories $\FI, \FSop, \VI_L,$ or $\VI\bC_L$.  Here $L$ is a field and $\VI_L$  (resp.  $\VI\bC_L$) denotes the category of finite dimensional vector spaces over $L$  and with injections (resp. injections together with a splitting, see \cite{putman2017representation}).  

When $\cC$ is an EI-category, there is a natural notion of $\cC$-module homology, which in the case of $\FI$ modules reduces to $\FI$ module homology \cite{church2017homology}.  The complexes we use admit a more general definition in terms of $\cC$ module homology which we discuss in \S\ref{extended}, although we do not use it elsewhere in the paper.

One feature of our arguments is that, for the most part, we do not deal with the complexes $\rK_d$ or $\rB_d$ directly.  Instead, we work with poset representations before applying the functors $\rK_d$ or $\rB_d$,  and use a couple of general facts to prove exactness.  So although,  when it exists, the complex $\rK_d$ is more elegant than $\rB_d$, for the first half of the paper the two are interchangeable. The applications to $\FSop$ modules in the second part use $\rK_d$ and the homology of the partition lattice.

\subsection{Relation to other work}

Beyond Sam--Snowden's theorem on the rationality of Hilbert series, we are only aware of two previously-known statements about the characters of finitely generated $\FSop$ modules in the literature.  They are both elementary: they follow from decomposition of the free $\FSop$ module generated in degree $d$ into irreducibles.  The first states that if $M$ is a subquotient of an $\FSop$ module generated in degree $d$, then the irreducible representations appearing in its decomposition have at most $d$ rows.  The second is that the multiplicity of $ (n, \lambda)$ in $M_{n + |\lambda|}$ is bounded above by a polynomial.  

There is a long history of using poset topology for what we now call categorification,  going at least back to Rota \cite{rota1964foundations}.  Our only novelty is in the specific application of these methods. The complexes associated to poset representations  that we use were first introduced by Baclawski \cite{baclawski1976whitney,baclawski1980cohen}, in terms of the cohomology of constructible sheaves on the Alexandroff space of the poset.  Baclawski also gave explicit Bar constructions, and the smaller complex $\rK_P$ in the case of Cohen--Macaulay posets.  They may also be thought of as computing derived functors for the tensor product with a representation of the opposite poset, as discussed in \S \ref{extended}. 
To minimize technical overhead, we use explicit complexes.

In the case $\cC = \FI$  the complex $\rK_d(M)$ is simply the complex obtained by iterating the functor $$M \mapsto \cone(M \to \Sigma M),$$  $d$ times. This construction and its categorification of rationality is well-known.  The complex $\rK_d(M)$ appears implicitly in many of arguments about the structure theory of $\FI$-modules, which proceed inductively using $M$, $\Sigma M$ and the cokernel of $M \to \Sigma M$.

    The way we construct chain complexes is, heuristically, adjoint to the central stability complexes which appear in the representation stability literature, introduced in \cite{putman2017representation} and axiomatized in \cite{patzt2019central}.  Whereas central stability complexes are constructed by induction along the monoidal functor $-\oplus-: \cC \times \cC \to \cC$, ours are constructed via restriction.  We do not know of a direct relationship between these two types of complexes,  although there is one with $\FI$-module homology \cite{church2017homology, putman2015stability} (which goes under the name central stability homology in Putman's original paper on central stability).

In \cite[\S 5]{sam2018hilbert}  Sam--Snowden  prove that the Hilbert series of a  finitely generated $\FI_r$ module (see \cite[\S7.1]{sam2017grobner} for definition) is a sum of functions of the form ${X \choose \nu} A^{\bbX - \nu}$  for $A = 1^{d}$ and $d \leq r$.   They also prove a statement that categorifies the rationality of Hilbert series of finitely generated $\FI_r$ modules. Our categorifications are in the same spirit as Sam--Snowden's, but their complexes do not fit into our framework.   Given a finitely generated $\FI_r$ module $M$, the complex Sam--Snowden use to categorify the rationality of $h_M(t)$ depends on the choice of subspaces of $k^{\oplus r}$ specifically adapted to $M$.  In contrast, the complexes we use depend on a discrete set of choices.  By generalizing our definitions slightly, it is possible to obtain Sam--Snowden's complexes corresponding to the subspaces of $k^{\oplus r}$  that are linearizations of subsets $s \subseteq [r]$.  

 Sam--Snowden's Gr\"obner theory, developed in \cite{sam2017grobner}, is the source of many results about Hilbert series of representations of combinatorial categories.  Our proof of Theorem \ref{FSop} relies heavily on their ideas:  we reduce to the case of monomial $\OSop$ modules, and eventually to proving a combinatorial statement about poset ideals associated to certain ordered DFAs   (concepts they introduced).   In \cite[\S11.2]{sam2017grobner}  they ask whether it is possible to prove rationality results for enhanced Hilbert series for the combinatorial categories they consider.  We view our results as an affirmative answer in the case of $\FSop$ modules, and the framework of this paper as a new strategy for approaching other cases.

\subsection{Guide to the Paper}For a reader who would like to extract  detailed results about the characters of $\FSop$ modules as quickly as possible, we suggest the following. Read Definitions \ref{def:yn}, \ref{def:Type}, and \ref{def:VAr}  for the symmetric functions $y_n$, the type of an $\FSop$ module, and the space $\cV_{A,r} \subseteq \hat \Lambda$.   Then read Definitions \ref{def:rank} and \ref{def:Fk} for the rank of an $\FSop$ module and the space $\cF_{\leq k} \subseteq \hat \Lambda$.  Then read \S \ref{sec:genfunctiondualring}, referring back as necessary.   Our conventions for symmetric functions are in \S\ref{symmconventions}, and we explain how to translate from symmetric functions to character functions  in \S \ref{sec:translation}.

In \S\ref{sec:Poset}, following Baclawski \cite{baclawski1980cohen} we discuss how to construct a chain complexes associated to representations of a poset.  In \S\ref{sec:Category} we apply the constructions of \S \ref{sec:Poset}  to build chain complexes associated to modules over combinatorial categories.  Then we prove a general result, Theorem \ref{generalprojprop}, which gives a condition for our complexes to be exact when applied to projective modules.   Combining Theorem \ref{generalprojprop}  with the work of Nagpal,  we obtain our categorifications for $\VI_q$ and $\FI$ modules.  In \S\ref{extended} we discuss a general definition of chain complexes in the setting of small monoidal EI categories.  

The next few sections of the paper are devoted to proving Theorem \ref{FSop}.  In  \S\ref{sec:Grobner}, we prove a proposition that we will use to reduce to the case of monomial $\OSop$ modules.    The heart of the proof is in \S\ref{sec:Languages}, which shows that the poset homology of certain ideals associated to ordered regular languages vanishes.   In \S \ref{sec:ThmProof} we introduce $\OSop$ and prove Theorem \ref{FSop}, using Sam--Snowden Gr\"obner theory and \S\ref{sec:Grobner} to reduce to the combinatorial result of \S\ref{sec:Languages}.
 
The second half of the paper is about the characters of $\FSop$ modules.  In \S\ref{sec:CharacterofKd} 
 we compute the Frobenius character of the complex $\rK_d(M)$  in terms of $\ch(M)$ and certain differential operators.   Using this computation, Theorem \ref{FSop} translates  into a system of differential equations for  $\ch(M)$. 
	In \S \ref{sec:DiffEqs} we solve these differential equations,  introduce the notion of the \emph{type}  of an $\FSop$ module,  and prove Theorem \ref{characterexps}.   In \S\ref{sec:characterspaces}, we study the intersection between the solution space of \S\ref{sec:DiffEqs} and the space spanned by characters of  $\rS_n$ representations whose Young diagrams have $\leq k$ rows, and  prove Theorem \ref{finitedim}.  In \S \ref{sec:genfunctiondualring} we summarize the results of previous sections, and use generating function methods to compute a basis of $\cU_{d,s}$ and prove Theorem \ref{GrowingRows}.


\subsection{Acknowledgements}  I would like to thank Steven Sam and Andrew Snowden for many helpful conversations about $\FSop$ modules,  during which I first defined the complex $\rK_d(M)$.   I would also like to thank Andrew Snowden for suggesting that I consider the category  $\VI_q$,  Nate Harman for conversations about $\VI$ and $\VI\bC$ modules, and Trevor Hyde for conversations about symmetric functions.   

\subsection{Conventions and Notation} $\,$ \label{notation}

\noindent\textbf{Categories:} Given a category $\cC$ and an object $c \in \cC$, we write $\cC/c$ for the \emph{over-category} of $C$.  This is the category whose objects are morphisms $f: d \to c$  and whose morphisms are morphisms $g: d \to d'$ such that $f' \circ g = f$.  For $c,d \in \cC$ we write $\cC(c,d)$ for the set of morphisms from $c$ to $d$.

We write $\Mod \cC = \Mod k^{\cC}$ to denote the category of functors from $\cC$ to $k$-modules, and $\Ch(\Mod \cC)$ to denote chain complexes of $\cC$ modules. We refer to elements of $\Mod \cC$ as $\cC$ represetations or $\cC$ modules.  For any functor $M: \cC \to \Mod k$ we write  $M_c$ for the value of $M$ at $c \in \cC$,  and $M_f: M_c \to M_d$ for any $f \in \cC(c,d)$.     If $M$ is a $\cC$ module, and $d \in \cC$,  we write $M|_{\cC/d}$  for the functor $\cC/d \to \Mod k$ defined by $f: c \to d \mapsto M_c$.  

  When a group $G$ acts (strictly) on the category $\cC$,  a $G$-equivariant $\cC$ representation is a representation $M$ of $\cC$  together with morphisms $M_{c} \to M_{g(c)}$ for every $g \in G$ which are compatible with  the maps $M_f: M_c \to M_d$ for $f \in \cC(c,d)$, and which satsify the identity and associativity conditions.  Concisely, $M$ is a representation of the Grothendieck construction $G \ltimes \cC$.

On size issues:  as usually defined  $\FSop, \VI_q$ and $\FI$ are essentially small, but not small.  However, we may replace these categories by equivalent small full subcategories containing all of the objects that we consider.    We do this implicitly, and treat these categories as small categories. 

   If $\cC$ and $\cD$ are categories and $F$ is a functor $F: \Mod \cC \to \Ch(\Mod \cD)$, there is a canonical extension of $F$ to a functor $\Ch(\Mod \cC) \to \Ch(\Mod \cD)$  given by applying $F$ levelwise and forming the total complex of the resulting bicomplex (using sums).   We will also denote this functor by $F$.  We do not specify our sign conventions for the total complex, since our results are independent of this choice.  \\

 \noindent\textbf{Posets:} \label{posetnotation} For $p,q$ elements of a poset $P$, we write $p \prec q$ if $p < q$ and there does not exist $r \in P$ such that $p < r < q$.  If $P$ has a top (resp. bottom) element , we denote it  by  $\hat 1$  (resp. $\hat 0$).   An \emph{upward ideal} or simply \emph{ideal} of a poset, $P$ is a subset $S \subseteq P$  which is \emph{upward closed}:  if $s \in S$ and $p \geq s$ then $s \in S$.   The ideal $S$ may be empty.   All of the poset ideals appearing in this paper are upward closed: we do not use downward closed ideals.  We write $P_{\geq x}$  for the ideal  $\{p ~|~ p \geq x\}$.

To any poset we may associate a category.   The objects of the category are the elements of  the poset, and there is a unique morphism $p \to q$  if $p \leq q$, and no morphisms otherwise.   We treat the poset and its associated category interchangeably and use the same symbol to denote both of them.  \\

\noindent{\textbf{Integer Partitions:}} An integer partition $\lambda$ is a descending sequence of nonnegative numbers $\lambda_1 \geq \lambda_2 \dots \geq \lambda_r$.  We may also specify $\lambda$ by its vector of multiplicities in $\bbN^{\oplus\infty}$:  we write $\lambda = 1^{m_1} 2^{m_2} \dots $  where $m_i$ is the number of times $i$ appears in the sequence $\lambda_1 \geq \dots$.   We say that $\lambda$ is a partition of $n$ and write $\lambda \vdash n$  or $|\lambda| = n$  when  $n = \sum_{i \geq 1} \lambda_{i\geq 1} = \sum_i i m_i$.    We will also refer to the Young diagram associated to $\lambda$  see \cite[\S7.2]{StanleyEnum}.

\noindent\textbf{Miscellaneous:}
\begin{itemize}
\item For an abelian category $\cA$  write $\Ch(\cA)$ (resp. $\Ch^{-}(A)$) for the category of chain complexes in $\cA$ (resp. non-negatively graded complexes).  
\item If $X$ is a set and $k$ is a ring, we write $k X$ for the free $k$ module with basis $X$.  Given a functor $F: C \to \Set$,  we say that the \emph{linearization of $F$}  is the functor $kF: C \to \Mod k$  given by $(kF)_c := k (F_c)$.  

\item We write $\bbN$ for the set of non-negative natural numbers, and $\bbN^{\oplus \infty}$ for the set of eventually zero sequences of natural numbers.

\end{itemize}

\part{Chain complexes categorifying rationality}
  
\section{Poset homology}\label{sec:Poset}

 In this section,  following Baclawski \cite{baclawski1980cohen}, we associate chain complexes to poset representations.  Then we discuss several examples, and recall elementary properties of these complexes.   Throughout, we let $k$ be a commutative ring.

\begin{defn} For any poset $Q$, we write $\rN Q$ for the \emph{order complex of $Q$}, the simplicial complex whose $r$ simplices are the set of chains  $$q_0 < q_1 < \dots < q_r \in Q.$$  Equivalently $\rN Q$ is the geometric realization of the simplical set $[r] \mapsto \{ q_0 \leq \dots \leq q_r \in Q\}.$ \end{defn}

\begin{defn}
	When $Q$ has top and bottom elements $\hat 1, \hat 0$,  we define $Z_Q \subseteq \rN Q$ to be the subcomplex $\rN(Q - \hat 0) \cup \rN(Q - \hat 1)$.  In other words,  $Z_Q$ is the union of all simplices corresponding to chains $q_0 < \dots < q_r$  where either $q_0\neq \hat 0$ or $q_1 \neq \hat 1$.     When $\hat 0 = \hat 1$,  $Z_Q = \emptyset$. 
\end{defn}

\begin{defn}
	We say that  a poset $P$ is \emph{graded} if for any $p \leq q \in P$ all maximal chains $p = p_0    \prec p_1  \prec   \dots \prec p_{\ell} = q $  have the same length.   When $P$ is graded, we define $\rr(p,q)$ to be $\ell$.  When $P$ has a top element, we put $\rr(p) := \rr(p, \hat 1).$
\end{defn}

\begin{defn}
		Let $P$ be a finite poset, with top element $\hat 1$.  For $p \in P$ we define the \emph{M\"obius number} of $p$, denoted $\mu(p)$, to be the Euler characteristic of the pair $(\rN[p, \hat 1], Z_{[p, \hat 1]})$.  When $P$ is graded, we define the \emph{unsigned M\"obius number} to be $\tilde \mu(p) = (-1)^{\rr(p)} \mu(p)$.  
\end{defn}

\begin{rmk} \label{CMcomparison}
 		Our definition of $\mu(p)$ agrees with $\mu(p, \hat 1)$  as usually defined (as for example in  \cite[\S1.2]{WachsPosetTopology}).  Instead of using the pair $(\rN[p,\hat 1], Z_{[p, \hat 1]})$,  it is more common to interpret the M\"obius function topologically in terms of the reduced homology of $\rN(p, \hat 1)$. Whenever $p \neq \hat 1$, the homology groups of these complexes are related by  $$\tilde H_{s-2} (\rN(p, \hat 1)) = H_s(\rN[p,\hat 1], Z_{[p, \hat 1]}) \text{ for all } s \geq 0,$$ as can be seen by writing the simplicial chains for both.    In this paper, we prefer the latter homology groups because they lead to cleaner formulas and behave more simply for products of posets.  
\end{rmk}

\begin{defn}\label{def:whitney}
		Let $P$  be a finite graded poset with top element $\hat 1$.  The \emph{Whitney Polynomial} of $P$ is $$W_P(t):=\sum_{p \in P}  \mu(p)  ~t^{\rr(p)}.$$
\end{defn}

Now we define the Bar complex associated to a $P$ representation.



\begin{defn}
A \emph{representation of $P$} is a functor $P \to \Mod k$.  For a representation $M$,  we write $M_p$ for the value of $M$ at $p$,  and if $p \leq q \in P$  we write $M_{pq}$ for the associated morphism $M_p \to M_q$.
\end{defn}


\begin{defn} \label{barconstruction} Let $P$ be a poset with top element $\hat 1$,  and let $M$ be a $P$ representation.  We define  $\rB_P(M)$ be the following chain complex.  In degree $s \in \bbN$ $$\rB_P(M)_s := \bigoplus_{\hat 1 > p_1  > p_2 > \dots > p_s} M_{p_s}.$$ The differential of $\rB_P(M)$ is given by an alternating sum $d_s = \sum_{i = 0}^s (-1)^i \del_i$, where if we represent an element of $\rB_P(M)_s$ as a tuple $(\hat 1> \dots > p_s, m)$ for $m \in M_{p_s}$, the boundary $\del_i$ is defined by $$ \del_i(\hat 1 > \dots > p_s, m) = \begin{cases} 0 & \text{if $i = 0$}  \\

(\hat 1 > \dots > p_{i-1} > p_{i+1} > \dots > p_s, m) & \text{if $s > i > 0$ } \\

(\hat 1 > \dots >  p_{s-1}, M_{p_{s}p_{s-1}}(m)) & \text{if $i = s$}.
	
 \end{cases}$$
\end{defn}

\begin{defn}
		Let $P$ be a poset, and $G$ a group acting on $P$. We say that $M$ is a  \emph{$G$-equivariant $P$ representation} if $M$ is a $P$ representation and there are morphisms $g: M_{p} \to M_{g(p)}$ for every $g \in G$, which are compatible with the maps $M_{pq}: M_p \to M_q, p \leq q \in P$ and satisfy the obvious associativity and identity properties.  
\end{defn}

When $P$ is a poset with a top element and $M$ is a $G$-equivariant $P$ representation for some $G$-action on $P$,  the chain complex $\rB_P(M)$ is a naturally a complex of $G$-representations.  The element $g \in G$ acts by taking $(p_0 < \dots < p_r, m)$ to $(g(p_0) < \dots < g(p_r), g(m))$.

\subsection{Koszul complexes} \label{sec:Koszul}  Following Baclawski \cite{baclawski1980cohen},  when the poset $P$ satisfies a connectivity condition it is possible to replace $\rB_P(M)$  by a smaller  complex  $\rK_P(M)$, which we define in this subsection.


 \begin{defn}
We say that a poset $P$ is \emph{upper CM over $k$}  if it is graded, has a top element $\hat 1$, and satisfies the following conditions for every $p \in P$:  \[ H_{j}(\rN[p, \hat 1], Z_{[p, \hat 1]}; k ) = 0,~  \text{ for all }     0 \leq j < \rr(p),\] and $H_{\rr(p)}(\rN[p, \hat 1], Z_{[p, \hat 1]}; k )$ is a free $k$ module.  
\end{defn} 

\begin{ex} If $P$ is Cohen--Macaulay over $k$ then by defintion $\tilde H_{s-2}(\rN(x,y)) = 0$ for $s < \rank(x,y)$ \cite[Defn 3.1]{bjorner1982introduction}. Therefore $P$ is upper CM over $k$.    Examples include arbitrary geometric lattices  \cite{folkman1966homology}.
\end{ex}

    Since $\rN[p, \hat 1]$ is $\rr(p)$ dimensional,  when $P$ is upper CM  the unsigned M\"obius number computes the rank of the the top nonvanishing homology group of $(\rN[p, \hat 1], Z_{[p,1]})$.  In other words we have $$H_{\rr(p)}(\rN[p, \hat 1], Z_{[p, \hat 1]}; k ) \iso k^{ \oplus \tilde \mu(p)}.$$

\begin{defn}
	When $P$ is upper CM over $k$,  there is a functor from $P$ representations to chain complexes of $k$ modules,  $\rK_P$,  defined as follows.   

 Let $M$ be a $P$-representation.  
We filter $\rB_P(M)$ by defining $$F_i(\rB_P(M)_s):= \bigoplus_{\hat 1 > \dots > p_s ~|~  \rr(p_s) \leq i}  M_{p_s}.$$  The complex $\rK_P(M)$ is defined to be the $E_1$ page of the spectral sequence associated to this filtration. 
\end{defn}

Let us describe $\rK_P(M)$ more concretely, by computing the $E_1$ page.  The associated graded complex of the filtration $F_i$ has the same chain groups as $\rB_P(M)$ but $\del_s$ acts by zero: in other words the differential of the $E_0$ page is   $d_{0}^s =\sum_{i = 0}^{s-1} (-1)^i \del_i .$  Since $d_0^s$ preserves the grading by $p \in P$,  the $E_0$ page  splits as a direct sum of complexes $\bigoplus_p C(p) \otimes M_p$ for $p \in P$  where $$C(p)_s := \bigoplus_{\hat 1 > p_1>  \dots > p_{s-1} > p}  k.$$ By definition, we see that $C(p)$ is the chain complex  computing the simplicial homology of $(\rN[p, \hat 1], Z_{[p,\hat 1]})$.  Because $P$ is upper CM  the homology of $C(p)$ is concentrated in degree $\rr(p)$, and it has rank $\tilde \mu(p)$.  Therefore the spectral sequence associated to the filtration degenerates after the $E_1$ page, and on the $E_1$ page it takes the form  $$M_{\hat 1} \leftarrow \bigoplus_{p,~ \rr(p,\hat 1) = 1} M_{p} \leftarrow \dots \leftarrow \bigoplus_{p, ~\rr(p, \hat 1) = s}   H_{s}(\rN[p, \hat 1], Z_{[p,\hat1]}; k) \otimes M_p \leftarrow \dots $$
This is the complex $\rK_P(M)$. Non-canonically it takes the form    $$M_{\hat 1} \leftarrow \bigoplus_{p\in P,~ r(p) = 1} M_p \leftarrow  \bigoplus_{p \in P,~ \rr(p) = 2}  M_p^{ \oplus \tilde \mu(p)} \leftarrow \dots  .$$ 
The differential is induced by the action of $(-1)^s\del_s$.

 Both $\rK_P(M)$ and $\rB_P(M)$ compute the same homology groups,  but $\rK_P(M)$ depends more closely on the combinatorics of $P$  and is much smaller than $\rB_P(M)$.  
\begin{prop}\label{KBarcomparison}There is an edge map $\rK_P(M) \into \rB_P(M)$ which induces a canonical isomorphism $H_\bdot(\rK_P(M)) \iso H_\bdot(\rB_P(M))$.\end{prop}  
\begin{proof}
	The complex $\rK_P(M)$  can be identified with the subcomplex of $\rB_P(M)$ which in degree $s$ is $$ \ker(\sum_{i = 0}^{s-1} (-1)^i \del_i) \subseteq  \bigoplus_{\hat 1 \succ p_1 \succ \dots \succ p_s}  M_{p_s}.$$  Since the spectral sequence degenerates, this inclusion is a quasi-isomorphism.  
\end{proof}

When $P$ is an upper CM poset, $G$ is a group acting on $P$, and $M$ is a $G$-equivariant $P$ representation,  the $G$ action on $\rB_P(M)$ induces one on $\rK_P(M)$.    The degree $s$ term, $$\bigoplus_{p, ~\rr(p, \hat 1) = s}   H_{s}(\rN[p, \hat 1], Z_{[p,\hat1]}; k) \otimes M_p,$$ is a $G$ representation in the obvious way:  $g \in G$ acts by tensoring and summing the maps $g: M_p \to M_{g(p)}$ and $g: H_{s}(\rN[p, \hat 1], Z_{[p,\hat1]}; k) \to  H_{s}(\rN[g(p), \hat 1], Z_{[g(p),\hat1]}; k)$ induced by the functoriality of homology.

\begin{rmk}\label{abstractkoszulremark}From an abstract perspective both $\rB_P(M)$ and $\rK_P(M)$ compute the derived functors 
 of tensoring with a  $P\op$ representation  $S({\hat 1})$ (see Definition \ref{S(hat1)}).  The complex $\rB_P(M)$ arises from the Bar resolution of $S({\hat 1})$, and the complex $\rK_P(M)$ arises from a minimal resolution (assuming that $k$ is a field).   Essentially, the complex $\rK_P(M)$ is a Koszul complex:  there is a generalization of the theory of Koszul duality from graded algebras to graded linear categories, such that the category $kP$ is Koszul if $P$ is Cohen--Macaulay  \cite{woodcock1998cohen,li2014generalized}. In this case, $\rK_P(M)$ is a Koszul complex.
\end{rmk}

\subsection{Examples}  \label{posetexamples}
We discuss the posets that arise in the main examples of this paper.

\begin{defn}
Let $x$ be a finite set. A \emph{(set) partition} $p$ of $x$ is a set of non-empty subsets of $x$, such that each $i \in x$ is contained in exactly one element of $p$.  We call the elements of $p$  the \emph{blocks of $p$}. Let $\rP(x)$ be the set of all partitions of $x$.  Given $p,q \in \rP(x)$, we define $p \leq q$ if every block of $q$ is contained in a block of $p$.    Under this relation,  $\rP(x)$ is a finite lattice,   called the \emph{partition lattice of $x$}.  The top element is the discrete partition and the bottom element is the indiscrete partition.    When $n \in \bbN$, we define $\rP(n) := \rP([n])$. 
\end{defn}

\begin{ex} The partition lattice $\rP(n)$ is Cohen--Macaulay over any ring and for $p \in \rP(n)$ $$\tilde \mu(p) = \prod_{b \text{ block of $p$}} (|{b}| - 1)!,$$  
and its Whitney polynomial is $$  W_{\rP(n)}(t) = \prod_{j= 1}^{n-1}  (1 -j t),$$  see  \cite[\S7]{stanley1982some}.
\end{ex}

\begin{defn}
  For a finite set $x$  the  \emph{Boolean lattice}  $\bB(x)$ is the poset of subsets of $x$.  We write $\bB(n) := \bB([n])$
\end{defn}

\begin{ex} The Boolean lattice $\bB(n)$ is Cohen--Macaulay  and $\tilde \mu(s) = 1$ for every $s \in \bB(n)$.  
Its Whitney polynomial is $W_{\bB(n)}(t) = (1-t)^n.$ See  \cite[\S4]{stanley1982some}. 
\end{ex}

\begin{defn}
	Let $V$ be finite dimensional vector space over the finite field $\bbF_q$.  We write $\bB_q(V)$ for the poset of subspaces of $V$ and $\bB_q(n):= \bB_q(\bbF_q^{\oplus n})$.    
\end{defn}

\begin{ex}
 The poset fon $\rB(V)$ is $\bB_q(n)$ is Cohen--Macaulay and $\tilde \mu(w) = q^{ c \choose 2},$ where $c= \rr(w)$ is the codimension of $w \subseteq \bbF_q^{n}$.  It follows from the $q$-binomial theorem that its Whitney polynomial is 
  $$W_{\bB_q(n)} (t) = \prod_{i = 0}^{n-1} (1 - q^{i} t).$$ See \cite[\S5]{stanley1982some}.  It is natural to think of $\bB(n)$ as a specialization of $\bB_q(n)$ to $q = 1$.  
\end{ex}

In fact, if $V$ is a finitely generated module over $L$ where $L$ is a field or a finite ring,  then the poset of submodules of $V$ is Cohen--Macaulay by the Solomon--Tits theorem and the Cohen-Macaulayness of supersolvable lattices, respectively.   There is another example, related to the category $\VI\bC_L$.   

\begin{ex}\label{splitbuildingex}
	Let $L$ be a field, and $V$ be a vector space over $L$.  There is a poset whose elements are pairs of subspaces $(W,W')$ satisfying $V = W \oplus W'$, with relation defined by $(W_1,W_2') \leq (W_2, W_2')$ if  $W_1 \subseteq W_2$ and $W_1' \supseteq W_2'$.   Lehrer--Rylands proved that this poset is Cohen--Macaulay \cite{lehrer1993split}.  
\end{ex}

\subsection{Complexes associated to product posets}
 
Let $P \times Q$ be a product of two posets with top elements $\hat 1_P$ and $\hat 1_Q$.  Suppose we are given a $P \times Q$ module $M$.   We may consider $M$ as a functor $P \to \Mod Q$,  and post-composing with $\rB_Q$ we obtain a chain complex of $P$ modules, $p \mapsto \rB_Q(M_{(p,-)})$.  We may apply $\rB_P$ to this complex, as described in \S\ref{notation}, to obtain a chain complex of $k$ modules. 

\begin{defn}
	We define $\rB_{(P,Q)}(M)$ to be the chain complex $\rB_P(p \mapsto \rB_Q(M_{(p,-)}))$  and $\rB_{(Q,P)}(M)$ to be the chain complex $\rB_Q(q \mapsto \rB_P(M_{(-,q)}))$.  
\end{defn}

There is an isomorphism of complexes $\rB_{(P,Q)}(M) \iso \rB_{(Q,P)}(M)$.  More concretely, their degree $n$ term is $$\bigoplus_{i+j = n} \bigoplus_{p_i < \dots < \hat 1_P} \bigoplus_{q_j < \dots < \hat 1_Q} M_{(p_i, q_j)},$$ and their differentials agree up to a sign introduced by the definition of the total complex.

\begin{prop}\label{product}
Let $P, Q$ be finite posets with top elements.   Let $M$ and $N$ be $P$ and $Q$ representations, respectively. Let $M \sqtimes N$ be the $P \times Q$ representation $(p,q) \mapsto M_{p} \otimes N_q$. Then $\rB_{(P,Q)}(M \sqtimes N) \iso \rB_P(M) \otimes \rB_Q(N)$.  In particular if $\rB_Q(N)$ is exact, then $\rB_{(P,Q)}(M \sqtimes N)$ is exact.  
\end{prop}
\begin{proof}
	This is true by the definition of the tensor product of chain complexes. 
\end{proof}

We may extend the definition of $\rB_{(P,Q)}$ to a finite product of posets with top elements $P_1  \times P_2 \times \dots \times P_r$  in the obvious way to obtain a functor $\rB_{(P_1, \dots, P_r)}$.  Proposition \ref{product} carries over to this setting, and we also have the following.

\begin{prop}\label{filtration}
Let $P_1, \dots, P_r$ be finite posets with top elements and let $L$ be a $P_1 \times \dots \times P_r$ representation.  We put $Q := P_2 \times \dots \times P_r$, and for $p \in P_1$ write $L|_{p \times Q}$ for the $Q$ representation $q \mapsto L_{(p,q)}$.  Suppose that that for every $p \in P_1$  we have that $\rB_{(P_2, \dots, P_r)}(L|_{p \times Q})$ is exact.   Then  $\rB_{(P_1, \dots ,P_r)}(L)$ is exact.
\end{prop}
\begin{proof}
Applying $\rB_{(P_2, \dots, P_r)}$ to $L$, by hypothesis we obtain an exact complex of $P_1$ representations $p \mapsto \rB_{(P_2, \dots, P_r)} (L|_{p \times Q})$. Denote this complex by $C$.  Since $\rB_{P_1}$ is exact,  $\rB_{P_1}(C) \iso \rB_{(P_1,\dots, P_r)}(L)$ is exact. 
\end{proof}

\begin{remark} If the definition of $\rB_{(P,Q)}$  seems ad hoc, we remark that for any $P \times Q$ representation  $L$   there is a canonical quasi-isomorphism $EZ : \rB_{(P,Q)}(L) \simto \rB_{P \times Q}(L)$.  To construct it, note that $\rB_{(P,Q)}(L)$ is the totalization of a normalized  bicomplex of a bisimplicial abelian group,  and $\rB_{P \times Q}(L)$ is the normalized chain complex of that bisimpicial group's diagonal.  Then apply the Eilenberg--Zilber theorem stated in \cite[\S8.5]{weibel1995introduction}.   In fact, if $P$ and $Q$ are upper CM,  then so is $P \times Q$  and $EZ$ induces an \emph{isomorphism} between the subcomplexes $\rK_{(P,Q)}(L)$ and $\rK_{P \times Q}(L)$.  We use $\rB_{(P,Q)}$ rather than $\rB_{P \times Q}$, because all the properties we need are immediate for it.  
\end{remark}

\subsection{Poset homology of ideals}

\begin{defn}
Associated to any upward ideal $I$ in a poset $P$, there is a functor $\rS I: P \to \Set$  given by $$p \mapsto \begin{cases}  * & \text{if $p \in I$}
\\ \emptyset &\text{otherwise},
\end{cases}$$ where $*$ denotes the terminal set $\{\{ \emptyset\}\}$. (The action of $\rS I$ on morphisms is uniquely determined by its value on objects).  Postcomposing with the free functor from sets to $k$ modules we obtain a $P$ representation, denoted $kI$, such that  $kI_p = k$ if $p \in I$ and $kI_p = 0$ otherwise.  
\end{defn}

The order complex $\rN(I - \hat 1)$ is a subcomplex of $\rN(I)$, and by definition $\rB(kI)$ is the chain complex computing the relative simplicial homology $(\rN I, \rN( I - \hat 1))$ with coefficients in $k$.
We need to know that the homology of the bar complex applied to a principal ideal is zero.  

\begin{prop}\label{minimalelt}
	Let $P$ be a finite poset with top element $\hat 1$.  Let $x \in P, x \neq \hat 1$, and let $P_{\geq x} = \{p \in P~|~ p \geq x\}$ be the principal ideal generated by $x$.  Then the homology of $\rB(k P_{\geq x})$ vanishes.  
\end{prop}
\begin{proof}
		Since both $I$ and $I - \hat 1$  contain $x$ as their smallest element, $\rN I$ and $\rN(I - \hat 1)$ are both contractible, and so the relative homology vanishes.
\end{proof}

We will also need a simple fact about $P$-sets.  
\begin{prop}\label{PSet}
	Let $P$ be a poset with top element $\hat 1$.  Let $F: P \to \Set$ be a functor such that $F_{pq}$ is an injection for all $p \leq q \in P$.   Given $x \in F_{\hat 1}$,  let $F\langle x\rangle_p := \{y \in F_p ~|~ F_{p \hat1}(y) = x\}$.  Then $F\langle x\rangle$ is a subfunctor, and there are isomorphisms $\bigsqcup_{x \in F_{\hat 1} } F\langle x\rangle \iso F$
and $F\langle x\rangle \iso \rS I \langle x \rangle$  where $I\langle x\rangle$ is the ideal  $$ I\langle x \rangle := \{p \in P ~|~ \exists y \in F_p, F_{p \hat 1}(y) = x\}.$$
\end{prop}

\begin{proof}
  This is straightforward. The map  $\bigsqcup_{x \in F_{\hat 1} } F\langle x\rangle \to F$ is induced by the inclusions $F\langle x\rangle \subseteq F$  and it is an injection because if $x,y$ are distinct elements of $F_{\hat 1}$  the fact that the transition maps are inclusion implies that $F\langle x\rangle \cap F\langle y\rangle = \emptyset$.  The isomorphism $\rS I\langle x \rangle \to F\langle x\rangle$ is given by defining the image of $*$ in $F\langle x\rangle_p$ to be the unique element $y$ such that   $F_{p \hat 1}(y) = x$, if it exists.  
\end{proof}

\section{Chain complexes associated to categories}\label{sec:Category}
 
In this section, we apply the definitions of $\rB_P$ and $\rK_P$  to construct chain complexes associated to modules over combinatorial categories.

\subsection{Poset vs Posetal} In our main examples, $\FI, \VI_q,$ and  $\FSop$,  the over categories of objects are equivalent to posets, but \emph{not} isomorphic to posets.  Although this may seem to be a fussy distinction, we want our complexes to be canonical so that they carry group actions.    Accordingly, we recall the following notion.

\begin{defn}
	A small category $\cJ$ is \emph{posetal}  if there is at most one morphism between any two objects of $\cJ$.    
\end{defn}

When $\cJ$ is posetal, its set of isomorphism classes $|\cJ|$ forms a poset, by defining $[j] \leq [j']$ if there is a morphism $j \to j'$. There is a canonical functor $$ \cJ \to (|\cJ|, \leq), ~ j \mapsto [j]$$
which is an equivalence of categories.  Inducing and restricting along this functor, we obtain an equivalence of categories between $\Mod \cJ$ and $\Mod |\cJ|$.    Given a $\cJ$ module $M$,  the associated $|\cJ|$ module takes an equivalence class $x \in |\cJ|$ to $$|M|_x := \frac{\bigsqcup_{j \in \cJ,~ [j] = x} M_j}{\text{ $(j, m) \sim (j',m')$ if for the unique map $f:j \to j'$, $M_f(m) = m'$}}.$$  For any $j \in \cJ$  such that $[j] = x$ there is a canonical isomorphism $M_j \to |M|_x$.  And if $x \leq y \in |\cJ|$  there is a unique $|M|_x \to |M|_y$  such for any $j,k \in \cJ$ with $[j] = x, [k] = y$  the obvious diagram commutes.

\begin{defn}
	Let $\cJ$ be a posetal category.  For any functor $F_{|\cJ|}: \Mod |\cJ| \to \cD$  then we define $F_{\cJ}$ to be the functor $F_\cJ(M) := F_{|\cJ|}(|M|)$.  
\end{defn}

Thus when $\cJ$ is posetal and has a terminal object  (equivalently $|\cJ|$ has a top element),  then we have a definition of $\rB_{\cJ}(M)$. And when $|\cJ|$ is upper CM we have a definition of $\rK_\cJ(M)$.  Similarly, given a product of posetal categories with terminal objects $\cJ_1 \times \dots \times \cJ_r$,  we may define $\rB_{(\cJ_1, \dots, \cJ_r)}$. 

Finally, suppose that $G$ is a group which acts on $\cJ$ by isomorphisms.  Then $G$ acts on $|\cJ|$, and if $M$ is a $G$-equivariant $\cJ$ representation,  $|M|$ is a $G$-equivariant $|\cJ|$ representation.  If $x \in |\cJ|, j \in \cJ$ and $[j] = x$, then for all $g \in G$ we have $g(x) = g([j])$  and $g: M_{x} \to M_{g(x)}$ is the unique morphism compatible with $g: M_{j} \to M_{g(j)}$.

\subsection{Construction of complexes}
Let $(\cC, \oplus)$ be a monoidal category, and $d \in \cC$.  Restricting along the functor $\oplus: \cC \times \cC/d \to \cC$, $(c,c' \to d) \mapsto c \oplus c'$, we obtain a functor $$\Mod \cC \to  (\Mod \cC/d)^{\cC},$$ taking a $\cC$ module $M$ to the functor $c \mapsto (M_{c \oplus -})|_{\cC/d}$. 
Note that $\id: d \to d$ is always a terminal object of $\cC/d$, so if $\cC/d$ is posetal  contains a top element.  

\begin{defn}
	If  $\cC/d$ is posetal, we define $\rB_{d}: \Mod \cC \to \Ch(\Mod \cC)$ to be the composite $$\Mod \cC \to  (\Mod \cC/d)^{\cC} \overset{\rB_{\cC/d}}\longrightarrow  \Ch(\Mod k)^\cC = \Ch (\Mod \cC).$$  Thus for every $c \in \cC$, the complex $\rB_d(M)_c \in \Ch(\Mod k)$ is $\rB_{\cC/d}(M_{c \oplus -} |_{\cC/d})$.  
\end{defn}

We define the complex $\rK_d$ similarly.  

\begin{defn}If $\cC/d$ is equivalent to an upper CM poset,  then we define $\rK_d: \Mod \cC \to \Mod \cC$ to be the composite  functor  $ \Mod \cC \to  (\Mod \cC/d)^{\cC} \overset{\rK_{\cC/d}}\longrightarrow \Ch(  \Mod \cC)$.  
\end{defn}  

The following elementary proposition shows that $\rB_d$ is always defined for categories in which all morphisms are monomorphisms.  

\begin{prop}
 	Let $\cC$ be a category.   Then  $\cC/d$ is posetal for all $d \in \cC$ if and only if every morphism of $\cC$ is a monomorphism.  
\end{prop}

For $d \in \cC$, the group  $\Aut(d)$  acts on the category $\cC/d$  by postcomposition.  For $N$ a $\cC$ module,  the $\cC/d$ module  $N|_{\cC/d}$ is $\Aut(d)$ equivariant: an element $g \in \Aut(d)$ acts by $(f, n) \mapsto (g \circ f, n)$  for $f: c \to d$ and $n \in N_f$.    Therefore, when $\cC/d$ is posetal  (resp. $\cC/d$ is equivalent to an upper $CM$ poset) the group $\Aut(d)$ acts on  $\rB_{d}(M)$  (resp. $\rK_d(M)$) for all $\cC$ modules $M$.

Assume that $\cC/d$ is posetal.  Instead of passing through $|\cC/d|$, we may obtain a complex that is isomorphic to $\rB_d(M)$  by restricting $(M_{c \oplus -})|_{\cC/d}$  to a skeleton  $\cS \subseteq \cC/d$, and applying  $\rB_\cS$ (since $\cS$ is necessarily a poset).  Recall that a full subcategory  $\cS$ is \emph{a skeleton of $\cC/d$} if any two isomorphic objects  of $\cS$ are equal, and every object of $\cC/d$ is isomorphic to one in $\cS$.

\subsection{Examples} Now we discuss the complexes arising in our main three examples.  

\begin{defn}When $\cC, \oplus$ is a monoidal category, for any $c \in \cC$ we define $\Sigma^c M$ to be the $\cC$ module $ (\Sigma^c M)_{c'} := M_{c \oplus c'}$.  
\end{defn}

\begin{ex}
	 Let $d \in \FSop$.  Then $|\FSop/d|$ is $\rP(d)$, the lattice of set partitions of $d$.   A surjection $f: d \onto y$ corresponds to the following partition of $d$:   $i,j \in d$ are in the same block of $p$ if $f(i) = f(j)$. Two surjections are isomorphic if and only if their associated partitions agree. 

Let $M$ be an $\FSop$ module.  Then $\rK_d(M)$ takes the form $$\Sigma^{[d]} M \leftarrow \bigoplus_{p \in \rP(d), ~r(p) = 1} \Sigma^{p} M \leftarrow  \dots \leftarrow \bigoplus_{p \in \rP(d),~ r(p) = s} \Sigma^p M  \otimes H_{s}(\rN[p, \hat 1], Z_{[p, \hat 1]}) \leftarrow \dots $$ (Recall that for us the partition $p$ equals its set of blocks).   The first differential is given by summing the maps induced by the surjection $d \onto p$ for each partition $p$.  
\end{ex}

\begin{defn}
	Let $\FI$ be the category of finite sets and injections. 
\end{defn}

\begin{ex}
	Let $d \in \FI$.  Then $|\FI/d|$ is the poset of subsets of $d$, and for an $\FI$ module $M$, the complex $\rK_d(M)$ takes the form $$\Sigma^d M \leftarrow \bigoplus_{i \in d} \Sigma^{d-i} M \bigoplus_{\{i,j\} \subseteq d} \leftarrow \Sigma^{d - \{i,j\}} M  \leftarrow \dots \leftarrow M,$$ where, up to signs, the differentials are induced by the inclusions of subsets.   Notice that the action by $\rS_d$ is nontrivial: for instance $\rS_d$ acts on the highest term of the complex by the $\id_M \otimes \sgn$.  
\end{ex}

\begin{defn}\label{def:VIR} Let $R$ be a ring.  Then  $\VI_R$ is the category of $R$ modules that are isomorphic to $R^{\oplus n}$  for some $n \in \bbN$,  and injections between them.   Let $q$ be a prime power.  We define $\VI_q := \VI_R$.  
\end{defn}

\begin{ex}
 For $w \in \VI_q$ the poset $|\VI_q/w|$ is the poset of subspaces of $w$.  Then $\rK_w(M)$ takes the form $$\Sigma^w M \leftarrow  \dots \leftarrow M \otimes {\rm Stein}(w),$$  where ${\rm Stein}(w)$ is the Steinberg representation of $\GL(w)$.    The degree $s$ term is  $$\bigoplus_{v \subseteq w, ~\codim(v) = s} \Sigma^{v} M  \otimes {\rm Stein}(w/v).$$ 
\end{ex}

More generally, the complex $\rK_d$ is defined for any object $d \in \VI_L$  where $L$ is a field or a finite ring,  since as discussed in \S\ref{posetexamples}, the poset of submodules of a finitely generated module is Cohen--Macaulay. 

\begin{ex}  Let $L$ be a field and $\VI \bC_L$ be the category of vector spaces and complemented injections of \cite{putman2017representation}.  For any $V \in \VI \bC_{L}$,  the  overcategory  $\VI \bC_{L} / V$ is equivalent to the poset defined in Example \ref{splitbuildingex}.  Therefore the complex $\rK_V$ is defined.  
\end{ex}

\subsection{Rational Hilbert series}
We discuss how exactness of the complex $\rB_d(M)$ categorifes rationality, in an axiomatic setting.  (In our three examples, it is easy to see the categorification directly using $\rK_d(M)$).   The key is the following general identity. 

\begin{prop}\label{Kclass} Let $\cC$ be a small monoidal category, and let $d \in \cC$ be an object with $\cC/d$ equivalent to a finite poset.    Let $M$ be a $\cC$ module over a commutative Noetherian ring $k$ , such that $M_c$ is finitely generated for all $c \in \cC$. Then for every $c \in \cC$ we have $$ [\rB_d(M)_c] = \sum_{[f : c' \to d] \in |\cC/d|}  \mu([f]) [M_{c \oplus c'}]$$ in the Grothendieck group of finitely generated $k$ modules.  
\end{prop}
\begin{proof}
	We may neglect the differentials of $\rB_d(M)$.  Then, as in \S\ref{sec:Koszul},  the complex $\rB_d(M)$ splits as a sum of $C([f]) \otimes M_{c \oplus c'}$  over $[f: c' \to d] \in |\cC/d|$.    By the definition of $\mu$ as an Euler characteristic,   $[ C([f])] = \mu([f]) [k]$, so the result follows.
\end{proof}

We specialize to the following situation, which holds in our main examples. 

 Let $k$ be a field, and let $\cC$ be a monoidal category, all of whose morphisms are monomorphisms.  Its set of isomorphism classes $|\cC|$ is a monoid,  suppose that we have an isomorphism $s:(|\cC|, \oplus) \to (\bbN, +)$, which we use to identify isomorphism classes of objects with natural numbers.   Suppose that if there is a morphism from the $n^{\rm th}$ object to the $m^{\rm th}$ one, then $n \leq m$.   For any $f: c \to d \in \cC/d$  define $r(f) = s(d) - s(c)$.

\begin{defn}
 	The \emph{Hilbert series}  of a pointwise finite dimensional $\cC$ module  (relative to $s$) is $$H_M(t):= \sum_{n \in \bbN}  \dim( M_n) ~ t^{n} \in \bbZ[[t]]. $$
\end{defn}

\begin{defn}
 	 We define the \emph{Whitney polynomial} of $\cC/d$ to be  $$W_{\cC/d}(t) := \sum_{p \in |\cC/d|}   \mu(p)  ~t^{r(p)}  \in \bbZ[t].$$
\end{defn}

In the cases $\cC = \FI, \VI_q,$ and $\FSop$,  this definition of the Whitney polynomial agrees with the previous one.  The following statement shows that when $\rB_d(M)$ is exact the Hilbert series of $M$ is rational.

\begin{prop}
	Let $M$ be a $\cC$ module, and suppose that $\rB_{d_1} \circ \dots \circ \rB_{d_r} (M)_c$ is exact for all  $c \in \cC$ with $s(c) \gg 0$.   Then $H_M(t)$ is rational, with denominator $\prod_{i = 1}^r W_{\cC/d_i}(t)$.
\end{prop}
\begin{proof}
		We consider the special case $r = 1$, and put $d := d_1$.   We write $h_1(t) \approx h_2(t) \in \bbZ[[t]]$ if their difference is a polynomial.  Then we have that by hypothesis \begin{align}0 \approx \sum_{n \geq 0} \chi(\rB_d(M)_n) t^{n + s(d)} & = \sum_{n \geq 0} \sum_{[f: c \to d] \in |\cC/d|} \mu(f) \chi(M_{n + s(c)}) ~ t^{n + s(d)}    \nonumber
\\ \nonumber & =  ~\sum_{[f: c \to d] \in |\cC/d|} ~\sum_{m \geq s(c)} \chi(M_m) \mu(f) t^{m + r(f)}  \approx H_M(t)  W_{\cC/d}(t),
\end{align}
where the first equality follows from Proposition \ref{Kclass}.  (Here $\chi(\rB_d(M)_n)$ denotes $\chi(\rB_d(M)_c)$ for some $c \in \cC$ with $s(c) = n$).  Therefore $H_M(t) W_{\cC/d}(t)$ is a polynomial, and so $H_M(t)$ is rational with denominator $W_{\cC/d}(t)$.    The general case is similar,  iterating applications of Proposition \ref{Kclass}.  
\end{proof}

\subsection{Composites and products}\label{subsec:prodcomposite}
Suppose that $\cC$ is a small monoidal category and every morphism is a monomorphism.  Given $d_1, \dots, d_r \in \cC$  we may iterate $\rB_{d_i}$  to obtain a complex  $$\rB_{d_1} \circ \rB_{d_2} \circ \dots \circ \rB_{d_r}(M)$$
for every $M \in \Mod \cC$.  When the over categories of $\cC$ are equivalent to upper CM posets, we may do the same with $\rK_{d_i}$. Alternately, we may construct this complex in one step, by using the functor $$\oplus: \cC \times \cC/d_{1} \times \dots \times \cC/d_r \to  \cC$$
to restrict a $\cC$ module to an object of $\Mod (\cC/d_1 \times \dots \times \cC/d_r)^{\cC}$  and applying the construction  $\rB_{(\cC/d_1, \dots, \cC/d_r)}$.  The complex we obtain is isomorphic to $\rB_{d_1} \circ \dots \circ \rB_{d_r}(M)$.  

Similarly, we note the following elementary proposition.  
\begin{prop}\label{symmonoidal}
	Suppose that $\cC, \oplus$ is a symmetric monoidal small category, and every morphism is a monomorphism.  Then there are natural isomorphisms  $\rB_d(\Sigma^c M) \iso \Sigma^c \rB_d(M)$ and $\rB_{d_1} \circ \rB_{d_2}(M) \iso \rB_{d_2} \circ \rB_{d_1}(M)$  for every $c,d,d_1,d_2 \in \cC$.  
\end{prop}

\subsection{Categorifications for $\FI$ and $\VI_q$}

 Let $k$ be a commutative ring.  Let $\cC$ be a category such that every morphism is a monomorphism, and every endomorphism is an isomorphism.  Because every endomorphism is an isomorphism, the isomorphism classes of $\cC$ inherit a natural partial order defined by $[c] \leq [d]$ if there is a map from $c$ to $d$.   We write $\bP(c)$ for the projective $\cC$ module $d \mapsto k \cC(c,d)$.  Then we have the following general result.

 \begin{thm}\label{generalprojprop}
Let $c \in \cC$.  Suppose that $\Res^{\oplus } \bP(c)$  splits as a sum of $\cC \times \cC$  modules of the form $\bP(c_1) \otimes_k \bP(c_2)$ where $c_i \in \cC$ and  $[c_i] \leq [c]$.   Then   $\rB_d(\bP(c))$ is exact for all $d \in \cC$ such that $[d] > [c]$.  
\end{thm}
\begin{proof}
	 
   Fix $c' \in \cC$.  To show that $\rB_d(\bP(c))_{c'}$ is exact, we consider the  representation $\cC/d \to \Mod k$  $$ (b \to d)  \mapsto  \bP(c)_{c' \oplus b},$$ from which $\rB_d(\bP(c))_{c'}$ is constructed by applying $\rB_{\cC/d}$.     By our hypothesis, this representation splits as a sum of representations of the form  $(\bP(c_1)_{c'}) \otimes_k \bP(c_2)|_{\cC/d}$  for some choice of $c_i \in \cC$ with $[c_i ]\leq [c] < [d]$.   Since  $(\bP(c_1)_{c'})$ is a free $k$ module and $\rB_{\cC/d}(-)$ preserves sums and tensor products with $k$-modules,  it suffices to show that the homology of $\rB_{\cC/d}(\bP(c_2)|_{\cC/d})$ vanishes for all $c_2 \in \cC$ with $[c_2] < [d]$.  
Further, note that $$\bP(c_2)|_{\cC/d} = \bigoplus_{f: c_2 \to d } k(\cC/d)(f, -).$$
Under the equivalence between $\cC/d$ modules and $|\cC/d|$ modules,   $k(\cC/d)(f,-)$ corresponds to the $|\cC/d|$ module associated to the principal ideal $\{ x \in |\cC/d|,  ~ x \geq [f] \}$.   Since $[c_2] < [d]$,  $f$ is never an isomorphism, and so $[f] < \hat 1_{|\cC/d|}$.  So  by Proposition \ref{minimalelt}, it follows that $\rB_{\cC/d}(k\cC/d(f,-))$ is exact.  Thus $\rB_d(\rP(c))$ is exact.  
\end{proof}

Theorem \ref{generalprojprop} is formal, and can be extended to a general setting where  we no longer assume every morphism of $\cC$ is a monomorphism, as discussed in \S\ref{extended}.   Using it, we may prove our categorification theorems for $\FI$ and $\VI_q$ modules.

\begin{thm}\label{VIq}
Let $k$ be a commutative noetherian ring and let $q$ be a prime power that is invertible in $k$.  Let $M$ be a finitely generated $\VI_q$ module over $k$.  Then there exists $D \geq 0$ and $n \geq 0$   such for all $d \geq D$  and $W \in \VI_q$ with $\dim W \geq  n$ the homology of $$(\rK_{\bbF_q^{d}}(M))_W$$ vanishes.
\end{thm}
\begin{proof}  By Proposition \ref{KBarcomparison} we may prove the statement for $\rB_{\bbF_q^d}(M)$.   Since the direct sum functor is symmetric monoidal, by Proposition \ref{symmonoidal} have that $\rB_{\bbF_q^d}(\Sigma^{V} M) \iso \Sigma^{V} \rB_{\bbF_q^d}(M)$.  Thus to prove the statement for $M$ it suffices to  prove  that  $\rB_{\bbF_q^d}(\Sigma^V M)$  is exact for $\dim V \gg 0$ and $d \gg 0$.  
Theorem $1.2$ of Nagpal \cite{nagpal2019vi} states that if $q$ is invertible in $k$ there is $V \in \VI_q$ such that $\Sigma^V M$ admits a finite filtration whose associated graded modules are induced.  The spectral sequence for this filtration reduces us to the case where $M$ is an induced module.   Recall that $M$ is \emph{induced} if there is a sequence of $\GL_n(\bbF_q)$ representations,  $(W_n)_{n \in \bbN}$ with $W_n  = 0$ for $n \gg 0$  such that $$M \iso \bigoplus_{n \in \bbN}  W_n  \otimes_{\GL_n(\bbF_q)} \bP(\bbF_q^n).$$  We have that $$\rB_{\bbF_q^d}(M) \iso \bigoplus_n W_n\otimes_{\GL_n(\bbF_q)} \rB_{\bbF_q^d}(\bP(\bbF_q^n)).$$ For every $V \in \VI_q$  the $k[\GL_n(\bbF_q)]$ module $\bP(\bbF_q^n)_V$ is flat, because $\GL_n(\bbF_q)$ acts freely on the set of linear injections $\VI(\bbF_q^n, V)$.  Thus $\rB_{\bbF_q^d}(\bP(\bbF_q^n))$ is a complex of flat $\GL_n(\bbF_q)$ modules, and so it suffices to prove that the homology of $\rB_{\bbF_q^d}(\bP(\bbF_q^n))$ vanishes for $d > n$. 

To complete the proof, we verify the hypothesis of Theorem \ref{generalprojprop} for $(\VI_q, \oplus)$.    Indeed we have that for $U,V_1, V_2 \in \VI_q$ $$ \VI_q(U, V_1 \oplus V_2) = \bigsqcup_{A_1, A_2 \subseteq U,~ A_1 \cap A_2 = 0}  \VI_q(U/A_1, V_1) \times  \VI_q(U/A_2, V_2) .$$
Thus $$\Res^{\oplus}(\bP(U)) \iso \bigoplus_{A_1, A_2 \subseteq U,  A_1 \cap A_2 = 0} \bP(U/A_1) \otimes \bP(U/A_2),$$ and $U/A_i \leq U$ in the order on the set of objects of $\VI_q$.    
\end{proof}

The analog for $\FI$ modules has an identical proof, using the identity $$\FI(x, a \sqcup b) = \bigsqcup_{x = r \sqcup s} \FI(r,a) \times \FI(s,b),$$ to apply Theorem \ref{generalprojprop} and Nagpal's shift theorem for $\FI$ modules  \cite[Theorem A]{nagpal2015fi}.

\begin{thm}\label{FI}
 Let $k$ be a commutative noetherian ring.   Let $M$ be a finitely generated $\FI$ module over $K$.  Then there exist $D \geq 0$ and $n \geq 0$   such that  for every $d \geq D$ the complex $(\rK_{[d]}(M))_x$ is exact for all $x \in \FI$  of size $\geq n$. 
\end{thm}

\begin{rmk}		Theorem \ref{VIq}  is a categorification of the result, due to Nagpal, that if $M$ is a finitely generated $\VI_q$ module in non-describing characteristic then  $h_M(t)$ is rational with denominator $\prod_{i = 0}^{d-1} (1 - q^{d} t)$  for some $d > 0$.   However, when $k$ has the same characteristic as $\bbF_q$, this is false.  For instance, consider the case $k = \bbF_q$.  Then if $M$ is the $\VI_q$ module defined by the identity functor $V \mapsto V$, we have $$h_M(t) = \sum_{i}  i t^i = \frac{t}{(1-t)^2}.$$  Therefore  Theorem \ref{VIq} cannot extend to representations of $\VI_q$ in equal characteristic.    However, in analogy with Theorem \ref{FSop}, we make the following conjecture.
\end{rmk}

\begin{conj}\label{VIconjecture}
Let $k$ be an arbitrary field.   Let  $M$ be a $\VI_q$ module which is a subquotient of a module generated in degree $d$.  Then there exists an $r \in \bbN$ such that the homology of $(\rK_{\bbF_q^{d+1}})^{\circ r}(M)$  vanishes in sufficiently large degrees.
\end{conj}

\section{Extension to EI categories}
\label{extended}

In this section, we explain how the complexes $\rB_d$ can be interpreted using $\cC$  module homology.   This allows us to define the complexes $\rB_d$ more generally  when $\cC$ is a monoidal EI category.   The material in this section is not used elsewhere in the paper, and it is also more technical.  We include it since it may be of interest to some readers,  but others may prefer to skip it.

\subsection{Homology of EI categories}
First we recall the definition of the homology of EI categories, paralleling the discussion for $\FI$ module homology in \cite[\S5]{church2017homology}.  
\begin{defn}
		An \emph{EI category} is a category in which every endomorphism is an isomorphism. If $\cD$ is an EI category, then the set of isomorphism classes of $\cD$ carries a natural partial order: define $[d_1] \leq [d_2]$ if there is a morphism $d_1 \to d_2$ in $\cD$.  
\end{defn}

For the remainder of this section, $\cD$ denotes a small EI category.  

\begin{defn}	
 Let $M: \cD \to \Mod k$ be a $\cD$ module. Define $$H_0^{\cD}(M)_d := \frac{M_{d}}{\{M_f(m) ~|~  f: c \to d,~ [c] < [d],~  m \in  M_c\}},$$ the left derived functors of $H_0^{\cD}(-)_d$  are denoted by $H_i^{\cD}(-)_d$ and called the $\cD$ module homology in degree $d$.
\end{defn}

The functor $H_0^{\cD}(-)_d$  is naturally isomorphic to $S(d) \otimes_{\cD} -$  for the following $\cD \op$ module.  

\begin{defn}
Let $d \in \cD$.  We let $S(d) : \cD\op \to \Mod k$  be the $\cD \op$ representation  $$c \mapsto \frac{ k \cD(c,d)}{k\{f ~|~ f:c  \to d ,~ [d] < [c]\}}.$$ Concretely it takes $c$ to $k\cD(c,d)$ if $c$ is isomorphic to $d$,   and zero otherwise.  
\end{defn}



For any category $\cC$,  using  $H_0^\cD(-)_d: \Mod \cD \to \Mod k$ we may construct a functor $\id \times H_0^\cD(-)_d$ $$ \Mod (\cC \times \cD) = \Mod(\cD)^{\cC} \longrightarrow^{H_0^{\cD}(-)} (\Mod  k)^\cC  = \Mod {\cC}.$$ 
The following proposition interprets $\rB_d$ in terms of $\cD$ module homology.

\begin{prop}\label{FIhomologyinterpretation}
		Let $\cD$ be a monoidal EI category, and $d \in D$. If $\cD /d$ is equivalent to a poset, then $M \mapsto \rB_d(M)$ is a model for the derived functor $$(\rL(\id \times H_0^{\cD}(-)_d ) )\circ \Res^{\oplus}: \bD^-(\Mod \cD) \to \bD^-(\Mod \cD).$$    Concretely if $\rC_d^{\cD}(-)$ is a functorial complex computing $\cD$ module homology in degree $d$,  coming from a projective $\cD\op$ resolution of $S(d)$,  then there is a quasi-isomorphism $$(\id \times \rC_d ^{\cD}(-)) \circ \Res^\oplus (M) \simeq  \rB_d(M)$$ for any $\cD$ module $M$.
\end{prop}  

\begin{remark}
		The reason we must derive the product functor $\id \times H_0^{\cD}(-)_d$ in the statement of Proposition \ref{FIhomologyinterpretation}  instead of using the functor  $$\id \times  \rL H_0^{\cD}(-)_d : \bD(\Mod \cD)^{\cD} \to \bD(\Mod(k))^\cD,$$ is that for ordinary  (unenhanced)  derived categories there is no equivalence $\bD(\Mod \cD) \simeq \bD(\Mod(k))^\cD$.
However, if $N$ is a $\cD \times \cD$ module such that the $\cD$ module  $N_{(c,-)}$ is  $\rL H_0^{\cD}(-)_d$  acyclic for all $c \in \cD$  then  $N$ is $\rL(\id \times H_0^{\cD}(-)_d)$ acyclic.  So we may compute $\rL(\id \times H_0^{\cD}(-)_d) $  exactly how we would compute  $\id \times  \rL H_0^{\cD}(-)_d$.

We also note that from a more sophisticated point of view,  Proposition \ref{FIhomologyinterpretation} follows immediately by interpreting  $\cD$ module homology as the homology of $$ \cone(\left(\rL j_! j^* M)_d \to M_d \right),$$  where $j: \cD_{< d} \to \cD$ is the inclusion of the subcategory spanned by objects   $< d$ and $j^*, j_!$ are the restriction and left Kan extension functors. Then  $(\rL j_! j^* M)_d$ can be computed as a homotopy colimit over $\cD_{<d}/d$,  which is modelled by a bar construction.
\end{remark}

Before proving  Proposition \ref{FIhomologyinterpretation} we introduce some definitions.  

\subsection{General Bar complexes}

Let $\cJ$ be a small category containing a terminal object,  and such that there are no  morphisms in $\cJ$ from a terminal object to a non-terminal object.

\begin{defn}	\label{S(hat1)}
Define $S(\hat 1_\cJ)$ to be the $\cJ\op$ module,  $$S(\hat 1)_x  = \begin{cases} k &\text{if } x \text{ is terminal }  \\ 0  &\text{otherwise} \end{cases}.$$  Where for $f \in \cJ$,  $S(\hat 1)_{f}$ is $\id_{k}$ if $f$ is a morphism between two terminal objects, and the $0$ map otherwise.  
\end{defn}

Write $\rO\rb(\cJ)$ for the subcategory of $\cJ$ consisting of the set of objects of $\cJ$ and identity maps.  Note $\rO\rb(\cJ)   = \rO\rb(\cJ\op)$.   
Via the adjunction  $$(\Mod k)^{\rO\rb(\cJ)}  \adjoint{\rm Free}{ \Res } (\Mod k)^{\cJ\op},$$ we may resolve  $S(\hat 1_\cJ)$ using the associated monad $\bot = {\rm Free} \circ {\rm Res}$ as in  \cite[\S8.6]{weibel1995introduction}.  We obtain a simplicial abelian $\cJ\op$ module $F$ whose $r$ simplices are $$x \in \cJ \op \mapsto k \{ x \to x_r \to x_{r-1} \to \dots  \to^{f_1} x_0 ~|~ x_0 \text { is terminal}\}.$$
Let $\tilde F$ be the normalized chain complex associated to $F$.  Then $\tilde F \to S(\hat 1)$  is a projective resolution, and so the derived functor of $S(\hat 1_{\cJ}) \otimes_{\cJ} -$  is modelled by $\tilde F \otimes_{\cJ} -$.   For any $\cJ$ module $M$,  we have that $\tilde F \otimes_{\cD} M$  is a chain complex, in degree $r$ it is $$(\tilde F \otimes_{\cD} M)_r = \bigoplus_{x_{r} \to^{f_r} x_{r-1} \to \dots \to x_0 ~|~  x_0 \text{ terminal, } f_i \neq \id_{x_i}}  M_{x_r}$$

\begin{defn}We define $\rB\ra\rr_{\cJ}(M) :=  (\tilde F \otimes_{\cD} M)$.\end{defn}  It is clear that when $\cJ$ is a poset with top element $\hat 1$, we have  $\rB\ra\rr_{\cJ}(M)  \iso \rB_{\cJ}(M)$.  As a warning, note that if there is an equivalence of categories  $f: \cJ' \simto \cJ$,   we do not have $\rB\ra\rr_{\cJ}(M) \iso \rB\ra\rr_{\cJ'}(f^*M)$.  However, because there is an isomorphism of $(\cJ')\op$ modules $f^* S(\hat 1_\cJ) \iso S(\hat 1_{\cJ'})$, both complexes compute the same derived functor, hence there is a quasi-isomorphism  $$\rB\ra\rr_{\cJ}(M) \simeq \rB\ra\rr_{\cJ'}(f^*M).$$

\begin{defn}Using $\rB\ra\rr_{\cJ}$  we may extend our construction of chain complexes to any small monoidal category $\cC$ satisfying the retraction-isomorphism property:
\begin{itemize}
	\item[$(\bdot)$] if $c \in \cC$ and  $c \to^f c' \to^g c$ is a retraction (i.e. $g \circ f = \id$), then $f$ is an isomorphism,
\end{itemize}
by defining  $$\rB\ra\rr_{d} (N) := (\id \times \rB\ra\rr_{\cC/d}(-))\circ \Res^\oplus (N)$$ for any $C$ module $N$.   Indeed, the over-category  $\cC/c$ always contains $\id_{c}: c \to c$ as a terminal object, and the condition $(\bdot)$ corresponds to the statement that any morphism from $\id_c$ to $c' \to^g c$ in $\cC/c$  is an isomorphism. 
\end{defn} 

When the over-categories of $\cC$ are posetal,  the discussion above implies that there is a quasi-isomorphism of complexes of $\cC$ modules  $\rB\ra\rr_{d}(N) \simeq \rB_d(N)$ for every $d \in \cC$.  

\subsection{Proof of Proposition \ref {FIhomologyinterpretation}}
With the general definitions of Bar complexes in hand,  we may prove  Proposition \ref {FIhomologyinterpretation}.  In fact we we may weaken our assumption that the overcategories of $\cD$ are posets, to the assumption that $\cD$ satsifies property $(\bdot)$.  
\begin{proof}[Proof of Proposition \ref {FIhomologyinterpretation}]
Let  $F: \Ch^-(\Mod( \cD \times \cD)) \to \Ch^-( \Mod \cD)$ be an exact sum-preserving functor from bounded below chain complexes $\cD \times \cD$ modules to  bounded below chain complexes of $\cD$ modules.  To show that $F$ models $\rL(\id \times H_0^{\cD}(-)_d)$ it suffices to prove that there is a natural isomorphism $H_0(F(N)) \iso \id \times H_0^{\cD}(-)_d$  for all $\cD \times \cD$ modules $N$,  and that for every $c_1, c_2 \in \cD$ we have that $H_i(F (k(\cD(c_1, -) \times \cD(c_2,-))) = 0$ for $i> 0$.    This follows from the fact that $\{k(\cD(c_1, -) \times \cD(c_2,-))\}_{c_1, c_2 \in \cD}$ form a collection of projective generators of $\Mod(\cD \times \cD)$, by the Yoneda lemma.

The equivalence between the abstract statement of the proposition and the concrete one follows from the fact that if $P_\bdot \to S(d)$ is a projective resolution of $S(d)$  and we define $C(-) := P_\bdot \otimes_{\cD} -$,  then $F:= \id \times C(-): \Mod(\cD \times \cD) \to \Mod(\cD)$  satisfies the above criteria.  

 Let $r: \cD/d \to \cD$ be the functor that forgets the morphism to $d$.  For any $\cD$ module $M$ there is a natural isomophism $$S(\hat 1) \otimes_{\cD/d} r^*M \iso \frac{M_d}{\{M_{g}(m)~|~ g: c \to d, m \in M_{c}  \text{ $g$ not an isomorphism} \}} \iso H_0^{\cD}(M)_d.$$
Further we have that $$r^*(k \cD(c,-)) = \bigoplus_{f: c \to d}  k(\cD/d)(f, -).$$ 
Because ${\rm Bar}_{\cD/d}(-)$ models $S(\hat 1) \otimes_{\cD}^{\rL} -$ and $k(\cD/d)(f, -)$ is a projective $\cD/d$ module by Yoneda,   it follows that $H_i({\rm Bar}_{\cD/d} (r^* (k \cD(c,-)) = 0$ for all $i > 0$.     
Therefore  $\id \times {\rm Bar}_{\cD/d}(-)$ satisfies the criteria to model $\rL(\id \times H_0^{\cD}(-)_d)$.  When $\cD/d$ is equivalent to a poset $Q$,  the quasi-isomorphism between ${\rm Bar}_{\cD/d}$ and ${\rm B}_{Q}$  finishes the proof of the proposition.
\end{proof}

\begin{remark}
Using the material in this section, we may extend the definition of the complexes $\rB_d(-)$ to an arbitrary small monoidal EI--category $\cD$,  by taking Proposition \ref{FIhomologyinterpretation} as a definition.   It is natural to ask: which results from \S\ref{sec:Category}  continue to hold in this level of generality?

  The answer is that Theorem \ref{generalprojprop}  holds for formal reasons.   However, at this level of generality, exactness of complexes does not imply rationality of Hilbert series.   Of course, exactness always encodes \emph{some} relation between the virtual $\Aut(c)$ representations $[M_c]$.  But this relation only implies a relation between the dimensions $\dim M_c$  if for every $d \in \cD$  the over-category  $\cD/d$  is \emph{directed}:  every endomorphism of an object in $\cD/d$ is the identity. (Directedness of $\cD/d$ is equivalent to $\Aut(c)$ acting freely on $\cD(c,d)$ for all $c$).    In fact,  all of \S\ref{sec:Poset} and \S\ref{sec:Category} generalize easily if we replace ``posetal" by ``directed" and ``poset" by ``skeletal directed category."    We chose not to work at that level of generality because posets are more familiar, and all of the examples we consider have posetal over-categories.  
\end{remark}

\section{Gr\"obner theory} \label{sec:Grobner}

  In this section, we use Sam--Snowden's Gr\"obner theory of categories to give a criterion for chain complexes such as $\rK_d(M)$ and $\rB_d(M)$ to be exact. For background see \cite[\S4]{sam2017grobner}.

 Let $(\cC, \sqcup)$ be a small monoidal category.  In this paper, we will only consider the case $\cC = \OSop$  (see Definition \ref{def:OSop}),  but in the interest of future applications work in a more general setting.   
 Fix $d \in \cC$,  and let $\bS(d): \cC \to \Set$ be the functor $c \mapsto \cC(d, c)$.  We assume that $\cC(d,c)$ is finite. 

 \begin{defn} An \emph{ordering} on  $\bS(d)$ consists of a total ordering $\leq$ on the set of functions $\cC(d,c)$ for every in $c$,  such that post-composition is strictly order preserving:    if $f_1 < f_2$  then $g \circ f_1 < g \circ f_2$ for all $g \in \cC(c,c')$.
\end{defn}  

We assume that there is an ordering on $\bS(d)$,   and let $\bP(d) := k \bS(d)$.     For any $c \in \cC$,  the $k$ module $\bP(d)_c$ is freely spanned by monomials $e_f$  for $f\in \cC(d,c)$.   

\begin{defn} The \emph{initial term} of an element $\sum_{f} \lambda_f e_f \in \bP(d)_c$ is $\lambda_{f_0} e_{f_0}$  where $f_0 \in \bP(d)_c$ is the largest element such that $\lambda_{f_0} \neq 0$.  
\end{defn}

\begin{defn} Let $J \subseteq \bP(d)$ be a $\cC$ submodule.  Let $\init(J)_c \subseteq \bP(d)_c$ be the subspace $\{\init(j) ~|~ j \in J_c\}$.
  By \cite[Theorem 4.2.1]{sam2017grobner} $c \mapsto \init(J)_c$ is a \emph{monomial submodule} of $\bP(d)$.  This means that  $\init(J)_c$ is spanned by the elements $e_f \in \bP(d)_c$ that it contains.   We say that $\init(J)$ is the \emph{monomial ideal} associated to $J$.  
\end{defn}

Let $c$ and $x$ be objects of $\cC$.  Let $\cQ$ be a subcategory of $\cC/c$, and let $\cB$ be an exact functor from $\cQ$ representations to chain complexes.     Consider the $\cQ$ representations  $J_{x \sqcup -}$ and $\init(J)_{x\sqcup -}$   defined by $(g: y \to c)\mapsto  J_{x\sqcup y}$ and $(g: y \to c) \mapsto \init(J)_{x \sqcup y}$  respectively.  In this context, we prove the following result.  

\begin{prop} \label{grobner}
If  $\cB(\init(J)_{x \sqcup -})$  is exact,  then $\cB(J_{x \sqcup -})$ is exact.  
\end{prop}
\begin{proof}
Using the order on $\bS(d)$ we define a filtration of the $\cQ$ representation  $\bP(d)_{x \sqcup -}$, as follows.  Given $f : d \to x \sqcup c$,  let $F_{\leq f} \bP(d)_{x \sqcup -}$ be the following subrepresentation:   $$(g: y \to c) \mapsto  {\rm span}\{e_h ~|~  h: d \to x \sqcup y \text{ such that  } (\id_x \sqcup g) \circ h \leq f \}.$$
This is a subrepresentation because a morphism $g \to g'$  consists of a morphism $s: y \to y'$ such that $g' \circ s = g$,  and therefore  $(\id_x \sqcup g' )\circ (\id_x \sqcup s) \circ h = (\id_x \sqcup g) \circ h$.   If $f \leq f'$ then $F_{\leq f} \bP(d)_{x \sqcup -} \subseteq F_{\leq f'} \bP(d)_{x \sqcup -}$, so we obtain a filtration using the total order on $\cC(d, x \sqcup c)$.    
   
We define $F_{\leq f} J_{x \sqcup -}$  to be the intersection of $F_{\leq f} \bP(d)_{x \sqcup -}$ and $J_{x \sqcup -}$.  Then we have:

\begin{lem} The associated graded $\cQ$ representation is isomorphic to $\init(J)_{x \sqcup -}$. 
\end{lem} 
\begin{proof}
Define a map $(F_{\leq f} J_{x \sqcup -})_g \to (\init(J)_{x \sqcup -})_g$ by    $$\sum_h \lambda_h e_h \mapsto \begin{cases}    \lambda_{h_f} e_{h_f} &\text{if there exists $h_f: d \to x \sqcup y$  such that $(\id_x \sqcup g )\circ h_f = f$} \\ 0 & \text{otherwise.} \end{cases}$$   Because $\id_x \sqcup g$ strictly preserves the total order on functions, this yields a well-defined morphism of representations $$\bigoplus_{f \in \bP(d)_{x \sqcup c}} \frac{F_{\leq f} J_{x \sqcup -}}{F_{< f} J_{x \sqcup -}}\to \init(J)_{x \sqcup -}.$$
It is injective because for each $g$, the monomials $e_{h_f}$ are linearly independent for every $f$.    And it is surjective, because each $\lambda_h e_h \in \init(J)_{x \sqcup y}$ is hit by the map corresponding to $f = (\id_x \sqcup g) \circ h$.  
\end{proof}
Thus we have a filtration of $J_{x\sqcup -}$ whose associated graded is isomorphic to $\init(J)_{x\sqcup -}$.   Because  $\cB$ is an exact functor,  this filtration induces a filtration of the complex $\cB(J_{x\sqcup -})$ whose associated graded is $\cB(\init(J)_{x\sqcup -})$.  This filtration is finite, because we have assumed $\bP(d)_{x \sqcup c}$ has finitely many elements.  Therefore the spectral sequence associated to a filtered complex shows that if $\cB(\init(J)_{x\sqcup -})$ is exact,  then so is $\cB(J_{x\sqcup -})$.
\end{proof}

\section{Poset ideals associated to ordered languages}\label{sec:Languages}

Let $\Sigma  = \{a_1, \dots, a_{d}\}$ be a finite alphabet with $d$ letters.  

\begin{defn}Let $w \in \Sigma^*$ be a length $n$ word.   We consider $w$ as a function $w: [n] \to \Sigma$.   Given a set partition $p \in \rP(n)$ with $m$ blocks, order the blocks according to their smallest element to obtain an identification between $p$ and $[m]$.  We write $f_p:[n] \to [m]$ for the surjection that takes $i\in [n]$ to the block that contains it.    
   We say that $w$ \emph{factors through $p$}  if there exists a word $w_p: [m] \to \Sigma$ such that $w_p \circ f_p = w$.  Because $f_p$ is a surjection, $w_p$ is necessarily unique if it exists.  If $w$ factors through $p$, then we call $w_p$ the \emph{quotient word of $w$ by $p$}. 
\end{defn}

 Let $L \subseteq \Sigma^*$  be a regular language. Throughout we assume that $L$ satisfies the following property. \begin{itemize}\item[($*$)] Let $w_1,w_2,w_3 \in \Sigma^*$ and $a \in \Sigma$.  If $w_1 a w_2 w_3 \in L$,   then  $w_1 a w_2 a w_3 \in L$.   \end{itemize}

\begin{defn}We define an upward  ideal, $\bI(w,L) \subseteq \rP(n)$ associated to $w \in \Sigma^*$ and $L$ a language satisfying property $(*)$.   A partition $p$  is contained in $\bI(w,L)$ if and only if $w$ factors through $p$ and $w_p \in L$.  

\end{defn}
The subset $\bI(w,L)$ is upward-closed because of our assumption ($*$) on $L$.  If $w \not \in L$, then $\bI(w,L) = \emptyset$.  

\begin{ex}
Let $\Sigma = \{a,b\}$   let $w = abba.$  Then $w$ factors through the partitions $1|2|3|4$, $14|2|3$, $1|23|4$ and $14|23$.   The associated quotient words are $abba, abb, aba,$ and $ab$  respectively.  If $L$ is the language defined by the regular expression $ab^*a(a^* b^*)^*$, then $L$ satsifies property $(*)$  and we have that $I(w,L) = \{1|2|3|4, 1|23|4\}$.  

\end{ex}


Recall from \cite[\S 5.2]{sam2017grobner} that an \emph{ordered DFA} is a deterministic finite automaton $A$, together with a partial order on the set of states of $A$,  such that $x_1 \leq x_2$ if and only if there is a word $u \in \Sigma^*$ such that running $A$ with start state $x_1$ on input $u$  results in the final state $x_2$; symbolically we will write $x_1 \to^u x_2$.    A regular language $L$ is \emph{ordered} if and only if there is an ordered DFA that accepts it.  All the DFA's we consider will be \emph{connected}:  if $\alpha$ is the start state of $A$ and $\beta$ is some other state of $A$, then there is some word $u \in \Sigma^*$ such that $\alpha \to^u \beta$.  Any (ordered) DFA may be replaced by a connected (ordered) DFA accepting the same language, by discarding the states which are not connected to the start state.   It is convenient to introduce the following operation, in order to make inductive arguments.

\begin{defn}
	Let $A$ be an ordered DFA and let $\alpha$ be a state of $A$. Then we define the \emph{truncation of $A$ at $\alpha$},  denoted $A_{\geq \alpha}$  to be the following ordered DFA.   The (accept) states of $A_{\geq \alpha}$ are the (accept) states of $A$ that are $\geq \alpha$ in the partial order.  The transition function of $A_{\geq \alpha}$ is the restriction of the transition function of $A$ to $A_{\geq \alpha}$:  this is well defined  because $A$ is ordered.  The start state of $A_{\geq \alpha}$ is $\alpha$.
\end{defn}

Let $A$ be an ordered DFA accepting a language $L \subseteq \Sigma^*$, which satisifies property $(*)$.    Let $w_1, \dots, w_r \in \Sigma^*$ be words of length $\ell_1, \dots, \ell_r \in \bbN$, and put $w = w_1 \dots w_r$. Let $\ell = \sum_{t} \ell_t$ be the length of $w$.

There is an embedding of posets $\prod_{t = 1}^r \rP(\ell_t) \to  \rP(\ell)$,  constructed using the bijection $[\ell_1] \sqcup \dots \sqcup [\ell_r] \to [\ell]$  which maps $[\ell_t]$ to the interval $( \sum_{u = 1}^{t-1} \ell_u, \sum_{u = 1}^{t} \ell_u ]$.  Given partitions $q_t \in \rP(\ell_t)$  for $t = 1, \dots, r$  the associated element of $\rP(\ell)$ is $q_{1} \sqcup \dots \sqcup q_r$,  considered as a partition of $[\ell]$ via the bijection. Then $\prod_{t = 1}^r \rP(\ell_t)$ is isomorphic to its image in $\rP(\ell)$ (with subposet structure induced by $\rP(\ell)$),  so we identify both posets and write $\prod_{t = 1}^r \rP(\ell_t) \subseteq \rP(\ell)$.  


 \begin{defn}Let $\bJ(w,L)$ be the ideal  $\bI(w,L) \cap \prod_{t = 1}^r \rP(\ell_t)$.  \end{defn}

The key combinatorial result underlying Theorem \ref{FSop} is the following.

 \begin{thm}\label{maintechnical}\label{Languages}
Let $d \geq 1$ and let $\Sigma$ be an alphabet of size $d$.  Let $A$ be a connected, finite, ordered DFA, accepting a language $L \subseteq \Sigma^*$ which satsifies property $(*)$.    Let $w_1, \dots, w_r \in \Sigma^*$ be words of length $\ell_1, \dots, \ell_r$.  
Suppose that $r$ is greater than or equal to the length of $A$ (considered as a poset), $\ell_t \geq d$ for $t = 1, \dots, r-1$ and $\ell_r \geq d+1$.  Then $\rB_{(P(\ell_1), \dots, P(\ell_r))}(k\bJ(w,L))$ is exact.   
\end{thm}
\begin{proof}
		We write $\cB_{i}$ for the functor $\rB_{(P(\ell_i), \dots, P(\ell_r))}$  from $\Rep \prod_{t = i}^r \rP(\ell_t)$ to chain complexes.  

		We  induct on the length of $A$.  In the base case, the length of $A$ is $1$  and $A$ consists of a single state.  If the state is a reject state, then $k\bJ(w,L) = 0$  so the statement is trivially true, so assume that the state is an accept state.   Then $A$ accepts every word and $\bJ(w,L)$ is the ideal consisting of all $q \in \prod_{i =t}^\ell \rP(\ell_t)$  such that $w$ factors through $q$.  
Let $p_t$ be the partition of $[\ell_t]$ associated to the function $w_t: [\ell_t] \to \Sigma$ (in other words, $i$ and $j$ are in the same block of $p_t$ if and only if the $i^{\rm th}$ letter of $w_t$ equals the $j^{\rm th}$ letter of $w_t$).   Then $
\bJ(w,l) = \prod_{t = 1}^r \rP(\ell_t)_{\geq p_t}$. so  $k\bJ(w,l) =k(\prod_{t = 1}^{r-1} \rP(\ell_t)_{\geq p_t})  \sqtimes k \rP(\ell_r)_{\geq p_r}$. Because $\ell_r > d$ and $|\Sigma| = d$,  the Pigeon-hole principle implies at least one letter of $\Sigma$ must be repeated in $w_r$.  Therefore $p_r \neq \hat 1_{\rP(\ell_r)}$  and so  by Proposition \ref{minimalelt}  the complex  $\rB(k \rP(\ell_r)_{\geq p_r})$ is exact.  Then by Proposition \ref{product},  $\cB_1(k\bJ(w,l))$ is exact.  

    For the inductive step assume that the length of $A$ is $\geq 2$, and let $\alpha$ be the starting state of $A$.   Consider the word $w_1$.  There are two cases:  either (1) $\alpha \to^{w_1} \alpha$,  or (2) $\alpha \to^{w_1} \beta$  for some state $\beta > \alpha$.    

	 The argument in case (1) is similar to the base case.   Let $p_1 \in \rP(n)$ be the partition associated to the function $w_1: [\ell_1] \to \Sigma$.  We claim that $$\bJ(w,L) = \rP(\ell_1)_{\geq p_1} \times \bJ(w_2 \dots w_r, L).$$
 Indeed, we may write an element of $q \in \rP(\ell_1) \times  \prod_{t=2}^r\rP(\ell_t)$ uniquely as $q_1\sqcup q'$ where $q_1 \in \rP(\ell_1)$ and $q' \in \prod_{t=2}^r\rP(\ell_t)$.  Then $w$ factors through $q$ if and only if $w_1$ factors through $q_1$ and $w':= w_2 \dots w_r$ factors through $q'$.    If  $w$ factors through $q$, then $w_q = ({w_1})_{q_1}   w'_{q'}$.  Because $({w_1})_{q_1}$ contains only letters which appear in $w_1$, and $\alpha \to^t \alpha$ for every letter of $w$,  we have that $\alpha \to^{({w_1})_{q_1}} \alpha$.   Therefore $A$ accepts $w_q$ if and only if $A$ accepts $w'_{q'}$.  Because $w_1$ factors through $q_1$ if and only if $q_1 \geq p_1$,  we have that $q \in \bJ(w,L)$ if and only if $q_1 \in \rP(n)_{\geq p_1}$ and $q' \in \bJ(w', L)$.   This establishes the claim.

So we have  $$k \bJ(w,L) = k \rP(\ell_1)_{\geq p_1} \sqtimes k\bJ(w_2 \dots w_r, L).$$  
The word $w_1$  only contains letters $a \in \Sigma$ which satsify $\alpha \to^{a} \alpha$.  Since $\alpha$ is not the only state of $A$, there are at most $d-1$ such letters.   Thus by the Pigeon-hole principle it follows that $w_1$ must contain a repeat letter and therefore $p_1 \neq \hat 1_{\rP(\ell_1)}$.    So Proposition \ref{minimalelt}  implies that  $\rB(k \rP(\ell_1)_{\geq p_1})$ is exact.   Therefore by Proposition \ref{product},  $\cB_1(k\bJ(w,L))$  is exact in case (1).

In case (2),  let $q_1 \in \rP(\ell_1)$, and consider the ideal of $\prod_{t = 2}^r \rP(\ell_t)$   $$I_{q_1} := \bJ(w,L) \cap \left(q_1 \times \prod_{t = 2}^r \rP(\ell_t)\right).$$   We note that $$k I_{q_1}  =  (k \bJ(w,L))|_{{q_1} \times \prod_{t = 2}^r \rP(\ell_t)},$$ so by Proposition \ref{filtration}, it suffices to show that the homology $\cB_2(k I_{q_1})$ vanishes for all $q_1$.  If $w_1$ does not factor through $q_1$, then $I_{q_1} = \emptyset$, and $kI_{q_1} = 0$ so $\cB_2(kI_{q_1}) = 0$.   So suppose that $w_1$ factors through $q_1$, and let $(w_1)_{q_1}$ be the quotient word.    Notice that $\alpha \to^{({w_1})_{q_1}} \beta'$ for some state $\beta' > \alpha$.  Indeed, if $a$ is the first letter of $w_1$  which causes $A$ to transition to a state distinct from $\alpha$,  then $(w_1)_{q_1}$ also contains the letter $a$, and so also transitions to a state distinct from $\alpha$.  

 Now let $A':= A_{\geq \beta'}$ be the truncation at $\beta'$.  Let $L'$ be the language recognized by $A'$.  We claim that $$I_{q_1} = \bJ(w', L'),$$ where $w' = w_2 \dots w_r$.    Indeed since $w_1$ factors through $q_1$, we have that for all $q' \in \prod_{t = 2}^{r} \rP(\ell_t)$,  $w$ factors through $q_1 \sqcup q'$ if and only if $w'$ factors through $q'$.  If $w'$ factors through $q'$,  then $w_{q_1 \sqcup q'} = ({w_1})_{q_1} w'_{q'}$.  So $w_{q_1 \sqcup q'}$ is accepted by $A$ if and only if $\alpha \to^{{w_1}_{q_1}}  \beta' \to^{w'_{q'}}  \gamma$, where $\gamma$ is an accept state of $A$.   This occurs if and only if $w'_{q'} \in L'$.  This shows the claim.

The length of $A'$ is less than the length of $A$, hence is $\leq r-1$.   The language $L'$ satsifies property $(*)$ because $L$ does and $u \in L'$ if and only if $w_1 u \in L$.  So the induction hypothesis gives that the homology of $\cB_2(k \bJ(w',L'))$ vanishes.  Thus the homology of $\cB_2(k I_{q_1})$ vanishes for all $q_1 \in \rP(\ell_1)$,  so we are done.  
\end{proof}

\section{$\OSop$ modules and Theorem \ref{FSop}} \label{sec:ThmProof}

In this section, we combine the results of the previous sections in order to prove Theorem \ref{FSop}.

\begin{defn}\label{def:OSop} Let $\OS$ be the category of \emph{ordered surjections} of \cite[\S8]{sam2017grobner}. The objects of $\OS$ are finite sets equipped with a total order.   A morphism from $X$ to $Y$ in $\OS$ is an \emph{ordered surjection}:  a surjection $f: X \onto Y$  such that if $y_1 < y_2$  then $\min f\inv(y_1)  < \min f\inv(y_2)$.  
\end{defn}

 $\FSop$  is equivalent to the subcategory spanned by the sets $[n] := \{1, \dots, n\}$ for $n \in \bbN$.  Similarly $\OSop$ is equivalent to the subcategory spanned by the sets $[n]$ equipped with their canonical order $1 \leq 2\leq \dots \leq n$.  A functor with domain $\FSop$  or $\OSop$ determined by its restriction to the corresponding  subcategory,  so we will freely pass between the ambient categories and their respective subcategories.  When we use a natural number $n \in \bbN$ to denote an object of $\OSop$ or $\FSop$,  we are referring to the object $[n]$.  

Disjoint union endows $\OSop$  with a monoidal structure  $$- \sqcup - :\OSop\times \OSop \to \OSop,$$  where we define the total order on $S \sqcup T$ by extending the order on $S$ and $T$  and declaring every element of $S$ to be less than every element on $T$.  
 There is an essentially surjective, monoidal functor  $\pi:\OSop \to \FSop$, given by forgetting the total ordering.  Thus we may restrict any $\FSop$ module $M$ to obtain an $\OSop$ module $\pi^* M$.  We now prove two propositions relating $\OSop$ and $\FSop$ modules.

\begin{prop}\label{agrees1}

    For $X \in \OSop$, the functor $\OS\op/X \to  \FSop/X$ is an isomorphism of categories.   Consequently, $\rK_X$ and $\rB_X$ are defined for $\OSop$ modules, and 
 for any $\FSop$ module $M$  there are natural isomorphisms  $\rB_X(\pi^*M) \iso \pi^*\rB_X (M)$ and $\rK_X(\pi^*M) \iso \pi^* \rK_X(M)$.
\end{prop}
\begin{proof}
Let $X \in \OSop$ be a set with a total order.  Given any surjection $f: X \onto Y$  there is a unique total order on $Y$ such that $f$ is an ordered surjection: define $y \leq y'$ if $\min f\inv(y) \leq \min f\inv(y')$.   And given any $X \onto Y_1 \onto Y_2$, the orders induced on $Y_1$ and $Y_2$ are compatible so that $Y_1 \onto Y_2$ is an ordered surjection.  Thus $\OS\op/X \to  \FSop/X$ is an isomorphism and the other statements follow immediately.
\end{proof}

\begin{prop}\label{agrees2}
	Let $N$ be an $\FSop$ module generated in degree $\leq d$.  Then $\pi^*N$ is generated in degree $\leq d$. 
\end{prop}
\begin{proof}
		 Let $l \in \bbN$ and let $x \in N_{l}$ be an element.  Then for any $m \in \bbN$, any surjection $f: [m] \onto [l]$ factors as an ordered surjection, followed by a bijection $[l] \to [l]$.  Thus the $\FSop$ submodule generated by $x$ equals the $\OSop$ module generated by $\{\sigma x\}_{\sigma \in \bS_l}$.  
\end{proof}

For $d \in \bbN$, let $\bP(d)$ be the free $\OSop$ module generated in degree $d$:  $\bP(d)_m := k \OSop(d,m)$.  The next proposition applies Theorem \ref{Languages} and Proposition \ref{grobner}  to prove a variant of Theorem \ref{FSop} for submodules of $\bP(d)$.

\begin{prop}\label{OSopcase}
	 Let $k$ be a field.  Let $J$ be an $\OSop$ submodule of $\bP(d)$.  Then there exists an $s \in \bbN$ such that $\rK_{\ell_1} \circ \dots \circ \rK_{\ell_r}(J)$ is exact for all $r \geq s$ and all $(\ell_t)_{t = 1}^r \in \bbN^{r}$ satisfying $\ell_t \geq d$  for $t = 1, \dots, r-1$ and $\ell_r \geq d+1$ .  
\end{prop}
\begin{proof}

By \cite[Theorem 8.1.1]{sam2017grobner} the category $\OSop$ is \emph{Gr\"obner}: in particular the functor $\bS(d)$,  defined by $n \mapsto \OSop(d,n)$, carries an ordering satisfying the conditions of \S \ref{sec:Grobner}. (We will not use the particular definition of the ordering, so we do not recall it).  We let $\init(J)$ be the monomial ideal associated to $J$ with respect to this ordering.    

For $n \in \bbN$,  let $T_n \subseteq \init(J)_n$ be the set of $f \in \bS(d)_n$ such that $e_f \in \init(J)$.   Then $n \mapsto T_n$ is functor $\OSop \to \Set$, and there is a canonical morphism $k T \to \init(J)_n$  given by $e_f \mapsto e_f$.  This map is injective for any ring,  and because $k$ is a field it is surjective: if $\lambda e_f \in \init(J)$ is a nonzero element, then $e_f = \lambda\inv \lambda e_f \in \init(J)$.

Let $\Sigma = [d]$  and $\iota: \bigsqcup_{n \in \bbN} \OSop(n,d) \to \Sigma^*$   be the tautological embedding taking  $f: [n] \to [d]$ to its word.   In the terminology of Sam--Snowden \cite{sam2017grobner} $$\bigsqcup_{n \in \bbN} \OSop(n,d) = |\bS(d)| = |\cC_{[d]}|,$$  where $\cC = \OSop$.  In \cite[\S8.2]{sam2017grobner},  Sam--Snowden show that the embedding $\iota: |\bS(d)| \to \Sigma^*$ endows $\bS(d)$ with an \emph{O-lingual structure} \cite[Definition 6.2.1]{sam2017grobner}.  In particular this implies there is an ordered DFA $A$,  which accepts the language $L := \iota(\sqcup_{n} T_n) \subseteq \Sigma^*$  (we take $A$ to be connected).  We will prove the statement for $s$ equal to the length of $A$, considered as a poset. So fix $r \geq s$.

There is a quasi-isomorphism $\rK_{\ell_1} \circ \dots \circ \rK_{\ell_r} \simeq \rB_{\ell_1} \circ \dots \circ \rB_{\ell_r}$ by Proposition \ref{KBarcomparison}.   Fix $x \in \OSop$.  As discussed in \S\ref{subsec:prodcomposite}, it suffices show that $\rB_{(\OS\op/\ell_1, \dots ,\OSop/\ell_r)}(J_{x \sqcup -})$ is exact.  To do this, we identify $\rP(\ell_t)$ with the full subcategory of $\OSop/\ell_t$  spanned by the canonical ordered surjections $[\ell_t] \onto p$,  for every $p \in \rP(\ell_t)$.   Then $$\rB_{(\OS\op/\ell_1, \dots \OSop/\ell_r)}(J_{x \sqcup -}) \iso \rB_{(P(\ell_1), \dots, \rP(\ell_r))}(J_{x \sqcup -}),$$ and by Proposition \ref{grobner} it suffices to show that $$ \rB_{(P(\ell_1), \dots, \rP(\ell_r))}((\init J)_{x \sqcup -}) \iso  \rB_{(P(\ell_1), \dots, \rP(\ell_r))}(k T_{x \sqcup -})$$ is exact.    So we consider the functor $ T_{x \sqcup -}:  \rP(\ell_1) \times \dots \times \rP(\ell_r) \to \Set$    $$(q_1, \dots, q_r) \mapsto T_{x \sqcup q_1 \sqcup \dots \sqcup q_r}.$$   Since $\bS(d)_f$ is injective for all $\OSop$ morphisms $f$,  by Proposition \ref{PSet} we see that $T_{x \sqcup -} =  \bigsqcup_{u}  \rS I \langle u \rangle$ where the sum is over words $u: x \sqcup [\ell_1] \sqcup  \dots \sqcup [\ell_r] \to [d]$.

Fix $u: x \sqcup [\ell_1] \sqcup  \dots \sqcup [\ell_r] \to [d]$  and write $u = z w_1 \dots w_r$, where $z: x \to [d]$ and $w_t$ is a word on $\ell_t$ letters for $t = 1, \dots, r$.   Let $\alpha$ be the state obtained by running $A$ with input $z$   Let $L'$ be the language accepted by $A_{\geq \alpha}$.

Because $T$ is an $\OSop$ set, the languages  $L$ and $L'$ both satisfy property  ($*$) of \S\ref{sec:Languages}:  under the identification between words and functions,  the substitution $w_1 a w_2w_3 \mapsto w_1 a w_2 a w_3$ corresponds to precomposition with the obvious ordered surjection $[n] \onto [n-1]$.   Therefore the ideal $\bJ(w_1 \dots w_r, L') \subset \rP(\ell_1) \times \dots \times \rP(\ell_r)$ is defined.  

\begin{lem}The ideal $I\langle u \rangle$ equals the ideal $\bJ(w_1 \dots w_r, L')$.
\end{lem}
\begin{proof}  Let $\ell = \sum {\ell_t}$.   By definition, a partition $p \in \rP(\ell_1) \times \dots \times \rP(\ell_r)$  lies in $I\langle u\rangle$ if and only if there is an element $u' \in T_{x \sqcup p}$  such that for the the ordered surjection $f_p: [\ell] \onto p$  we have  $T_{\id_x \sqcup f_p}( u') = u$. This occurs if and only if (1)  there is an element of $u' \in \bS(d)_{x \sqcup p}$ satisfying  $\bS(d)_{\id_x \sqcup f}( u')$  and (2)  $u' \in T_{x \sqcup p}$.  These two conditions correspond to (1)   the word $w_1 \dots w_r$ factoring through $p$  and (2) the word $z(w_1 \dots w_r)_p$  lying in $L$. (Here $(w_1 \dots w_r)_p$ is the quotient word).   Since $L'$ is the language accepted by $A_{\geq \alpha}$,   $(2)$ holds if and only if $w_1 \dots w_r \in L'$.  
\end{proof}

Therefore, by Theorem \ref{maintechnical} we have that $\rB_{(\rP(\ell_1), \dots, \rP(\ell_r))}(k I_u)$ is exact.  Summing over $u$, we see that $\rB_{(\rP(\ell_1), \dots, \rP(\ell_r))}(kT_{x \sqcup -})$ is exact, completing the proof.
\end{proof}

The remainder of the Proof of Theorem \ref{FSop} is a d\'evissage.

\begin{proof}[Proof of Theorem \ref{FSop}]
Let $M$ be an $\FSop$ module that is a subquotient of one generated in degree $\leq d$. By Proposition \ref{agrees1} it suffices to show that there is an $s \in \bbN$ such that the complex $$ \rK_{\ell_1} \circ \dots \circ \rK_{\ell_r}( \pi^* M)$$ is exact for all $r \geq s$ and choices of $\ell_t \geq d$,  $t = 1, \dots, r-1$ and $\ell_r \geq d+1$.

  Embed $M$ into an $\FSop$ module $N$ which is finitely generated in degree $\leq d$.   By Proposition \ref{agrees2}, $\pi^* N$ is also generated in degree $\leq d$.     Choose a surjection  $p: G \onto  \pi^* N$  where $G = \bigoplus_{i = 1}^R \bP(n_i)$ for  $\{n_i\}_{i= 1}^R$ a finite list of natural numbers $n_i \leq d$.  We obtain the following commutative diagram, with exact rows.  

\begin{center}
\begin{tikzcd}
 0& \arrow[l] \pi^*N           & G \arrow[l,"p"]                         & \ker p \arrow[l]                   & \arrow[l] 0             \\
 0 &\arrow[l] \pi^*M \arrow[u] &J \arrow[l] \arrow[u] & J' \arrow[l] \arrow[u] & \arrow [l] 0
\end{tikzcd},
\end{center}
where $J = p^{-1}(\pi^* M)$ and $J' = \ker p \cap p^{-1}(\pi^* M)$. If $\cB$ is any exact functor from $\OSop$ modules to chain complexes,  we have that $\cB(\pi^* M)$ is exact if both $\cB(J)$ and $\cB(J')$ are.
Now filter $G$ by defining $F_I G = \bigoplus_{i = 1}^I \bP(n_i)$.  Intersecting with $J$ and $J'$  we obtain filtrations $F_IJ$ and $F_IJ'$.  To show that $\cB(\pi^* M)$ is exact, it suffices to show that the complexes obtained by applying $\cB$ to the associated graded modules $ F_I J/F_{I-1}J$ and $F_I J'/F_{I-1}J'$ are exact.  These associated graded modules are submodules  of $F_IG/F_{I-1} G = \bP(n_I)$.     By Proposition \ref{OSopcase} there exists an $s_I \in \bbN$  such that for all  $r \geq s_I$  and  $(\ell_t)_{t = 1}^{r} \in \bbN^{s}$ with $\ell_t \geq n_I$ and $\ell_{r} \geq n_I + 1$  the complex  $\rK_{(\ell_1, \dots, \ell_{r})} (F_I J/F_{I-1}J)$ is exact.   Identically, there exists an $s_I' \in \bbN$  such that for all $r' \geq s_I'$ and  $(\ell_t')_{t = 1}^{r'} \in \bbN^{s_I'}$ with $\ell_t' \geq n_I$ and $\ell_{r'}' \geq n_I+1$  the complex  $\rK_{(\ell_1, \dots, \ell_{r'})} (F_I J'/F_{I-1}J')$ is exact. Take $s$ to be $\max( \cup_{I} \{s_I, s_I'\})$.  Then  we conclude that $\rK_{(\ell_1, \dots, \ell_{r})}(M)$ is exact for all choices of $r \geq s$ ,  $\ell_t \geq d$ and $\ell_r \geq d+1$.  
\end{proof}


\part{Characters of $\FSop$ modules}

\section{The character of $\rK_d(M)$} \label{sec:CharacterofKd}
 
Let $M$ be a chain complex of $\FSop$ modules over a commutative ring $R$.  The purpose of this section is, when $R$ is a field of characteristic zero, to compute the virtual  $\rS_d \times \rS_n$ representation associated to $\rK_d(M)_n$, in terms of certain differential operators on the ring of symmetric functions.

\begin{defn}
	We write $\rS_\bdot$ to denote the groupoid $\sqcup_{n\in \bbN} \rS_n$.   Any $\FSop$ module $M$ determines a representation of $\rS_\bdot$  by $n \mapsto M_{[n]}$.  
\end{defn}

The category $\Ch(\Mod R)$ of chain complexes of $R$ modules is symmetric monoidal, using the Koszul braiding. Thus there is a symmetric monoidal  induction product $$\oast: \Ch(\Mod R)^{\rS_\bdot  } \times \Ch(\Mod R)^{ \rS_\bdot} \to  \Ch(\Mod R)^{ \rS_\bdot},$$ constructed by tensoring and applying the adjoint to the restriction functor  $$\Res^{\sqcup}: \Ch(\Mod R)^{\rS_\bdot \times \rS_\bdot} \leftarrow  \Ch(\Mod R)^{ \rS_\bdot}.$$  
Concretely,  if $M$ and $N$ are  $\rS_\bdot$ representations,  then  $$(M \oast N)_n = \bigoplus_{i + j = n} \Ind^{\rS_{n}}_{\rS_i \times \rS_{j}} M_i \otimes N_j.$$
Finally there is a composition product $$ \odot: \Ch(\Mod R)^{\rS_\bdot} \times \Ch(\Mod R)^{ \rS_\bdot} \to   \Ch(\Mod R)^{\rS_\bdot}.$$  defined by $$M \odot  N := \bigoplus_{r} M_r \otimes_{\rS_r}  N^{\oast r}$$ where $\rS_r$ acts on $N^{\oast r}$ in the way induced by the Koszul sign rule.    $\QED$  

Our first step will be to describe the underlying representation of $\rK_d(M)$ in terms of these operations and the top Whitney homology of the partition lattice.  

\begin{defn}
	For $b$ a set of size $\geq 1$, let $\cW_{b}$  be the degree $n-1$ homology group of  the pair $(\rN(\rP(b)), Z_{\rP(b)})$ with coefficients in $R$. 
 We consider $\cW_n := \cW_{[n]}$ to be a chain complex concentrated in homological degree $n-1$,  so that $\cW_n \in \Ch(\Mod R)^{\rS_n}$.  We define $\cW_0 = 0$ and $\cW := \bigoplus_{n \in \bbN} \cW_n \in \Ch(\Mod R)^{\rS_\bdot}$.  
\end{defn}

\begin{defn}Let $M$ be a $\rS_\bdot \times \rS_\bdot$ representation.  Then for an $\rS_\bdot$ representation $N$ we write $M \odot_2 N$ to denote the $\rS_\bdot \times \rS_\bdot$ representation, $$\bigoplus_{(n,r), r \geq 1}  M_{(n,r)} \otimes_{\rS_r}  N^{\oast r}$$  where the subscript $2$ indicates that we are performing the composition product with respect to the second factor. 
\end{defn}

\begin{prop}\label{Whitneyiso}
Let $M$ be a chain complex of an $\FSop$ modules.  Then the $\rS_\bdot \times \rS_\bdot$  complexes
$\bigoplus_{d \geq 1} \rK_d(M)$  and $$  \left(\Res^\sqcup   M \right) \odot_{2} \cW,$$  have  isomorphic underlying graded $\rS_\bdot \times \rS_\bdot$ representations.
\end{prop}
\begin{proof}
First we describe $\left(\Res^\sqcup   M \right) \odot_{2} \cW$ more concretely.   We have an isomorphism of $\rS_d$ representations $$(\cW^{\oast r})_d = \bigoplus_{e_1 + \dots + e_r = d, e_i \geq 1}  \Ind_{\rS_{e_1} \times \dots \times \rS_{e_r}} ^{\rS_d}( \cW_{e_1} \otimes \dots \otimes  \cW_{e_r} )\iso  \bigoplus_{[d] = b_1 \sqcup \dots \sqcup b_r}  \cW_{b_1} \otimes \dots \otimes \cW_{b_r},$$
where the second sum is over partitions of $[d]$ into $r$ labelled blocks.  To construct this isomorphism, note that $\rS_d$ acts on the set of such ordered partitions, with orbits in bijection with  $\{(e_1, \dots, e_r) ~|~ e_i \geq 1, \sum_{i} e_i = e\}$  and each orbit contains a distinguished element with stabilizer $\rS_{e_1} \times \dots \times \rS_{e_r}$.   
Now,  $\rS_r$  acts on the right by relabeling the blocks and permuting the tensor factors by the Koszul sign rule.  And we have that $$(\left(\Res^\sqcup   M \right) \odot_{2} \cW)_{n,d} \iso \bigoplus_r  M_{[n] \sqcup [r]} \otimes_{\rS_r} \left( \bigoplus_{[d] = b_1 \sqcup \dots \sqcup b_r}  \cW_{b_1} \otimes \dots \otimes \cW_{b_r} \right).$$

Next, we have that $$\rK_d(M)_n = \bigoplus_{p \in \rP(d)}  M_{[n] \sqcup p} \otimes H_{\rr(p)}(\rN[p, \hat 1], Z_{[p, \hat 1]}; R).$$
Let $p \in \rP(d)$  be a partition with $r$ blocks. There is an isomorphism of posets $$\prod_{b \text{ block of $p$} }  \rP(b) \to [p, \hat 1],$$ given by disjoint union.  Choose an arbitrary labelling of the blocks of $p$ as $b_1, \dots, b_r$, giving a bijection $\sigma_p : p \to [r]$.    Since for any pair of posets $P,Q$ with top and bottom elements, we have that $Z_{P \times Q}  = \rN P \times Z_Q \cup Z_P \times  \rN Q$, and  the Kunneth formula  for pairs gives  \begin{align} {\rm Kun}_{\sigma_p}:~ H_{\rr(p)}(\rN([p, \hat 1]), Z_{[p, \hat 1]}) &\iso  H_{|b_{1}| - 1}( \rN\rP(b_1), Z_{\rP(b_1)}) \otimes \dots \otimes   H_{|b_{r}| - 1}(\rN \rP(b_r), Z_{\rP(b_r)}) \nonumber \\  &= \cW_{b_1} \otimes \dots \otimes \cW_{b_r}. \nonumber \end{align}  Here the coefficients are in $R$,  and we have used the fact that by the universal coefficient theorem $\cW_b$ is a free $R$ module.  For a different ordering of the blocks, the Kunneth isomorphism we obtain differs by a permutation of the tensor factors modified by a sign according to the Koszul sign rule.  Then summing the morphism $M_{\id_{[n] \sqcup \sigma_p}} \otimes {\rm Kun}_{\sigma_p}$ over all partitions  $p$ of $[d]$ with $r$ blocks,  we obtain an isomorphism $$\bigoplus_{p} M_{[n] \sqcup p}  \otimes H_{\rr(p)}(\rN([p, \hat 1]), Z_{[p, \hat 1]}) \iso   M_{[n] \sqcup [r]} \otimes_{\rS_r} \left( \bigoplus_{[d] = b_1 \sqcup \dots \sqcup b_r}  \cW_{b_1} \otimes \dots \otimes \cW_{b_r} \right).$$ This isomorphism does not depend on the choices of $\sigma_p$,  because any two choices differ by an element of $\rS_r$, and in the target we have quotiented by the action of $\rS_r$.  It is compatible with the action of  $\rS_n$ and $\rS_d$, and hence an isomorphism of representations.

 To check the compatibility with $\tau \in \rS_d$,  let $p \in \rP(d)$  and  label the blocks of $p$ and $\tau(p)$  such that $\tau$ maps the $i^{\rm th}$ block of $p$ to the $i^{\rm th}$ block of $\tau(p)$.     Then both pairs of maps $M_{\id_{[n] \sqcup \sigma_p}}, M_{\id_{[n] \sqcup \sigma_{\tau(p)}}}$ and ${\rm Kun}_{\sigma_p}$ ,  ${\rm Kun}_{\sigma_{\tau(p)}}$ are compatible with the action of $\tau$.  
\end{proof}

\subsection{Symmetric functions and Frobenius characters}\label{symmconventions}
We now specialize to $R = \bbQ$.  For the remainder of this paper, we will work with $\FSop$ modules over $\bbQ$.  Since all of the irreducible representations of $\rS_n$ are defined over $\bbQ$, all of our arguments and results carry over to $\FSop$ modules over any field of characteristic zero.

\begin{defn}  Let be $\Lambda$ be the ring of symmetric functions with $\bbQ$ coefficients, $\Lambda_n$ the subspace of symmetric functions of degree $n$, and $\hat \Lambda := \prod_{n \in \bbN}  \Lambda_n$  its completion with respect to the filtration by degree.

We use standard notation for symmetric functions. Let  $\lambda = 1^{m_1}2^{m_2} \dots$ be an integer partition.  We write $p_{\lambda}$ for the power sum monomial  $\prod_{i} p_i^{m_i}$.  Similarly $e_\lambda, h_\lambda$ and $s_\lambda$, denote respectively the elementary, homogeneous, and  Schur symmetric functions corresponding to $\lambda$.  An element of $f \in \hat \Lambda$ admits a unique expansion as a power series with respect to any of these collections of monomials.  For instance we may write $$ f = \sum_{\lambda} a_\lambda p_\lambda$$   for a unique vector  $(a_\lambda) \in \prod_{\lambda \text{ integer partition }} \bbQ.$

We set $\lambda!:= \prod_{i} m_i !$ and  $z_\lambda := \lambda! \prod_{i} (i)^{m_i}$  and $\sgn(\lambda) = \prod_{i  \text{ even } } (-1)^{m_i}$.   We define  $|\lambda| := \sum_{i} i m_i$  and $\rank(\lambda) := \sum_{i} m_i,$  and write $\sigma_\lambda$ to denote an element of the conjugacy class of $\rS_{|\lambda|}$ corresponding to $\lambda$.  Thus $\sgn(\lambda)$ is the trace of $\sigma_\lambda$ on the sign representations of $\rS_{|\lambda|}$.

The Hall inner product on $\Lambda$ is defined by $$\langle p_\lambda/z_\lambda  , p_\nu \rangle = \delta(\lambda,\nu) = \langle s_\lambda, s_\nu \rangle, $$ for $\delta(\lambda,\nu)$ the Kronecker delta.  The Hall inner product extends uniquely to a pairing $\langle -, - \rangle:  \Lambda \times \hat \Lambda  \to \bbQ$,  which induces an isomorphism $\hat \Lambda \to \Lambda^*$, given by $s \in \hat \Lambda \mapsto \langle -, s \rangle$.
    $\QED$ 

 \end{defn}

\begin{defn}
  Let $N$ be a finite dimensional $\rS_{n}$ representation over $\bbQ$.  Then its \emph{Frobenius character}  $\ch(N) \in \Lambda_n$  is defined to be $$\rc\rh(N):=\sum_{\lambda \vdash n}  \Tr(\sigma_{\lambda}, \ch(N)) \frac{p_\lambda}{z_\lambda} = \sum_{\lambda \vdash n} \mult_{\lambda}(N) s_\lambda.$$
Let $M$ be a finite dimensional representation of $\rS_{n_1} \times   \dots \times \rS_{n_r}$.  Then $$\ch(M) \in \Lambda_{n_1} \otimes \dots \otimes \Lambda_{n_r},$$ is uniquely defined by setting $\ch(\otimes_{i = 1}^r (N_i)) = \otimes_{i = 1}^r \ch(N_i)$  and extending bilinearly.  
For $D$ a bounded finite dimensional chain complex of $\rS_{n_1} \times   \dots \times \rS_{n_r}$  representations, we define  $$\ch(D) := \sum_{i} (-1)^i \ch(H_i(D)) = \sum_{i} (-1)^i \ch(D_i).$$
Finally, if $C$ is a chain complex of finite dimensional $(\sqcup_{n} \rS_n)^r$  representations which is bounded for each $(n_1, \dots n_r)$, we define $\ch(C) \in  \widehat{\Lambda^{\otimes r}} = \prod_{(n_1, \dots, n_r)}  \Lambda_{n_1} \otimes \dots \otimes \Lambda_{n_r}$  to be $\sum_{(n_1, \dots, n_r)}  \ch(C_{n_1, \dots, n_r})$.  
\end{defn}

The Frobenius character intertwines operations on chain complexes of $\rS_\bdot$ representations and operations on symmetric functions.  We have that $\ch(M \odot N) = \ch(M)[\ch(N)]$   where the square brackets denote plethysm \cite[A.2.6]{StanleyEnum}.  And $\ch(\Res^{\sqcup} M) = \Delta(\ch(M))$  where $\Delta: \hat \Lambda \to \widehat{\Lambda^{\otimes 2}}$ is the coproduct, characterized by $\Delta(p_i) = 1 \otimes p_i + p_i \otimes 1$ for $i \in \bbN$.  These compatibilities are standard for ordinary $\rS_n$ representations:  for a discussion of their extension to chain complexes see for instance \cite[\S5]{getzler1995operads}.  

\subsection{Character computation}

To translate Propostion \ref{Whitneyiso} into a statement about Frobenius characters, we need the following. 
\begin{prop}[Stanley \cite{stanley1982some}]
Let $\mu: \bbN_{\geq 1} \to \{1,0,-1\}$ be the arithmetic M\"obius function,  then
$$\ch(\cW) = w := \sum_{d \geq 1 } \frac{ \mu(d)}{d} \log( 1 + p_d).$$
\end{prop}

The element $w$ is the plethysistic inverse of $\sum_{n \geq 1} h_n$.    

\begin{defn}
		We write $\del_n$ for the differential operator on   $\frac{\del}{\del p_n}: \hat \Lambda \to \hat \Lambda$, given by expressing $s \in \hat \Lambda$ as a power series in the $p_i$ and differentiating with respect to $p_n$.  We define $$D_n := \sum_{d |n }  \frac{\mu(d)}{d}  \del_{n/d},$$  where $\mu: \bbN \to \{+1, 0 , -1\}$ is the arithmetic M\"obius function.   
 To an integer partition $\lambda = 1^{m_1} 2^{m_2} \dots$ we associate  a differential operator, which is a polynomial in the $D_i$  defined by $${D \choose \lambda} := \prod_{i} {D_i \choose m_i}.$$
\end{defn}

The main theorem of this section describes $\ch(\rK_d(M))$ in terms of $\ch(M)$ and the differential operators  ${D \choose \lambda}$.

\begin{thm}\label{charactercomputation}
Let $M$ be a chain complex of $\FSop$ modules which is bounded and finite dimensional in each degree. Let $d \in \bbN, d \geq 1$.   Then there is an identity in $\widehat{\Lambda \otimes \Lambda } $$$\rc \rh(\rK_d(M)) =  \sum_{\lambda \vdash d}  {D \choose \lambda}  \rc \rh(M) \otimes p_\lambda  .$$   
\end{thm}
\begin{proof}

In terms of symmetric functions Proposition \ref{Whitneyiso} implies that
$$\sum_{d}  \rc\rh(\rK_d(M))  = \Delta(\ch(M)) [ w]_2.$$ 
The brackets $[-]_2$ denote plethysm in the second factor.
 
To calculate this expression, note that since $\Delta(p_i) = p_i \otimes 1 + 1 \otimes p_i$ for any $s \in \hat \Lambda$, expanding $s$ as a power series in the $p_\lambda $ we have  $$\Delta(s) = \sum_{\lambda \text{ integer partition }} \frac{ \del_{\lambda} (s)}{\lambda!} \otimes p_\lambda,$$ where if $\lambda = 1^{m_1} 2^{m_2} \dots$  then $\del_{\lambda} := \prod_{i} \del_i^{m_i}$.  
Therefore we have
$$\Delta(\ch(M)) = \sum_{\lambda \text{ integer partition }} \frac{ \del_{\lambda} (\ch(M))}{\lambda!} \otimes p_\lambda.$$  Since $p_\lambda[w] =\prod_{i}  (w[p_i])^{m_i}$ and $w[p_i] = \sum_{d \geq 1} \frac{\mu(d)}{d} \log(1 + p_{id})$  we  have $$  \sum_{d \geq 1}  \rc\rh(\rK_d(M)) =   \sum_{\lambda = 1^{m_1} 2^{m_2} \dots}  ~ \prod_{i \geq 1} \frac{\del_i^{m_i} (\ch(M))}{m_i !} \otimes  \prod_{i \geq 1} \left(\sum_{d \geq 1} \frac{\mu(d)}{d} \log(1 + p_{id})\right)^{m_i}.$$

 To rearrange this sum, we introduce formal variables $t_i$ which commute with $p_i$ and consider the formal expression $$\sum_{\lambda = 1^{m_1} 2^{m_2} \dots}  ~    \prod_{i \geq 1}  \frac{t_i^{m_i}}{m_i!} \left(\sum_{d \geq 1} \frac{\mu(d)}{d} \log(1 + p_{id})\right)^{m_i}  = \exp\left( \sum_{i}  t_i \left( \sum_d \frac{\mu(d)}{d}  \log(1 + p_{id})\right) \right)  $$  $$ = \exp \left(\sum_{n \geq 1} \log(1 + p_n) \sum_{d|n} \frac{\mu(d)}{d} t_{n/d} \right) = \prod_{n \geq 1} (1 + p_n)^{T_n}$$ where $T_n := \sum_{d|n}\frac{\mu(d)}{d}  t_{n/d}$ and the final expression is interpreted as a power series via the binomial theorem $(1+p_n)^{T_n} = \sum_{k} {T_n \choose k}  p_n^{k}$.  Expanding the product, we see that this expression equals  $$\sum_{\lambda = 1^{m_1} 2^{m_2} \dots } p_\lambda  \prod_{i} {T_i \choose m_i}. $$  Now, we may interpret the formal variables $t_i$  (resp. $p_i$) as the operators  $\widehat{ \Lambda \otimes \Lambda}\to \widehat {\Lambda \otimes \Lambda}$ defined by $a \otimes b \mapsto \del_i(a) \otimes b$ (resp. $a \otimes b \mapsto a \otimes p_i b$), because these  operators  commute with each other.  Applying the operator defined by our formal expression in $p_i,t_i$ to $\ch(M) \otimes 1,$  we obtain an equality $$ \sum_{d}  \rc\rh(\rK_d(M)) = \sum_{\lambda}  {D \choose \lambda}  \ch(M)\otimes p_\lambda.$$  Equating degree $d$ pieces in the second factor, we obtain the desired result.  
\end{proof}


\section{Differential equations and characters of bounded type}\label{sec:DiffEqs}  In this section, motivated by Theorem \ref{charactercomputation}, we first describe the solution space to a system of differential equations involving ${D \choose \lambda}$.   Then we introduce a notion of \emph{type} of $\FSop$ module and rephrase the results so far in  terms of the type.  We also prove Theorem \ref{characterexps}.   

 \subsection{Solution space to a system differential equations}  It is convenient to introduce generators of $\hat\Lambda$ adapted to $D_i$.  

\begin{defn}\label{def:yn}
  For $n \in \bbN$ define $$y_n := \sum_{k \geq 1} \frac{p_{nk}}{k} = \sum_{i }  - \log(1 - x_i^n).$$ 
\end{defn}

The elements $y_n$ satisfy several elementary identities: for example  $y_1 = \log( \sum_{n \geq 0} h_n)$ and $y_n = y_1[p_n]$.    The most important one follows from M\"obius inversion: 
 \begin{prop} \label{kronecker} For all $n,m  \geq 1$ we have  $D_n(y_m) =\delta(n,m),$ where $\delta(n,m)$ is the Kronecker delta.  \end{prop}

\begin{proof}
	We have $$y_m = \sum_{e,~ m|e}  (m/e) p_e. $$  So   $$D_n(y_m) = \sum_{d|n} \frac{\mu(n/d)}{n/d} \del_{d}  \sum_{e, ~ m|e} (m/e) p_e     = \sum_{e, ~ m|e|n} \frac{\mu(n/e)}{(n/e)}   (m/e).$$   Factoring out $m/n$ and applying M\"obius inversion for the poset $(\bbN_{\geq 1}, |)$  , we see that this sum is $1$ if $n = m$ and zero otherwise. 
\end{proof}

 A similar M\"obius inversion to Proposition \ref{kronecker} yields    $$p_n = \sum_{d \geq 1} \frac{\mu(d)}{d}  y_{nd}.$$

\begin{defn}Let $(\hat \Lambda)^{\rm pow}_i$ denote the closure of the span of the degree $i$ monomials in the power sum symmetric functions.  Thus an arbitrary element of $(\hat \Lambda)^{\rm pow}_i$ takes the form  $$\sum_{(m_1, m_2, \dots),  \sum m_j  = i} a_{(m_1, \dots)} \prod_{j} p_j^{m_j}, $$ for some choice of $a_{(m_1, \dots)} \in  \bbQ$.    We also write $(\hat \Lambda)^{\rm pow}_{\leq i} := \bigoplus_{j \leq i} (\hat \Lambda)^{\rm pow}_{j} $.
\end{defn}

Our main theorem of this section describes the space of solutions to a system of differential equations.

\begin{thm}\label{solutionspace}
  	Let $j_1 \geq j_2 \geq \dots \geq j_r \in \bbN$.  Then the space of solutions in $\hat \Lambda$ to the system of linear differential equations  \begin{equation}\label{mainequations}{ D \choose \lambda_1} \dots {D \choose \lambda_r}  s  = 0,    \text{ for all } (\lambda_1, \dots, \lambda_r) \text{ integer partitions such that } |\lambda_t| \geq j_t,\end{equation}   is the subspace  $$\sum_{t = 1}^r  (\hat \Lambda)^{\rm power}_{t-1}   \left\{ \exp(\sum_{i \geq 1} a_i y_i)~\bigg|~ A = 1^{a_1} 2^{a_2} \dots,~ |A| < j_t \right\}.$$
\end{thm}

\begin{ex}
		Let $r = 2$,  $j_1 = 3$ and $j_2 = 2$.  Then Theorem \ref{solutionspace} states that the space of solutions to \eqref{mainequations} is the closure of the span of $$\{1, \exp(y_1), \exp(y_2), \exp(2y_1)\} \cup \{p_k, p_k \exp(y_1)\}_{k \in \bbN}.$$
\end{ex}

The proof of Theorem \ref{solutionspace} is straightforward, but occupies some space.  We defer the the proof until the end of this section, and first explain the consequences of Theorem \ref{solutionspace}.  

\subsection{Spaces of characters of type $< J$}  

To summarize the results so far, we  introduce a notion of type for $\FSop$ modules.   
\begin{defn}\label{def:Type}
		Let $M$ be an $\FSop$ module let $J$ be the integer partition $j_1 \geq j_2 \geq \dots \geq j_r$.  We say that $M$ has \emph{type $< J$}    if  $$\rK_{\ell_r} \circ \dots \circ \rK_{\ell_1}(M)$$ is exact for all choices of $\ell_t \geq j_t$.     
\end{defn}

\begin{remark}
	By Proposition \ref{symmonoidal} and the quasi-isomorphism $\rK_d \simto \rB_d$,   we may equivanlently define  $M$ to have type $< J$ if $\rK_{\ell_1} \circ \dots \circ \rK_{\ell_r}(M)$  is exact for all choices of $\ell_t \geq j_t$.  
\end{remark}


 In terms of the type,  Theorem \ref{FSop} states that if $M$ is a subquotient of an $\FSop$ module generated in degree $d$,  then there exists an $r \geq 0$ such that  $M$ has type $< (d+1)d^{r-1}$.  (Here $(d+1)d^{r-1}$ is the partition $d+1 \geq d \geq  \dots \geq d$,  with $d$ repeated $r-1$ times).    
Theorem  \ref{charactercomputation}  implies that if $M$ has type $< J$,  then $s = \ch(M)$  satisfies the system of differential equations  \eqref{mainequations}.  (Note that the collection of operators $\{{D \choose \lambda}\}_{\lambda \text{ partition}}$ all commute).     Theorem \ref{solutionspace} describes the space of solutions to this system of equations, and thus constrains characters of $\FSop$ modules of type $<J$.  We introduce some terminology in order to restate these results in a convenient way.

\begin{defn}\label{def:VAr}Let $A = 1^{a_1} 2^{a_2} \dots$  be an integer partition and $r \in \bbN, r> 1$.  We define $\cV_{A, r}\subseteq \hat \Lambda$ to be the subspace $$\cV_{A, r} :=  \exp(\sum_{i} a_i y_i) (\hat \Lambda )^{\rm pow}_{< r}.$$    By convention $\cV_{A,0} = 0$.  
Note that $\cV_{A,r}$ is $1$ dimensional for $r = 1$  and infinite dimensional for $r > 1$. 
\end{defn}

Let $J = j_1 \geq \dots \geq j_r$  be an integer partition.   Theorem \ref{charactercomputation} and Theorem \ref{solutionspace}  combine to show:  

\begin{thm}\label{VAreform} If $M$ is an $\FSop$ module of type $< J$ for $J = j_1 \geq \dots  \geq j_r$,  then
$$\ch(M) \in \bigoplus_{A \text{ integer partition } }  \cV_{A, t(A,J)} \subseteq  \hat \Lambda,$$ where $t(A,J)$  is the number of  $i \in \{1, \dots, r\}$  such that $j_i > |A|$.  
\end{thm}
\begin{proof}
	By Theorem \ref{charactercomputation}, if $M$ has type $< J$  then  $$ 0 =\ch(\rK_{\ell_1} \circ \rK_{\ell_2} \dots \circ \rK_{\ell_r}(M)) = \sum_{\lambda_1 \vdash \ell_1,  \dots, \lambda_r \vdash \ell_r} \left(  \prod_{i=1}^r  {D \choose \lambda_i} \ch(M) \right)\otimes p_{\lambda_1} \otimes \dots \otimes p_{\lambda_r},$$ for all $\ell_t \geq j_t$.  Since the terms $p_{\lambda_1} \otimes \dots \otimes p_{\lambda_r}$ are linearly independent,  $\ch(M)$ satisfies the equations \eqref{mainequations}.  Since $\bigoplus_{A \text{ integer partition } }  \cV_{A, t(A,J)}$ is simply an alternate description of the solution space of Theorem \ref{solutionspace}, the result follows.  
\end{proof}

 In particular, we have the following immediate corollary, in the case $J = d$.  

\begin{cor} \label{onehit} If $M$ is an $\FSop$ module and $\rK_d(M)$ is exact,  then $\ch(M)$  is contained in the finite dimensional subspace of $\hat \Lambda$ spanned by $\exp(\sum_{i} a_i y_i)$ for all integer partitions $A = 1^{a_1} 2^{a_2} \dots$  such that $|A| < d$.  
\end{cor}

In fact the statement of  Corollary \ref{onehit} is sharp, in the sense that the characters of projective $\FSop$ modules generated in degree $d$ span the subspace.

In contrast, when $J = j_1 \geq \dots \geq j_r$ for $r > 1$ the space of solutions to (\ref{mainequations}), is no longer finite dimensional.   For instance the space  $\cV_{0,1}$  consists of arbitary linear combinations of the power sum symmetric functions. 

   However,  for any $d > 0$  the image of  $\cV_{A,r}$ under the ring homomorphism  $$\hat \Lambda \to  \widehat{\bbQ[p_1, p_2, \dots, p_d]}$$ setting $p_i = 0$ for $i > d$  is finite dimensional.   In the next section, we combine this fact together with the  constraint that the irreducible representations appearing in a finitely generated $\FSop$ module have bounded rank, to produce finite dimensional spaces of $\FSop$ characters.  

Finally, in the special case $J = d^{s-1} (d+1)^1$,  Theorem \ref{VAreform}  implies Theorem \ref{characterexps},  as follows.

\subsection{From symmetric functions to class functions}\label{sec:translation}
We show how to recover the results of the introduction,  stated in terms of class functions,  from the results of the body of the paper, stated in terms of symmetric functions.

There is a one-to-one correspondence between elements of $\hat \Lambda$ and class functions on $\sqcup_{n} \rS_n$,  under which $\ch(M)$ corresponds to the character of $M$, $\sigma \mapsto \Tr(\sigma, M)$.  If $s \in \hat \Lambda$, then $s$ corresponds to the class function $\sigma_{\lambda} \mapsto \langle p_\lambda, s \rangle$,  where $\sigma_{|\lambda|} \in \rS_{|\lambda|}$ is an element in the conjugacy class of the integer partition $\lambda.$

The following fact allows us to translate into the language of the introduction.  
\begin{prop}\label{translation}
		Let $A = 1^{a_1} 2^{a_2} \dots$ and $\nu = 1^{\ell_1} 2^{\ell_2} \dots$ integer partitions.   Then symmetric function $\frac{p_\nu}{z_{\nu}} \exp(\sum_{i} a_i y_i)$  corresponds to the class function  $$ {\bbX \choose \nu} A^{\bbX - \nu }  =  \prod_{n \geq 1}  {\bbX_n \choose \ell_n } \left(\sum_{d |n}  d  a_d \right)^{\bbX_n - \ell_n}.$$
\end{prop}
\begin{proof}
 Let $\lambda$ be an integer partition.  We wish to determine $\langle p_\lambda, \frac{p_\nu}{z_{\nu}} \exp(\sum_{i} a_i y_i) \rangle,$ in other words the coefficient of $p_\lambda/z_\lambda$ in the power series expansion of  $\frac{p_\nu}{z_{\nu}} \exp(\sum_{i} a_i y_i).$  To do this, it is convenient to introduce the variables $t_n := p_n/n$,  so that if $\lambda = 1^{m_1} 2^{m_2} \dots$ we have $\frac{p_{\lambda}}{z_\lambda} = \prod_i  \frac{t_i^{m_i}}{m_i!},$  and $$\exp(\sum_{i \geq 1} a_i y_i) = \prod_{n \geq 1}  \exp( t_n \sum_{d|n} d a_d ) = \left(\prod_{n \geq 1}  \sum_{m = 0}^{\infty} \frac{(t_n)^{m}}{m!} (\sum_{d|n} d a_d)^m\right).$$
Now, if  $\nu = 1^{\ell_1} 2^{\ell_2} \dots$  we have $$\frac{p_\nu}{z_{\nu}} \exp(\sum_{i} a_i y_i) = \prod_{n \geq 1} \left(   \sum_{m = 0}^{\infty} \frac{t_n^{\ell_n}}{\ell_n !} \frac{(t_n)^{m}}{m!} (\sum_{d|n} d a_d)^m\right). $$
So the coefficient of $\prod_n  \frac{t_n^{m_n}}{m_n!}$ is  $$\prod_{n}   {m_n \choose \ell_n}\left(\sum_{d|n} d a_d \right)^{m_n - \ell_n},$$ as desired.
\end{proof}

\begin{proof}[Proof of Theorem \ref{characterexps}]
Let $M$ be an $\FSop$ module of class $(d,s)$.  Then, by definition  $M$ has type $ < d^{s-1} (d+1)^1$.   Thus by Theorem \ref{VAreform} $$\ch(M) \in \bigoplus_{ A  \vdash d } \cV_{A,1}  \oplus \bigoplus_{A \text{ integer partition, } |A| < d  } \cV_{A, s}.$$ Concretely this means that $\ch(M)$ is an infinite sum of terms proportional to $p_{\nu}\exp(\sum_i a_i  y_i)$, with restrictions on the $\nu$ and $a_i$ appearing.    Using Proposition \ref{translation} to translate,  we obtain Theorem \ref{characterexps}.
\end{proof}

\subsection{Proof of Theorem \ref{solutionspace}} 
\begin{proof}[Proof of Theorem \ref{solutionspace}]
Using the transition matrix between $\{p_i\}$ and $\{y_i\}$, we may write any $s \in \hat \Wedge$ as a power series in the variables $\{y_i\}_{i \geq 1}$:  $$s=\sum_{(d_1,d_2, \dots) \in \bbN^{\oplus \infty}}  c_{d_1,d_2\dots}\prod_{i} y_i^{d_i}.$$
We cut down the space of possible solutions to (\ref{mainequations}) in several steps.

 We first note that if $s$ is a solution to \eqref{mainequations} and  $(d_i) \in \bbN^{\oplus \infty}$ satisfies $\sum_{i \geq j_1}  d_i  \geq r$  then  $c_{d_1, d_2, \dots} = 0$.   To see this apply the operator $D':= \prod_{i} \frac{D_{i}^{d_i}}{d_i!}$  to $s$.  On the one hand, the constant term of $D'(s)$ is $c_{d_1,d_2 \dots}$.  On the other hand by our hypothesis on $(d_i)$ we may choose a monomial $y_{i_1} y_{i_2} \dots y_{i_r}$ dividing $\prod_{i} y_i^{d_i}$ with $i_t \geq j_1$.  Then $D_{i_1} \dots D_{i_r}$ divides $D'$  and taking $\lambda_t = i_t$  we see that (\ref{mainequations}) implies that $D_{i_1} \dots D_{i_r}(s) = 0$ hence $D'(s) = 0$.   

Thus assuming that $s$ satisfies (\ref{mainequations})  we may write  \begin{equation} \label{highdegreeexpansion} s = \sum_{w} w f_w(y_1, \dots, y_{j_1 -1})\end{equation}  where the sum is over monomials $w$ of degree  $<d$ in the variables $y_{j_1}, y_{j_1 + 1} \dots$.   (Throughout this proof, if  $\prod_{i} y_i^{d_i}$ is a monomial  we define its degree to be $\deg(m) := \sum_{i} d_i$).   Now consider the operators ${ D_i \choose j_1}^r$  for $i = 1, \dots, j_1-1$.    Equation (\ref{mainequations}) implies that ${ D_i \choose j_1}^r (s)$ equals zero.  Because $D_i$ only affects the $y_i$ variable, we have that $${D_i \choose j_1}^r f_w = 0, \text{ for all $i = 1, \dots, j_1-1$.}$$  These $j_1-1$ equations have the following consequence.

\begin{lem}\label{quallemma}
 The function $f_w$ is a linear combination of functions of the form $$\exp(\sum_{i = 1}^{j_1 - 1} a_i y_i)\prod_{i = 1}^{j_1 - 1} y_i ^{d_i},$$ where $(a_i) , (d_i) \in \bbN^{j_1 - 1}$,  and  $a_i < j_1$ and $d_i < r$ for all $i$.  
\end{lem}

For each $A = 1^{a_1} 2^{a_2} \dots $  let $\cV_{A,r} \subseteq \hat \Lambda$  be the subspace of functions of the form   \begin{equation} \exp(\sum_{i} a_i y_i) \left(\sum_{m} c_{m} m\right),\end{equation}   where the sum is over monomials $m$ of degree $<r$ and $c_{m} \in\bbQ$.  Then (\ref{highdegreeexpansion}) and Lemma \ref{quallemma}    imply the crude statement that if $s$ is a solution to (\ref{mainequations}) then $$s \in \sum_{A = (a_i) \in \{0, \dots, j_1-1\}^r}  \cV_{A,r}.$$  Note that the linear independence of finitely many exponentials of distinct orders implies that the right hand sum is in fact a direct sum.  And for any $A =1^{a_1} 2^{a_2} \dots$,  the action of $\{D_i\}_{i \in \bbN}$  preserves $\cV_{A,r}$.    This means that the solution space to the system (\ref{mainequations})  is a sum over $(a_i) \in \{0, \dots, j-1\}^r$  to the space of solutions in $\cV_{A,r}$.  We compute these spaces of solutions in the following lemma.

\begin{lem}
\label{keycomputation}
Let  $(a_i) \in \bbN^{\oplus \infty}$.   Let $\sum_{m} c_m m$,  $c_m \in \bbQ$ be a sum over monomials in the $y_i$ of bounded degree.    By convention we set $j_{t} = 0$ if $t> r$.  

 The symmetric function  $s := \exp(\sum a_i y_i) (\sum_{m} c_m m)$  is a solution to  (\ref{mainequations})  if and only if for every monomial $m$ appearing in the sum with $c_m \neq 0$ we have  $\sum_{i} i a_i  < j_{\deg(m) +1}$.   
\end{lem}

Assuming the Lemma, we have that the space of solutions is spanned by functions of the form $ \exp(\sum_i a_i y_i) (\sum_{m} c_m m) $   where the sum is over monomials $m$ in the $y_i$ of degree $t-1$  and $(a_i) \in \bbN^{\oplus \infty}$  satisfies $\sum_i i a_i < j_{t}$.  Since the change of coordinates between $p_i$ and $y_i$  is linear,  this completes the proof.  
\end{proof}

\begin{proof}[Proof of Lemma \ref{quallemma}]
 Set $j := j_1$ and $f := f_w$. 
Note that the space of solutions to the one-variable differential equation ${{\del_u} \choose j}^r  = 0$ in the ring $\bbQ[[u]]$ is spanned by the functions $u^{d} \exp(a u)$  for  $a,d \in \bbN$,  $d < r$ and  $a < j$,  since  ${{\del_u} \choose j}^r$  is proportional to $\del_u^r( \del_u - 1)^r \dots (\del_u - j + 1)^r$. 

To obtain the multivariable statement from the single variable one,  first expand $f(y_1, \dots, y_j)$ as $\sum_{m} m  f_m(y_1)$ where $m$ is a monomial in $y_2, \dots, y_j$.    Then each $f_m(y_1)$ satisfies the single variable equation, and hence is a linear combination of $y_1^{d_1} \exp(a_1 u)$  for $a_1 < j$ and $d_1 < r$.  So we can write $$f = \sum_{m}  m \sum_{a_1,d_1 } c_m(a_1,d_1) y_1^{d_1} \exp(a_1 y_1)  = \sum_{a_1,d_1} y^{d_1} \exp(a_1 y_1)  f_{a_1, d_1}(y_2, \dots, y_{j-1}), $$ where $c_m(a_1,d_1) \in \bbQ$ is some choice of coefficients; the sum is over $a_1$  (resp. $d_1$)  ranging from $0$ to $j-1$   (resp. $r -1$);  and $f_{a_1,d_1}(y_2, \dots y_j) = \sum_{m}  c_m(a_1,d_1) m$.   Applying this same process to the variable $y_2$ in the function $f_{a_1,d_1}$  and continuing on like this, expanding out the resulting sums,  we obtain the result.  
\end{proof}

\begin{proof}[Proof of Lemma \ref{keycomputation}]
First suppose that there is a monomial $m$ in the sum such that $\sum_{i} i a_i \geq j_{\deg(m)+1}$.  We may assume that $m$ has maximal degree, since any monomial of degree $\geq \deg(m)$ will also satisfy the inequality.  Write $m = \prod_{i} y_i^{d_i}$ and $d = \deg(m)$.   

  We will choose a sequence of partitions $\lambda_1, \dots, \lambda_{\max(r,d)}$  such that $|\lambda_{t}| \geq j_t$ and $\prod_{t = 1}^{\max(r, d)} {D \choose \lambda_t}$ does not annihilate $s$.  This suffices for the forward direction, because it shows that $\prod_{t = 1}^{r} {D \choose \lambda_t}$ does not annihilate $s$.       Let $\nu := 1^{a_1} 2^{a_2} \dots$,  and write $\prod_{i} y_i ^{d_i}$  as  $y_{i_1} \dots y_{i_d}$.    For $1 \leq t \leq d$  we choose  $\lambda_t = i_t^{\kappa_t}$   where $\kappa_t \in \bbN $ is sufficiently large that $\kappa_t > a_{i_t}$ and $\kappa_t i_t \geq j_t$.  If  $ d < t \leq r$,  we set  $\lambda_t := \nu$.    In both cases, by construction we have $|\lambda_t| \geq j_t$,  since $|\lambda_t| = \kappa_t i_t \geq j_t$ and $|\lambda_t| = \sum_{i} i a_i \geq j_{d+1} \geq j_t$ respectively.  

That the operator corresponding to this choice of $(\lambda_t)_{t = 1}^{\max(r,d)}$  does not annihilate $s$ follows from two claims:
\begin{enumerate}
\item    ${D \choose \lambda_1} \dots {D \choose \lambda_d} (s)$ is a nonzero scalar times $\exp(\sum_{i} a_i y_i)$.

\item   $\exp(\sum_{i} a_i y_i)$  is an eigenvector for ${D \choose \nu}$ with nonzero eigenvalue.  
\end{enumerate}

To see claim (1), we factor $${D \choose \lambda_1} \dots {D \choose \lambda_d} = {D_{i_1} \choose \kappa_{1}} \dots {D_{i_d} \choose \kappa_{d}} $$  as $D'  \prod_{i} (D_i - a_i)^{d_i}$,  where $D'$  is a nonzero scalar times a product of terms of the form $(D_i - b_i)$  for $b_i \neq a_i$.     Let $n = \prod_{i} y_i^{e_i}$ be a monomial distinct from $m$. Then by maximality of the degree of $m$ there exists an $i_0 \geq 1$ such that $e_{i_0} < d_{i_0}$, so the factor $(D_{i_0}- a_{i_0})^{d_{i_0}}$  annihilates $n \exp( \sum_{i} a_i y_i)$.    Therefore we have that ${D \choose \lambda_1} \dots {D \choose \lambda_d} (s) ={D \choose \lambda_1} \dots {D \choose \lambda_d} (c_m m \exp(\sum_i a_i y_i)).$     Further, we have that $\left(\prod_{i} (D_i - a_i)^{d_i}  m \exp(\sum_{i} a_i y_i) \right)$  is a nonzero scalar times $\exp(\sum_{i} a_i y_i)$.  And because $(D - b_k)\exp(\sum_{i} a_i y_i) = (b_k - a_k) \exp(\sum_{i} a_i y_i)$  we see that  $\exp(\sum_{i} a_i y_i)$ is an eigenvector of $D'$ with nonzero eigenvalue, establishing the first claim.

To see the second claim, notice that ${D \choose \nu}$ is a nonzero scalar times a product of operators of the form $(D_i - b_i)$ for $b_i < a_i$.

For the converse, let $\prod_{i} y_i^{d_i}$ be a monomial of degree $d$ and suppose that $\sum_{i} i a_i < j_{d+1}$.  (Note that this implies $d< r$).  Let $(\lambda_1, \dots, \lambda_{d+1})$ be a sequence of integer partitions  with $|\lambda_t| \geq j_{t}$.  We show that  \begin{equation}\label{ann} \left(\prod_{t = 1}^{d+1} {D \choose \lambda_t} \right) \exp(\sum_{i} a_i y_i) (\prod_{i}y_{i}^{d_i}) = 0.\end{equation} The converse follows from this statement via linearity and the fact that if the product of the first $d$ operators annihilates $s$  then so does the full product.

To see that equation (\ref{ann}) holds, note that for each $\lambda_t = 1^{m_{1,t}} 2^{m_{2,t}} \dots$  there must exist an $i$ such that $m_{i,t} > a_{i}$,  since otherwise we would have $$|\lambda_t|  = \sum_{i} i m_{i,t} \leq \sum_{i} i a_i < j_{d+1} \leq j_t, $$ contradicting the hypothesis on $\lambda_t$.  For each $t$ choose such an $i$, and denote it by $i_t$.  Then $\prod_{t = 1}^{d+1} (D_{i_t} - a_t)$  is a factor of  $\left(\prod_{t = 1}^{d+1} {D \choose \lambda_t} \right)$.  Because $(D_{k} - a_k)  (m \exp(\sum_{i} a_i y_i) = \exp(\sum_{i} a_i y_i) D_k( m)$,  we see that each factor in the product $\prod_{t = 1}^{d+1} (D_{i_t} - a_t)$ reduces  the degree of the coefficent of $\exp(\sum_{i} a_i y_i)$ by one.  Since there are $d+1$ factors and  $\prod_{i} y_i^{d_i}$ has degree $d$  we have that $$\left(\prod_{t = 1}^{d+1} (D_{i_t} - a_t) \right) \exp(\sum_{a_i} y_i) (\prod_{i}y_{i}^{d_i}) =0,$$ finishing the proof of the converse.  
\end{proof}

\section{Characters of bounded rank}\label{sec:characterspaces}

In \S\ref{sec:CharacterofKd} and \S\ref{sec:DiffEqs} we examined the constraints on the characters of $\FSop$ modules of type $< J$  for $J$ an integer partition.   The purpose of this section is to study an additional constraint on the characters of $\FSop$ modules. 

\begin{defn}\label{def:rank}
		Let $M$ be an $\rS_\bdot$ module over $\bbQ$.   Let $k \in \bbN$. We say that $\rank(M) \leq k$ if, in the decomposition of $M_n$ into irreducibles,  every irreducible representation appearing corresponds to a Young diagram with $\leq k$ rows.   We say that an $\FSop$ module has rank $\leq k$ when its underlying $\rS_\bdot$ representation does.   
\end{defn}

\begin{defn}\label{def:Fk}
We write $\cF_{\leq k} \subseteq \hat \Lambda$ for the closure of the span of Schur functions $s_{\lambda}$  such that the Young diagram of $\lambda$ has $\leq k$ rows.   
\end{defn}

Notice that an $\rS_\bdot$ representation $M$ has rank $\leq k$ if and only if $\ch(M) \in \cF_{\leq k}$. The reason that the rank is useful is the following folklore proposition, which appears in \cite[Theorem 4.1]{proudfoot2017configuration}.

\begin{prop}\label{rankprop}
	Let $M$ be an $\FSop$ module over $\bbQ$.  If $M$ is a subquotient of a module generated in degree $\leq d$, then $\rank(M) \leq d$.  
\end{prop}

Therefore, by Theorem \ref{FSop}  and Theorem \ref{VAreform},  if $M$ is a subquotient of an $\FSop$ module generated in degree $d$   there exists an $r \in \bbN$ such that $\ch(M) \in \hat \Lambda$ lies in the subspace 

$$ \cF_{\leq d} \cap \left(\bigoplus_{ A  \vdash d } \cV_{A,1}  \oplus \bigoplus_{A \text{ integer partition, } |A| < d  } \cV_{A,r} \right). $$
In this section, we prove that this intersection is finite dimensional, and characterize it using the Hall inner product.    


\begin{defn}Write $\epsilon_k: \hat \Lambda \to \widehat {\bbQ[p_1, \dots, p_k]}$ for the specialization homomorphism $p_i \mapsto p_i$ for $i \leq k$  and  $p_i \mapsto 0$ for $i > k$.
\end{defn}

We prove the following theorem over the course of this section.  

\begin{thm} \label{intersect}
Let $A_i, i \in I$ be a finite collection of distinct integer partitions, and $r_i \in \bbN, i \in I$ be a finite collection of  natural numbers (not necessarily distinct). Let $k \in \bbN$ and $k \geq \max_{i}(|A_i|)$.  Then we have
$$\cF_{\leq k} \cap \bigoplus_{i \in I} \cV_{A_i, r_i} = \bigoplus_{i \in I} \cF_{\leq k} \cap \cV_{A_i, r_i}.$$
Further,  given any element $s \in \hat \Lambda$, there is a unique element $\pi_{k}(s) \in \cF_{\leq k}$  such that $$\epsilon_k(\pi_k(s)) = \epsilon_k(s).$$  And if $s \in \cV_{A_i, r_i}$, then $\pi_k(s) \in \cV_{A_i, r_i}$.  
\end{thm}

Because the space $\epsilon_k(\cV_{A,r})$ is finite dimensional, Theorem \ref{intersect} implies that $\cF_{\leq k} \cap \cV_{A,r}$ is finite dimensional, since the elements of this intersection are uniquely determined by their image in $\epsilon_k(\cV_{A,r})$.  In particular the elements $\pi_k(p_\nu \exp(\sum_{i} a_i y_k))$  where $p_{\nu}$ ranges over degree $< r$ monomials in $p_1, \dots, p_k$  form a basis of $\cF_{\leq k} \cap \cV_{A,r}$,  because $\epsilon_k(p_\nu \exp(\sum_{i} a_i y_k))$ is a basis for $\epsilon_k(\cV_{A,r})$. 

Under the correspondence between class functions and symmetric functions,  the specialization $\epsilon_k$  corresponds to restricting the class function to elements $\sigma \in \rS_n$  whose cycles all have length $\leq k$.  So Theorem \ref{intersect} provides unique lift from restricted class functions to class functions.  

Assuming Theorem \ref{intersect}  we can prove Theorem \ref{finitedim}.
\begin{proof}[Proof of Theorem \ref{finitedim}]
Let $(d,s) \in \bbN^2$.  If $M$ has type $(d,s)$ then by Proposition \ref{rankprop}  we have that the rank of $M$ is $\leq d$ and $M$ has type $<J$ for $J =  d^{s-1} (d+1)$.  Define $$\cU_{d,s} :=   \bigoplus_{ A  \vdash d } \cF_{\leq d} \cap \cV_{A,1}  \oplus \bigoplus_{A \text{ integer partition, } |A| < d  } \cF_{\leq d} \cap \cV_{A,s} .$$

By Theorem \ref{intersect}  and Theorem \ref{VAreform} we have that $\ch(M) \in \cU_{d,s}$.  Since  $\pi_k(p_\nu \exp(\sum_{i} a_i y_k))$  where $p_{\nu}$ ranges over degree $< r$ monomials in $p_1, \dots, p_k$  form a basis of $\cF_{\leq k} \cap \cV_{A,r}$,  we  see that $\cU_{d,s}$ has dimension $p(d) + { d+ s - 1 \choose s-1} \sum_{i < d} p(i)$  where $p(i)$ is the number of integer partitions of $i$.  
\end{proof}

To prove Theorem \ref{intersect} and describe the projection operator $\pi_k : \hat \Lambda \to \cF_{\leq k}$, we characterize the spaces $\cF_{\leq k}$ and $\cV_{A,r}$  using the Hall inner product. 

\begin{defn}
		If $W$ is a subspace of $\Lambda$,  we write $W^{\perp} \subseteq \hat \Lambda$ for the \emph{perpendicular space} with respect to the Hall inner product,  characterized by the property  $t \in W^{\perp}$  if and only if $\langle w, t \rangle = 0$ for all $w \in W$.  
\end{defn}  

We record the following general fact from linear algebra.  

\begin{prop}\label{representable}
	Let $W \subseteq \Lambda$.  The Hall inner product induces an isomorphism  $W^{\perp} \to (\Lambda/W)^*$,  by $t \mapsto \langle -, t\rangle$.  In particular, there is a unique element of $W^\perp$  corresponding to every linear function $L: \Lambda/W \to \bbQ$  namely  $$\sum_{\lambda \text{ integer partition}  }  L([p_\lambda])~ p_\lambda.$$  where $[p_\lambda]$ denotes the equivalence class of $p_\lambda$ mod $W$.  
\end{prop}

The next two propositions characterize $\cF_{\leq k}$  and $\cV_{A,r}$  as the perpendicular spaces to two different ideals in $\Lambda$.  The first is certainly known, but we do not have a reference  so we include the proof.  

\begin{prop}\label{perpideal}
	 Let $(e_{k+1},e_{k+2} \dots) \Lambda$ be the ideal of $\Lambda$ generated by $e_i$ for $i > k$.  Then $$ \cF_{\leq k} = ((e_{k+1},e_{k+2} \dots) \Lambda)^{ \perp},$$ with respect to the Hall inner product.  
\end{prop}
\begin{proof}
For $\lambda$ an integer partition, we write $\rank(\lambda)$ for the number of rows in the Young diagram of $\lambda$. 
Write $t = \sum_{\lambda} b_\lambda s_\lambda$ for $b_\lambda \in \bbQ$.    Let $w \in \Lambda$ and $i > k$.  By \cite[7.15.4]{StanleyEnum}  $$\langle w e_i , t\rangle = \sum_{\lambda} b_\lambda  \langle w e_i, s_\lambda \rangle = \sum_{\lambda} b_\lambda  \langle w,  s_{\lambda/1^i} \rangle. $$
 If $t \in \cF_{\leq k}$ then $b_\lambda = 0$ for all $\lambda$ of rank $> k$.  By the (transpose of) the Pieri rule \cite[7.15.9]{StanleyEnum}  we have that  $s_{\lambda/1^i}$ is zero if $\lambda$ does not have rank $\geq i$.  Hence $\langle w e_i, t \rangle = 0$, and so $ \cF_{\leq k} \subseteq (e_{k+1}, \dots ) \Lambda^{\rm \perp}$.    

For the converse note that the degree $n$ component of the ideal $(e_{k+1}, \dots ) \Lambda$ is spanned by the set of $e_\lambda$ where $\lambda \vdash n$ is an integer partition with $> k$ columns in its Young diagram. Thus its perpendicular space in $\hat \Lambda_n$  has the same dimension as the number of integer partitions of $n$ with $\leq k$ columns.  Taking a transpose, we see that this is the same as the dimension of the degree $n$ component of $\cF_{\leq k}$.   Thus the containment  $ \cF_{\leq k} \subseteq (e_{k+1}, \dots ) \Lambda^{\rm \perp}$ is an equality.  
\end{proof}

\begin{prop}\label{Uperp}
Let $A = 1^{a_1} 2^{a_2} \dots$  and let $r \in \bbN$,  $r > 1$.   Define $u_n := p_n - \sum_{i|n} i a_i$ for all $n \geq 1$.    Then  Let $(u_1, u_2, \dots  )^{r} \Lambda$ be the $r^{\rm th}$ power of the ideal generated by the $u_n$.  Then $$\cV_{A,r} = \left((u_1, u_2, \dots  )^{r} \Lambda\right)^{\perp},$$ with respect to the Hall inner product.    

\end{prop}
\begin{proof}
 Any element $t \in \hat \Lambda$  can be expanded uniquely as $$t = \sum_{\lambda \text{ integer partition}} c_\lambda p_\lambda \exp(\sum_{i} a_i y_i)$$ for $c_\lambda \in \bbQ$, because the lowest degree term of $p_\lambda \exp(\sum_{i} a_i y_i)$ is $p_\lambda$.   Then $t \in \hat \Lambda$ if and only if $c_\lambda = 0$ for all $\lambda$ whose Young diagrams have $\geq r$ rows.  

	For $\lambda =1^{m_1} 2^{m_2} \dots$  put $u_\lambda := \prod_{i} u_i^{m_i}.$

\begin{lem}\label{Ucomputation}
We have that $$ \left\langle \frac{u_\lambda}{z_\lambda},~ p_\nu  \exp(\sum_{i} a_i y_i)) \right \rangle =  \delta(\lambda,\nu) ,$$ where $\delta(\lambda,\nu)$ is the Kronecker delta.  
\end{lem}
\begin{proof}
    First  notice that for any $\lambda,\nu$  we have  $$\langle \del_i (p_\lambda/z_\lambda) ,   p_\nu \rangle  =  \langle p_{\lambda}/z_\lambda , (p_i/i) p_\nu \rangle,$$  using the fact that $\langle p_\lambda/z_\lambda ,  p_\nu  \rangle = \delta(\lambda, \nu)$.  By linearity, it follows that the operator $\del_i$ and multiplication by $p_i/i$ are adjoint.  Thus for any pair of symmetric functions $f \in \Lambda, g \in \hat \Lambda $  we have that $$\langle f , g \exp(\sum_{d \geq 1} a_d y_d) \rangle = \langle f , g \exp(\sum_{d \geq 1} a_d d \sum_{d|k} p_k/k)  \rangle =  \langle \exp( \sum_d a_d d \sum_{d|k} \del_k)  f, g \rangle.$$  Taylor expanding with respect to $p_k$, we see that for any $k \in \bbN$, the operator $\exp(\del_k)$  acts on $\Lambda$ by the ring homomorphism $p_k \mapsto p_k + 1$ and $p_n \mapsto p_n$ for $d \neq k$.   Thus we have that $\exp(\sum_{d} d a_d \sum_{k} \del_k )$  acts on $\Lambda$ by the substitution $p_{n} \mapsto p_{n} + \sum_{d|n} d a_d$.  In particular we have that $\exp(\sum_{d} d a_d \sum_{k} \del_k ) u_\lambda = p_\lambda$.  So  $$\langle u_\lambda /z_\lambda, p_\nu \exp(\sum_{d} a_d y_d) \rangle = \langle p_\lambda /z_\lambda, p_\nu \rangle = \delta(\lambda,\nu),$$ as desired.  
\end{proof}
Applying Lemma \ref{Ucomputation} to the expansion of $t$,  we see that $t \in \cV_{A,r}$ if and only if $\langle u_{\lambda}/z_{\lambda}, t \rangle = 0$ for all $\lambda = 1^{m_1} 2^{m_2} \dots$ such that $\sum_{i} m_i \geq r$.    Rescaling by $z_\lambda$ we see that $t \in \cV_{A,r}$ if and only if every degree $\geq r$ monomial in the $u_i$ pairs to zero with it.   Taking linear combinations of monomials finishes the proof.
\end{proof}

Since sums of subspaces correspond to intersections of perpendicular spaces,  Propositions \ref{perpideal} and \ref{Uperp} imply.
\begin{prop}\label{ADKperp}
Let $A = 1^{a_1}2^{a_2} \dots$, 
	let $u_n = p_n -  \sum_{i|n} i a_i$, and let $r \in \bbN$.  Then  $$\cV_{A,r} \cap \cF_{\leq k}  =  {((u_1, u_2, \dots)^r + (e_{k+1}, e_{k+2}, \dots) )}^{\perp}.$$
\end{prop}

Next we need

\begin{prop}\label{powbasis}
The ring homomorphism $$\iota_k: \bbQ[p_1, \dots, p_k] \to \Lambda/(e_{k+1}, \dots ),$$  $\iota_k(p_i) := p_i$ is an isomorphism. 
In particular, the images of $\{p_{\lambda}/z_\lambda \}_{\lambda = 1^{m_1} \dots k^{m_k}}$  in $\frac{\Lambda}{(e_{k+1}, \dots )}$ form a basis.   Here $\lambda$ ranges over partitions whose associated Young diagrams have $\leq k$ columns.
\end{prop}
\begin{proof}
We have that $\frac{\Lambda}{(e_{k+1}, \dots )} = \bbQ[e_1, \dots, e_k] \iso \bbQ[x_1, \dots, x_k]^{\rS_k}$ is isomorphic to ring of symmetric functions in $k$ variables. It is well known that the power sums $p_1,\dots, p_k$ form an independent set of generators of this ring.    
\end{proof} 

Now Propositions \ref{representable}, \ref{perpideal} and \ref{powbasis}  allow us to define $\pi_k(s)$.  

\begin{defn}\label{def:pik} Let $s \in \hat \Lambda$. We define $\pi_k(s) \in \cF_{\leq k}$  to be the element corresponding (as in Proposition \ref{representable}) to the  linear function, $\frac{\Lambda}{(e_{k+1}, \dots )} \to \bbQ$ defined by    $p_\lambda/z_\lambda \mapsto  \langle p_\lambda/z_\lambda, s \rangle$ for all integer partitions $\lambda = 1^{m_1} \dots k^{m_k}$, whose Young diagrams have $\leq k$ columns.  
\end{defn}

Because $\langle p_\lambda/z_\lambda, t \rangle$ is precisely the coefficient of $p_\lambda$ in the power-sum expansion of $t$,  it is clear that $s$ and $\pi_k(s)$ have the same image under $\epsilon_k$. And by Propositions \ref{representable}, \ref{perpideal} and \ref{powbasis}, $\pi_k(s)$ is the only element  of $\cF_{\leq k}$  with this property.  

To actually compute $\pi_k(s)$ as a sum of power functions, as in Proposition \ref{representable},  we need to know how to write $p_{\nu}$ as a sum of monomials in $p_1, \dots, p_k$  under the specialization  $\Lambda \to\frac{\Lambda}{(e_{k+1}, \dots )} \iso \bbQ[x_1, \dots, x_k]^{\rS_k}$.


\begin{prop}\label{cyclicideal}
	Let $A = 1^{a_1} 2^{a_2} \dots$ be an integer partition and $r,k \in \bbN$  and let $u_n$ denote $p_n - \sum_{d|n} d a_d$.    Suppose that $k \geq n$.   Let $\overline{ (u_1,u_2, \dots )^r}$ be the image of $(u_1,u_2 \dots )^r$ in $\Lambda/(e_{k+1}, \dots)$.  Then $\iota_k\inv\left(\overline{ (u_1, u_2,\dots )^r}\right) = (u_1, \dots, u_k)^r$.
\end{prop}
\begin{proof}
 	It suffices to show that $\iota_k\inv\left(\overline{ (u_1,u_2 \dots )}\right) = (u_1, \dots, u_k)$,  where $\overline{ (u_1,u_2, \dots )}$ is the image of ${ (u_1,u_2, \dots )}$ in $\Lambda/(e_{k+1}, \dots)$.

Now $\iota_k\inv\left(\overline{ (u_1,u_2 \dots )}\right)$ clearly contains $(u_1, \dots, u_k)$ which is a maximal ideal.  So it suffices to prove that $\overline{ (u_1,u_2 \dots )}$ is a proper ideal of $\Lambda/(e_{k+1}, \dots)$.   To do this, we construct a nontrivial homomorphism $\Lambda/(e_{k+1}, \dots) \to \bbQ$  under which which $u_i$ vanishes.  

Consider the $\bbZ$-set  $S:=\bigsqcup_{d \geq 1} ( \bbZ/d \times [a_d])$.   Linearizing it we obtain an $|A|$ dimensional vector space $\bbQ S$  and an operator $T: \bbQ S \to \bbQ S$  defined to be the action of $+1 \in \bbZ$. 
 Then there is a specialization homomorphism  $\kappa: \Lambda \to \bbQ$ defined by $$p_n \mapsto \Tr(T^n , \bbQ S).$$  Since $\bbQ S$ has dimension $|A| \leq k$  we have that $\wedge^i \bbQ S   = 0$ for all $i > k$.  Thus $\kappa(e_i) = \Tr(T, \wedge^i \bbQ S) = 0$  for all $i > k$.  So $\kappa$ factors through $\Lambda/(e_{k+1}, \dots)$.  Finally, $\kappa(p_n)$ is the number of fixed points of the action of $+n \in \bbZ$ on $S$.   If $d|n$ then $+n$ fixes all of the points of $\bbZ/d$, otherwise it fixes none of them.  So we see that the number of fixed points of $+n$ is $\sum_{d|n} d a_d$.  Therefore $\kappa(u_n) = \kappa(p_n - \sum_{d|n} d a_d) =0$. 

Thus $\kappa$ is the desired homomorphism with kernel $(u_1,u_2 \dots) + (e_{k+1}, \dots)$  and so we are done.  
\end{proof}

\begin{remark}
	The homomorphism $\kappa$  constructed in the proof of Proposition \ref{cyclicideal} can also be directly by  $p_n \mapsto p_n(\zeta_1, \zeta_2, \dots, 0, 0, \dots)$  where the $\zeta_i$ are the roots of unity obtained by diagonalizing $T$.   
\end{remark}

Proposition \ref{cyclicideal} has the following immediate consequence.  
\begin{prop}\label{quotientringcomputation}  Let $A = 1^{a_1} 2^{a_2} \dots$  and let $u_n = p_n - \sum_{i|n} ia_i$.  
There is an isomorphism $$\frac{\bbQ[u_1, \dots, u_k]}{(u_1, \dots, u_k)^r}  \to  \frac{\Lambda}{(u_1, u_2, \dots)^r + (e_{k+1}, e_{k+2}, \dots)}   $$ induced by the map  $u_i \mapsto u_i$  for $i = 1, \dots, k$.
\end{prop}
\begin{proof}
	The morphism is surjective because its image equals the image of $\bbQ[p_1, \dots, p_k]$, and it is injective by Proposition \ref{cyclicideal}.
\end{proof}

\begin{prop}\label{ARpreserved}
	Let $A$ be an integer partition such that $|A|\leq k$.  If $s \in \cV_{A,r}$  then $\pi_k(s) \in  \cV_{A,r}$.  
\end{prop}
\begin{proof}
Let $A = 1^{a_1}2^{a_2} \dots$ and let $u_i = p_i -  \sum_{i} i a_i$.   By Proposition \ref{Uperp}, it suffices to show that $\langle b , \pi_k(s) \rangle = 0$  for every $b \in (u_1,u_2 \dots )^r$.   Consider $\iota_k \inv(b) \in \bbQ[p_1, \dots, p_k]$.  By Proposition \ref{cyclicideal} it lies in the ideal $(u_1, \dots, u_k)^r$.  By Propositions \ref{Uperp} and \ref{representable}, it suffices to show that  $\langle m, \pi_k(s) \rangle = 0$,  where $m$ is any degree $>r$ monomial in $u_1, \dots, u_k$.     Because $m$ is a linear combination of monomials in $p_1, \dots, p_k$,  by the definition of $\pi_k(m)$  we have that $\langle m, \pi_k(s) \rangle = \langle m, s \rangle = 0$ . 
\end{proof}

\begin{proof}[Proof of Theorem \ref{intersect}]
We have already constructed $\pi_k$ and shown that it is unique.  Proposition \ref{ARpreserved}  states that if $\pi_k$ preserves the subspace $\cV_{A,r}$  when $k \geq |A|$.   It remains to show that $\cF_{\leq k} \cap \bigoplus_{i \in I} \cV_{A_i, r_i} = \bigoplus_{i \in I} \cF_{\leq k} \cap \cV_{A_i, r_i}$.    Suppose that we have $s_i \in \cV_{A_i, r_i}$  such that $\sum_{i \in I} s_i \in \cF_{\leq k}$.    The operator $\pi_k$ is projection onto the subspace $\cF_{\leq k}$, so we have $\sum_{i \in I} s_i = \pi_k(\sum_{i \in I} s_i) = \sum_{i} \pi_k(s_i)$.  By linear independence of $\cV_{A_i, r_i}$  and the fact that $\pi_k(s_i) \in \cV_{A_i,r_i}$  it follows that $\pi_k(s_i) = s_i$ for all $i$.  Thus all of the $s_i$ are in $\cF_{\leq k}$.  This completes the proof.  
\end{proof}

Finally we include an alternate description of $\cF_{\leq k}$, that can be used to compute $\pi_k$ quickly some cases.
\begin{prop} \label{homogdeg}
	Let $(\hat \Lambda)^{\rm homog}_{\leq k}$ denote the closure of the span of monomials of degree $\leq k$ in the homogenous symmetric functions.   Then $\cF_{\leq k} = (\hat \Lambda)^{\rm homog}_{\leq k}$.  
\end{prop}
\begin{proof}
		We need to show that the closure of the span $s_\lambda$ where $\lambda$ are integer partitions with $\leq d$ rows equals $(\hat \Lambda)^{\rm homog}_{\leq d}$.  By the Pieri rule \cite[7.15.9]{StanleyEnum},  we have that every monomial $h_i$ of length $\leq d$ is a sum of such $s_\lambda$,  showing one direction of the containment.  To prove equality, observe that in degree $n$ the two subspaces have the same dimension: the number of integer partitions of $n$ with $\leq d$ rows.   
\end{proof}

We conclude this section with an example.  

\begin{ex}
		Consider the space $\cV_{1,2}$, which is the closure of the span of $\exp(y_1)$  and $\{p_n \exp(y_1)\}_{n \geq 1}$.  Then because $$\exp(y_1) = \sum_{n \geq 0} h_n,$$ we have that $\pi_2(\exp(y_1)) = \exp(y_1)$  and $\pi_2(p_1 \exp(y_1)) = p_1 \exp(y_1)$.  We also have that $\pi_2(p_n \exp(y_1)) = 0$ for $n \geq 3$,  because $\epsilon_2(p_n\exp(y_1)) = 0$.    So the interesting case is $\pi_2(p_2\exp(y_1))$.  Observe that we have $$\sum_{n \geq 2} p_n \exp(y_1) = \sum_{n \geq 2} n h_n$$ by Newtons identity. Hence this element lies in $\cF_{\leq 2}$ and its image under $\epsilon_2$ agrees with $p_2 \exp(y_1)$.  So it must be $\pi_2(p_2 \exp(y_1))$. 
\end{ex}

It is also possible to compute $\pi_2(p_2 \exp(y_1))$ directly from Definition \ref{def:pik}.  We describe a systematic method for computing $\pi_k$ in \S \ref{piksubsec}.

\section{Generating functions valued in the dual character space}\label{sec:genfunctiondualring}

We begin with a summary of what we know so far.    First Theorem \ref{intersect}  and Theorem \ref{VAreform} immediately imply the following.

\begin{thm}
Let $J = j_1 \geq  \dots \geq j_r $ be an integer partition, and let $k \in \bbN$ satisfy $k \geq j_1$.
Then if $M$ is an $\FSop$ module of type $< J$  and rank $\leq k$, we have
   
 \begin{equation}\ch(M) \in \bigoplus_{A \text{ integer partition }} \cF_{\leq k} \cap \cV_{A, t(A,J)},  \label{sumofterms}\end{equation}
where $t(A,J) = \#\{ i \in \{1, \dots r\} ~|~ j_i > |A|\}$. 
\end{thm}
 In this section, we study each summand of \eqref{sumofterms} individually.  
Thus, for the remainder of this section, we fix an integer partition $A= 1^{a_1} 2^{a_2} \dots $,  and $r,k \in \bbN$ such that $|A| \leq k$.   We set $u_n := p_n - \sum_{d|n} d a_d$  and $u_\lambda := \prod_{i} u_i^{m_i}$  for $\lambda = 1^{m_1} \dots$ an integer partition.  We will consider $\cV_{A,r} \cap \cF_{\leq k}$  and the projection operator $\pi_k:  \cV_{A,r} \to \cV_{A,r} \cap \cF_{\leq k}$.  

Now,  Propositions \ref{ADKperp} and \ref{quotientringcomputation} show that there are isomorphisms.
\begin{equation}
\cV_{A,r} \cap \cF_{\leq k} \iso \left( \frac{\Lambda}{(u_1, u_2, \dots)^r + (e_{k+1}, e_{k+2}, \dots) } \right)^* \iso \left(\frac{\bbQ[u_1, \dots, u_k]}{(u_1, \dots, u_k)^r } \right)^*, \label{measuringisos}
\end{equation}
which take an element $s \in \cF_{\leq k} \cap \cV_{a,r}$ to the function $\langle -, s\rangle : \Lambda \to \bbQ,$  and to the restriction of this function to degree $<r$  monomials in $u_1, \dots, u_k$.

\subsection{An interpretation} We can interpret the isomorphisms of \eqref{measuringisos} in the following way, which we find clarifying.  Let $\ch(V) \in \hat \Lambda$ be the character of some sequence of $\rS_n$ representations $V$.  Then elements $m \in \Lambda$ correspond to different \emph{measurements} that we can perform on $V$,  by taking the inner product  $\langle m , \ch(V) \rangle$.    For example:

\begin{itemize} \item The element  $m = p_\lambda$ measures the trace of an element of $\rS_{|\lambda|}$  with cycle type $\lambda$,
$$ \langle p_\lambda, \ch(V) \rangle = \Tr(\sigma_\lambda, V_{|\lambda|}).$$

\item The element  $m = h_n$ measures the dimension of the space of $\rS_n$ invariants:    $$\langle h_n, \ch(V) \rangle =  \dim (V_n)^{\rS_n}.$$

\item The element $m = u_\lambda$  measures  the coefficient $c_\lambda$ in the unique expansion $$\ch(V)  = \sum_{\lambda \text{ integer partition}} c_\lambda  \frac{p_\lambda}{z_\lambda} \exp(\sum_{i \geq 1} a_i y_i),$$ by Lemma \ref{Ucomputation}.

\end{itemize}

In these terms, the first isomorphism  states that if $\ch(V) \in \cV_{A,r} \cap \cF_{\leq k}$ then the value of the measurement $m \in \Lambda$  only depends on its image in the quotient  $\frac{\Lambda}{(u_1, u_2, \dots)^r + (e_{k+1}, e_{k+2}, \dots) }$.  (Measurements that were independent for arbitrary characters are correlated for elements of $\cV_{A,r}\cap \cF_{\leq k}$).  And the second isomorphism shows that to determine the value of an arbitrary  measurement $m \in \Lambda$ we can proceed in two steps: (1) compute the measurements which are monomials in $u_1, \dots, u_k$,  (2) express the image of $m$ in the quotient ring as a linear combination of monomials in $u_1, \dots, u_k$.

Step (2) is nontrivial,  since we have to actually compute the isomorphism which Proposition \ref{quotientringcomputation} provides abstractly. One can use Gr\"obner bases to algorithmically reduce $m$ into a normal form, but this does not yield explicit formulas.   The purpose of this section is to introduce a strategy for accomplishing step (2) in the cases $m = u_\lambda$ and $m = h_n$, using generating functions. We use this strategy to construct an explicit basis for $\cV_{A,r} \cap \cF_{\leq k}$,  to describe how to compute $\pi_k$, and to prove Theorem \ref{GrowingRows}.

\subsection{A generating function identity}
 We let $R := \frac{\Lambda}{(u_1, u_2, \dots)^r + (e_{k+1}, e_{k+2}, \dots)}$ be the quotient ring, and write $f: \Lambda \to R$  for the quotient homomorphism. 

Let $\lambda$ be an integer partition.  We wish to express $f(u_\lambda)$ as a linear combination of $f(u_\nu)$  where  $\nu = 1^{\ell_1} \dots k^{\ell_k}$ has $\leq k$ columns  (in other words $u_\nu$ is a monomial in $u_1, \dots, u_k$).  
Our strategy is to first compute this expression for $f(u_n)$ and then use the homomorphism property to compute it for $f(u_\lambda)$.  Accordingly, we compute an expression for the $R$-valued generating function $$\sum_{n \geq 0}  \frac{f(u_n)}{n} t^n \in R[[t]].$$

We will also consider generating functions valued in  $\Lambda$  (i.e. elements of $\Lambda[[t]]$).  There is a specialization homomorphism $f: \Lambda[[t]] \to R[[t]]$, given by applying $f: \Lambda \to R$ coefficient wise.  Now, consider

$$\exp(-\sum_n (u_n t^n) /n) \in \Lambda[[t]].$$
We have 

\begin{align} \exp(-\sum_{n \geq 1} (u_n t^n) /n)  & = \exp(-\sum_{n \geq 1} p_n t^n/n) \exp(\sum_{n \geq 1} \sum_{i|n} i a_i  t^n/n) \nonumber
 = \exp(-\sum_{n \geq 1} p_n t^n/n) \exp(\sum_{e,i \geq 1}  a_i \frac{(t^i)^e}{e})
\\ & =  (1+ \sum_{n \geq 1} (-1)^n e_n t^n) \exp(\sum_{i \geq 1} -a_i \log(1-t^i) ) =  \frac{1+ \sum_{n \geq 1} (-1)^n e_n t^n}{ \prod_{i \geq 1} (1 - t^i)^{a_i}}  \label{bigequation}
\end{align}
Where we have used the well-known identity, 
\begin{equation}(1+ \sum_{n \geq 1} (-1)^n e_n t^n) = \exp(-\sum_{n \geq 1} p_n t^n/n),\label{signexpid}\end{equation}
 which can be proved by writing $p_n = \sum_{j} x_j^n$  and using the power series expansion of $\log$.

Applying the specialization homomorphism $f: \Lambda[[t]] \to R[[t]]$  we obtain the identity:

$$\exp\left(-\sum_n \frac{f(u_n)}{n} t^n\right)   =\frac{1 - f(e_1) t + \dots + (-1)^k f(e_k) t^k} { \prod_{i \geq 1} (1 - t^i)^{a_i}} $$ 
Now, we wish to write express the numerator in terms of the variables $u_1, \dots, u_k \in R$.   Before doing this, we introduce some notation.  

\begin{defn}
	For $m \in \bbZ$, let $c_m$ be the coefficient of $t^m$ in $\prod_{i \geq 1}  (1 - t^{i})^{a_i}.$
Concretely $c_m$ is an alternating sum of products of binomial coefficients.  (We may also interpret it as the trace of an operator constructed as in the proof of Proposition \ref{cyclicideal}).
\end{defn}

\begin{defn}
	For $m \in \bbN$ we define $E_m \in \Lambda$ to be $E_m:=  - (-1)^m e_m + c_m$.   Notice that $E_0 = 0$ and $E_m = -(-1)^m e_m$ for $m > |A|$.    We extend this definition to integer partitions, by defining $E_\lambda := \prod_{i} E_i^{m_i}$ if  $\lambda  = 1^{m_1} 2^{m_2} \dots$.
\end{defn}

By definition, we have that $$ 1 + \sum_{n \geq 1} (-1)^n e_n t^n  = - \sum_{m \geq 0 } E_m t^m +  \prod_{i \geq 1} (1-t^i)^{a_i}.$$ Therefore:

$$\exp(-\sum_{n\geq 1}  u_n t^n/n) = 1 - \frac{\sum_{m \geq 1} E_m t^m}{\prod_{i \geq 1} (1- t^i)^{a_i}}.$$

 \begin{prop}\label{Ecomputation}
	 We have that  $$E_m=  - \sum_{n = 1}^m (-1)^n c_{m-n}  \sum_{\lambda \vdash n} \sgn(\lambda) \frac{u_\lambda}{z_\lambda}.$$  In particular $E_m \in (u_1, u_2, \dots) \Lambda$.   
\end{prop}
\begin{proof}
By \eqref{bigequation} we have: 
 $$\exp(-\sum_{ n \geq 1}  u_n t^n /n) \prod_{i \geq 1} (1- t^i)^{a_i} = 1 - \sum_{n \geq 1} (-1)^n e_n t^n.$$
We have  the well-known identity $$\exp(-\sum_{n \geq 1} v_n t^n / n) = \sum_n (-1)^n \sum_{\lambda \vdash n}   \sgn(\lambda) \frac{v_\lambda}{z_\lambda},$$  for any family of independent variables $v_n$.  (In particular, for $v_n = p_n$ this identity follows from \eqref{signexpid} and the character of the sign representation).  

Thus the coefficient of $t^m$ in $\exp(-\sum u_n t^n / n)  \prod_{i} (1-t^{i})^{a_i}$  is  $$c_m + \sum_{n \geq 1} c_{m-n} (-1)^n  \sum_{\lambda \vdash n} \sgn(\lambda) \frac{u_\lambda}{z_\lambda}.$$  Subtracting $c_m$ and multiplying by $-1$, we obtain the result.
\end{proof}

Since $|A| \leq k$  we have that $E_m = \pm e_m$ for all $m > k$.   Hence $f(E_m) = 0$ for all $m \geq k$,  and specializing we obtain:

$$\exp\left(-\sum_{n}  f(u_n) t^n/n \right) = 1 - \frac{\sum_{m = 1}^k f(E_m) t^m}{\prod_{i \geq 1} (1- t^i)^{a_i}}.$$ Therefore we have that  $$ \sum_{n} \frac{f(u_n)}{n} t^n = - \log\left(  1 - \frac{\sum_{m = 1}^k E_m t^m}{\prod_{i \geq 1} (1- t^i)^{a_i}} \right).$$  Because the numerator $\sum_{m=1}^k f(E_m) t^m$ is nilpotent of order $r$ we obtain a finite Taylor expansion 

\begin{align} \sum_{n} \frac{f(u_n)}{n} t^n  & =  \frac{\sum_{m = 1}^k f(E_m) t^m}{\prod_{i \geq 1} (1- t^i)^{a_i}} +\frac{(\sum_{m = 1}^k f(E_m) t^m)^2}{ 2 \prod_{i \geq 1} (1- t^i)^{2a_i}} + \dots + \frac{(\sum_{m = 1}^k f(E_m) t^m)^{r-1}}{ (r-1) \prod_{i \geq 1} (1- t^i)^{(r-1)a_i}} \nonumber \\
& =  \sum_{\lambda \in {\rm Part}(r,k) } f(E_\lambda)  t^{|\lambda|} \frac{   (\rank(\lambda) -  1)!}{\lambda! \left (\prod _{i} (1 - t^i)^{a_i} \right)^{\rank(\lambda)}},  \label{genfunctionid}
\end{align}
where the sum in the second line is over all integer partitions whose Young diagrams have $\leq k$ columns and $<r$ rows.   In the passage to the second line we have used the multinomial identity $$\left(\sum_{i = 1}^k  x_i \right)^n = \sum_{\lambda = 1^{m_1} \dots k^{m_k}, ~ \sum_i m_i = n}  {n \choose m_1, \dots, m_k}  x_\lambda.$$

Equation \eqref{genfunctionid} is our desired expression for $\sum_{n}  \frac{f(u_n)}{n} t^n$.   The remaining subsections are independent of each other.  In the next two subsections  we apply \eqref{genfunctionid} to construct a basis for $\cV_{A,r} \cap \cF_{\leq k}$,  then describe  how to use \eqref{genfunctionid} to compute the operator $\pi_k$.   After that, we prove Theorem \ref{GrowingRows} using a related identity.  Finally, we discuss the relationship between characters of $\FSop$ and $\FI_r$ modules.

\subsection{A convenient basis for $\cV_{A,r} \cap \cF_{\leq k}$}

\begin{defn}Let ${\rm Part} (r,k)$ be the set of integer partitions whose Young diagrams have $ \leq k $ columns and $< r$ rows.
\end{defn}

From the point of view of character functions, the elements $\pi_k(p_\nu \exp(\sum_{i} a_i y_k))$  for $\nu \in {\rm Part}(r,k)$ are a natural basis for $\cV_{A,r}$.  However, we do not know a simple expression for these elements.  In this subsection we construct a different basis, better adapted to Equation \eqref{genfunctionid}.


\begin{prop}\label{Ebasis}
		   There is an isomorphism $$R \iso \frac{\bbQ[E_1, \dots, E_k]}{(E_1,  \dots, E_k )^r}$$ defined by $E_i \mapsto f(E_i)$ for $i = 1, \dots, k$.  Consequently, the set $\{f(E_\nu)\}_{\nu \in \rP\ra\rr\rt(r,k)}$ is a basis for $R$.
\end{prop}
\begin{proof}
	We already know,  by Proposition \ref{quotientringcomputation}, that $u_i \mapsto f(u_i)$ for $i = 1, \dots, k$ gives an isomorphism $\frac{\bbQ[u_1, \dots, u_k]}{(u_1, \dots, u_k)^r} \to R$.    Since $E_m \in (u_1, \dots)$  we have that $f(s) = 0 $ for all $s \in (E_1, \dots, E_k)^r$  thus the homomorphism is well defined.  

To see that it is surjective, note that by Proposition \ref{Ecomputation} $$E_m \equiv u_m/m +\sum_{1 \leq n < m} c_{m - n} u_n/n  \mod  (u_1,u_2 \dots  )^2 .$$  Using this equation, by upper-triangularity we may write any $f(u_m)$ as a linear combination of $f(E_m)$ modulo $(u_1, \dots, u_k)^2$.  Since any collection of elements $x_1, \dots, x_k \in \bbQ[u_1, \dots, u_k]/(u_1, \dots, u_k)^r$  satisfying  $x_i \equiv u_i \mod (u_1, \dots, u_k)^2$ generate the ring,  it follows that these linear combinations of $f(E_m)$ generate.  Thus the homomorphism is surjective, and comparing dimensions it is also injective.  
\end{proof}

Since by Proposition \ref{Ebasis} the elements $\{f(E_\nu)\}_{\nu \in {\rm Part}(r,k)}$ form a basis of $R$, we may make the following definition.  

\begin{defn}
	We define $\{L_\nu\}_{\nu \in \rP\ra\rr\rt(r,k)}$  to be the basis of $\cV_{A,r} \cap  \cF_{\leq k}$  which is dual to $\{f(E_\nu)\}_{\lambda \in {\rm Part}(r,k)}$.  
\end{defn}

 Thus   $L_\nu$ is the unique element of $\cV_{A,r} \cap  \cF_{\leq k}$ satisfying $\langle E_\lambda,  L_\nu\rangle = \delta(\lambda,\nu)$ for all integer partitions $\lambda, \nu\in {\rm Part}(r,k)$.      We will now an explicit epxression for $L_\nu$.   

\begin{defn}
		Let $B = 1^{b_1} 2^{b_2} \dots$ be an integer partition.  For $n \in \bbZ$  we define $g_{B}(n)$ be the coefficient of $t^n$ in 	$\prod_{i}  (1-t^i)^{-b_i}$.  In particular $g_B(n) = 0$ for $n < 0$.    For $n \geq 0$ we have  $$g_B(n)  =  \sum_{\lambda \vdash n,~ \lambda  = 1^{m_2} 2^{m_2}\dots}  ~\prod_{i \geq 1} \multichoose{ b_i}{m_i }.$$
\end{defn}

The function $g_B(n)$ equals a quasi-polynomial of degree $\rank(B)$, for all $n \geq 0$.  


\begin{defn}
		If $\nu = 1^{\ell_1} 2^{\ell_2} \dots$  and $\nu' = 1^{\ell_1'} 2^{\ell_2'} \dots$  we write $\nu' + \nu$  for the partition $1^{\ell_1 + \ell_1'} 2^{\ell_2 + \ell_2'} \dots$, making the set of partitions into a commutative monoid  (corresponding to multiplication of monomials).  For $n \in \bbN$ write $n \cdot \nu$  for the $n$-fold addition of $\nu$ in this monoid.  
\end{defn}

\begin{defn}
  Define $H_\nu(n):=  n~ g_{\rank(\nu) \cdot A}( n - |\nu|)  \frac{(\rank(\nu) - 1)!}{\nu!}$
\end{defn}

By Equation \eqref{genfunctionid}, we have that $f(u_n) = \sum_{\nu \in \rP\ra\rr\rt(r,k), |\nu|> 0} H_\nu(n) f(E_\nu).$  So for $\lambda = n_1 \geq n_2 \geq \dots  $  we have that 
 $$f(u_\lambda) = \prod_{i} \left(\sum_{\nu \in \rP\ra\rr\rt(r,k)} H_\nu(n_i) f(E_\nu) \right)  = \sum_{\nu \in \rP\ra\rr\rt(r,k)}  f(E_\nu) \sum_{\nu = \sum_{i} \nu_i,  ~|\nu_i| > 0  }  \prod_{i \geq 1} H_{\nu_{i}}(n_i),$$
Pairing with $L_\nu$  we get  a formula for $\langle u_\lambda, L_\nu \rangle$.  By Lemma \ref{Ucomputation}  we obtain: 

\begin{prop}
There is an explicit formula for $L_\nu$.  
$$L_\nu =  \exp(\sum_{i} a_i y_i) \sum_{\lambda = n_1 \geq n_2 \geq \dots }  \frac{p_\lambda}{z_\lambda} \ \sum_{\nu = \sum_{i} \nu_i, |\nu_i| > 0} \prod_{i}  H_{\nu_i}(n_i),$$ where the last sum is over all finite tuples of integer partitions $(\nu_1, \nu_2, \dots, \nu_l)$  with $|\nu_{i}| > 0$  and $\nu = \sum_i \nu_i $.  

\end{prop}

\begin{ex}
		Let $\nu$ be the partition $3 \geq 2$.  Then $L_\nu$ is $ \exp(\sum_{i} a_i y_i)  $ times $$   \sum_{n \geq 1} \frac{p_n}{n} H_{3 \geq 2}(n)  + \sum_{n > m \geq 1} \frac{p_n p_m}{nm}(H_3(n)H_2(m) + H_2(n)H_3(m)) + \sum_{n\geq1} \left(\frac{p_n}{n}\right)^2 H_3(n)H_2(n).$$ 
\end{ex}

\subsection{Computing the operator $\pi_k$}\label{piksubsec}

Let $\lambda$ be an integer partition.  We discuss how to compute $\pi_k(p_\nu \exp(\sum_{i} a_i y_k))$.   

We assume that $p_\nu \exp(\sum_{i} a_i y_k) \in \cV_{A,r}$,  or equivalently that $\rank(\nu)< r$.  If the Young diagram of $\lambda$ has $>k$ columns,  then $\epsilon_k(p_\nu) = 0$  and thus $ \pi_k(p_\nu \exp(\sum_{i} a_i y_k)) = 0.$

Therefore, suppose that $\nu \in {\rm Part}(r,k)$.  By Lemma \ref{Ucomputation} the coefficients $b_{\lambda,\nu}$ in the expansion  $$\pi_k(p_\nu \exp(\sum_{i} a_i y_i))  = \sum_{\lambda}  b_{\lambda, \nu} \frac{p_{\lambda}}{z_\lambda} \exp(\sum_{i} a_i y_i),$$ are given by $\langle  u_\lambda, \pi_k(p_\nu \exp(\sum_{i} a_i y_k)) \rangle$.   Let \begin{equation}f(u_\lambda) = \sum_{\nu \in {\rm Part}(r,k) } \overline b_{\lambda, \nu} f(u_{\nu}) \label{expansioneq} \end{equation}  be the unique expansion $f(u_\lambda)$ in $R$.    By Proposition \ref{ADKperp},  since $\pi_k(p_\nu \exp(\sum_{i} a_i y_i)) \in \cF_{\leq k} \cap \cV_{A,r}$, we have that   $\langle  u_\lambda, \pi_k(p_\nu \exp(\sum_{i} a_i y_k)) \rangle$  only depends on $f(u_\lambda)$.  And by the definition of $\pi_k$, we have that for every $\psi \in \rP\ra\rr\rt(n,k)$  that $$\langle u_\psi,  \pi_k(p_\nu \exp(\sum_{i} a_i y_i)) \rangle = \langle u_\psi,  p_\nu \exp(\sum_{i} a_i y_i) \rangle = \delta(\psi,\nu).$$  Applying this to \eqref{expansioneq} we obtain $$b_{\lambda, \nu} = \overline b_{\lambda, \nu}.$$
The coefficients $\overline b_{\lambda, \nu}$  can be computed from $\{\overline b_{n,\nu}\}_{n \in \bbN}$  by multiplication. And Equation \eqref{genfunctionid}  can be used to compute $\overline b_{n,\nu}$.  However, the equation for $\overline b_{\lambda, \nu}$ one obtains in this way is unwieldy.  So we will just carry out this computation in an example.  

\begin{ex}
  	Suppose that $k = 2,r = 3$ and $a_1 = 1, a_i = 0$ for $i > 1$.   Then $c_0 = 1$, $c_1 = -1$ and $c_i = 0$ for $i > 1$.  So $f(E_1) = - f(u_1)$,  $f(E_2) = f(u_2)/2 - f(u_1)^2/2 + f(u_1)$, and $f(E_i) = 0$ for $i >3$. Then we have$$\frac{\sum_{m = 1}^k f(E_m) t^m}{\prod_{i \geq 1} (1- t^i)^{a_i}} =  -f(u_1) t+ \frac{(f(u_2)+ f(u_1)^2)t^2}{2(1-t)} .
$$
So by \eqref{genfunctionid} we compute that $\sum_{n \geq 1} \frac{f(u_n)}{n} t^n $ is
$$
f(u_1) (-t) + f(u_2)\frac{t^2}{2(1-t)} + f(u_1^2) \left( t^2 + \frac{t^2}{2(1-t)}\right) + f(u_1u_2) \frac{-t^2}{2(1-t)} + f(u_2^2)\frac{t^4}{4(1-t)^2}
$$ 
Thus for $n \geq 4$  we have that $$f(u_n) = \frac{n}{2} f(u_2) + \frac{n}{2}  f(u_1^2) - \frac{n}{2} f(u_1u_2) + \frac{n(n+1)}{4} f(u_2^2).$$  Which computes $\overline b_{n, \nu}$ for all $\nu$ and $n \geq 4$.    To compute $\overline b_{\lambda , \nu}$ for the remaining cases, $\lambda = n \geq m$, simply multiply $f(u_n)f(u_m)$.  
\end{ex}

\subsection{Proof of Theorem \ref{GrowingRows}}

We use a generating function argument to prove Theorem \ref{GrowingRows}.  

\begin{proof}[Proof of Theorem \ref{GrowingRows}]
Fix an integer partition $\lambda$ and consider the generating fuction  
 $$F_\lambda(t) = \sum_{k \geq 0}  s_{(k,\lambda)} ~ t^{k + |\lambda| } \in \Lambda[[t]].$$
The following identity appears in the theory of $\FI$ modules:  \begin{equation}F_\lambda(t) = q_\lambda(t) + \sum_{\nu~|~  \lambda - \nu \in \rV\rS} (-1)^{|\lambda| - |\nu|} s_\nu  t^{|\nu|}  \sum_{m \geq 0} h_m t^m, \label{FItransition}\end{equation}  where the sum is over all $\nu$ such that the Young diagram of $\lambda$ minus the Young diagram of  $\nu$ is a vertical strip, and $q_\lambda(t) \in \Lambda[t]$  has degree $\leq |\lambda|$ (in both $\Lambda$ and $t$).   In \cite[\S 5]{sam2016gl} these polynomials also go under the name $q_\lambda$, although they use the coordinates $t_i := p_i/i$ for the ring of symmetric functions. (They show that this identity corresponds to a resolution in the quotient of $\Mod \FI/\Mod \FI^{\rm tors}$  and interpret $q_\lambda$ in terms of local cohomology in  \cite[\S 7.4]{sam2016gl}).  

Now an integer partition $A$ and $r,k \in \bbN$ with $k \geq |A|$.   We use the notation $R, E_\lambda$ from above, relative to these choices if parameters.  From here on we consider the image of  our generating functions 
under the specialization  $\Lambda[[t]] \to R[[t]].$  We have that $$ s_\nu  t^{|\nu|} \sum_{m \geq 0} h_m t^m = \frac{s_\nu  t^{|\nu|} }{1 + \sum_{m \geq 1} (-1)^m e_m t^m} = \frac{s_\nu  t^{|\nu|}}{\prod_{i}(1 - t^i)^{a_i} - \sum_{m = 1}^k E_m t^m}. $$ factoring out $\prod_{i}(1 - t^i)^{a_i}$ and applying a geometric series expansion we have that this equals   $$\frac{s_\nu  t^{|\nu|}}{\prod_{i}(1 - t^i)^{a_i}}  \sum_{n = 0}^{r-1} \left(\frac{ \sum_{m = 1}^k E_m t^m}{\prod_{i}(1 - t^i)^{a_i}} \right)^n.$$
For any $u \in  \cV_{A,r} \cap \cF_{\leq k}$, applying $\langle - , u \rangle$ to this expression  produces a rational function in $\bbZ[[t]]$.  Thus $  \sum_{m \geq 0} \langle  s_\nu h_m, u \rangle t^{m + |\nu|}$ is rational and therefore so is $  \sum_{m \geq |\nu|} \langle  s_{(m, \nu)}, u \rangle t^{m + |\nu|}$  by (\ref{FItransition}).   The denominator of this function is of degree $ \leq r |A|$ and its roots are roots of unity of order $\leq |A|$.  

 Finally if $M$ is an $\FSop$ module of type $(d,s)$,  then $\ch(M)$ is a finite sum of $u_{A,r,k} \in \cV_{A,r} \cap \cF_{\leq k}$   for $|A| \leq d$  and $r \leq s$.  So the generating function $\sum_{m \geq |\nu|} \langle  s_{(m, \nu)}, \ch(M) \rangle t^{m + |\nu|}$ is rational and its denominator has  degree $\leq sd$, and the roots of the denominator are roots of unity of order $\leq d$.
\end{proof}

\begin{remark}
In practice, the strategy of the proof of Theorem \ref{GrowingRows} can be used to make computations.  For example, fix $A$ an integer partition,  and $d,k \in \bbN$ with $d > |A|$.  Then the proof of Theorem \ref{GrowingRows} yields a formula for $\sum_{n} \langle h_n, L_\lambda \rangle t^n$.  Namely it is  $$\sum_{n} \langle h_n, L_\lambda \rangle t^n = \frac{t^{|\lambda|}  \rank(\lambda)!}{\lambda! \left(\prod_{i} (1-t^i)^{a_i} \right)^{\rank (\lambda) + 1}}.$$ 
\end{remark}

\subsection{Relationship with $\FI_d$ module characters}

Sam--Snowden \cite{sam2018hilbert} characterized the Hilbert series of modules over the category $\FI_d$.  They express these characters as polynomials in the elements $\sigma_n := \sum_{k \geq 0} {n \choose k}  h_n$,  and observe that $\sigma_n$ takes the form $q_n(T_1, T_2, \dots) \exp(y_1)$  where $q_n$ is a polynomial in $T_k := \sum_{n \geq 0}  {n \choose k} \frac{p_n}{n}.$  In fact, $q_n$ is the $n^{\rm th}$ complete ordinary Bell polynomial.  For $\lambda = 1^{m_1} 2^{m_2} \dots$,  define $\sigma^\lambda := \prod_{i} \sigma_i^{m_i}$.  Then if $k \in \bbN$  and $\lambda$ satsifies $\rank(\lambda) \leq k$ we have that $$\sigma^{\lambda} \sigma_0^{k - \rank(\lambda)} \in \cF_{\leq k} \cap \cV_{1^k, |\lambda| + 1}$$ 
Sam--Snowden show that these elements are linearly independent.  Thus as $\lambda$ ranges over partitions with $\leq  k$ rows and $|\lambda| < r$,  we obtain a collection of linearly independent elements in $\cF_{\leq k} \cap \cV_{1^k, r}.$  Since the dimension of this space is the number of partitions with $\leq k$ rows and $< r$ columns, which is strictly larger than the number of $\sigma^{\lambda} \sigma_0^{k - \rank(\lambda)},$ this collection of elements does not span.     

Sam--Snowden compute the character of an $\FI_d$ module to be a sum of $s_\nu \sigma^{\lambda} \sigma_0^{k - \rank(\lambda)}$. We do not know if it is possible to use such elements to construct a basis of $\cF_{\leq k'} \cap \cV_{1^k, r'}$ for some $r',k'$.

\section{Further Questions}

We list some questions related to our work.  Many of these questions stem from the fact that we still do not understand the representation theory of $\FSop$ modules.    

\subsection{Behavior of the type under kernels and cokernels}   For applications, it would be useful to have more sophisticated methods of proving that a specific $\FSop$ module is of class $(d,s)$.  

There are two numerical invariants associated to an $\FI$ module $M$,  the stable degree $\delta(M)$ and the local degree $h^{\rm max}(M)$, introduced and used to great effect in \cite{church2018linear}.   The type of an $\FSop$ module is roughly analogous to these invariants.  Given a morphism of $\FSop$ modules $f: N_1 \to N_2$, is it possible to bound the type of $\ker f$ and $\coker f$  in terms of the type of $N_1$ and $N_2$   (or other invariants)?


\subsection{Grothendieck group of $\FSop$ modules} When do two finitely generated $\FSop$ modules have the same character?  In other words, what is the kernel of the map of Grothendieck groups:  $\rK_0(\FSop) \to \prod_{n}  \rK_0(\rS_n)$?

\subsection{Are $\FSop$ module characters cycle bounded?}  

   We say that a symmetric function  is \emph{cycle bounded} if it is a linear combination of finitely many functions of the form $$\exp(\sum_{i \geq 1} a_i y_i)\sum_{\lambda \text{ integer partition}} c_\lambda p_\lambda,$$  where only finitely many $c_\lambda$ are non-zero and  $(a_1,a_2, \dots) \in \bbN^{\oplus \infty}$.  Is the character of every finitely generated $\FSop$ module cycle bounded?  If not, what is an example of a finitely generated $\FSop$ module that is not cycle bounded?  

 We note that projective $\FSop$ modules are cycle bounded, and cycle bounded characters are closed under sums, induction product and Kronecker product.  And if $t$ is cycle bounded, then $s_\lambda[t]$ is also cycle bounded.    On the other hand,  Sam--Snowden show in \cite{sam2018hilbert} that when $d > 1$ there are $\FI_d$ modules that are not cycle-bounded.  

\subsection{Characters of $\FSop$ modules in characteristic $> 0$} What are the consequences of Theorem \ref{FSop} for fields of positive characteristic?   Is it possible to carry out a similar combinatorial analysis for the characters of $\FSop$ modules in this case?

\subsection{Other combinatorial categories}  The arguments we use to prove Theorem \ref{FSop} are mostly general. The specific combinatorics of $\FSop$ only enters in \S\ref{sec:Languages}.    We are interested in extensions of the method of proof of Theorem \ref{FSop}, to other categories studied by Sam--Snowden \cite{sam2017grobner}.  See for instance Conjecture \ref{VIconjecture}.   For which combinatorial categories does the method of proof of Theorem \ref{FSop} work?

\printbibliography

\end{document}